\documentclass[11pt]{article}
\usepackage{amssymb,amsmath,epsfig}
\usepackage[colorlinks=false,breaklinks=true,linkcolor=blue]{hyperref}
\usepackage{graphics,psfrag,graphicx,color}
\usepackage{float,hhline}
\usepackage[sort]{cite}

\numberwithin{equation}{section}
\newcommand{\ds}{\displaystyle}

\newcommand{\bbeta}{{\boldsymbol\eta}}
\newcommand{\bdelta}{{\boldsymbol\delta}}

\newcommand{\bsi}{{\boldsymbol\sigma}}

\newcommand{\bPi}{{\mbox{\boldmath $\Pi$}}}
\newcommand{\btau}{{\boldsymbol\tau}}
\newcommand{\bzeta}{{\boldsymbol\zeta}}

\newcommand{\bv}{{\mathbf{v}}}
\newcommand{\bw}{{\mathbf{w}}}
\newcommand{\f}{\mathbf{f}}

\newcommand{\bc}{\mathbf{c}}

\newcommand{\bu}{\mathbf{u}}

\newcommand{\bt}{{\mathbf{t}}}
\newcommand{\bn}{{\mathbf{n}}}
\newcommand{\be}{{\mathbf{e}}}

\newcommand{\0}{{\mathbf{0}}}

\def\bE{\mathbf{E}}

\def\bG{\mathbf{G}}

\def\bI{\mathbf{I}}

\def\bW{\mathbf{W}}
\def\bM{\mathbf{M}}

\def\bP{\mathbf{P}}

\def\bRT{\mathbf{RT}}

\def\bx{\mathbf{x}}
\def\bGu{\mathbf{Gu}}
\newcommand{\bL}{\mathbf{L}}
\newcommand\bH{\mathbf{H}}

\newcommand\bbM{\mathbb{M}}

\newcommand\bbH{\mathbb{H}}
\newcommand\bbX{\mathbb{X}}
\newcommand\bbW{\mathbb{W}}

\newcommand\bbL{\mathbb{L}}

\newcommand{\bbRT}{\mathbb{RT}}

\newcommand{\cA}{\mathcal{A}}
\newcommand{\cB}{\mathcal{B}}
\newcommand{\cC}{\mathcal{C}}

\newcommand{\cE}{\mathcal{E}}
\newcommand{\cF}{\mathcal{F}}
\newcommand{\cG}{\mathcal{G}}
\newcommand{\cJ}{\mathcal{J}}
\newcommand{\cN}{\mathcal{N}}
\newcommand{\cM}{\mathcal{M}}
\newcommand{\cT}{\mathcal{T}}

\newcommand{\cD}{\mathcal{D}}
\newcommand{\cO}{\mathcal{O}}

\newcommand{\cR}{\mathcal{R}}
\def\R{\mathrm{R}}

\def\H{\mathrm{H}}
\def\L{\mathrm{L}}
\def\M{\mathrm{M}}
\def\V{\mathrm{V}}
\def\W{\mathrm{W}}

\def\rc{\mathrm{c}}
\def\rd{\mathrm{d}}

\def\rP{\mathrm{P}}

\def\rp{\mathrm{p}}
\def\rq{\mathrm{q}}
\def\rs{\mathrm{s}}
\def\rt{\mathrm{t}}
\def\dc{\mathrm{dc}}

\def\tF{\mathtt{F}}

\def\error{\mathtt{error}}
\def\rate{\mathtt{rate}}
\def\iter{\mathtt{iter}}
\def\tol{\textsf{tol}}
\def\DOF{\mathtt{DOF}}
\def\esssup{\mathrm{ess\,sup}}

\def\bdiv{\mathbf{div}}

\def\tr{\mathrm{tr}\,}

\def\div{\mathrm{div}\,}

\def\coeff{\mathbf{coeff}}
\def\qin{{\quad\hbox{in}\quad}}
\def\qon{{\quad\hbox{on}\quad}}
\def\qan{{\quad\hbox{and}\quad}}

\def\ov{\overline}

\def\wt{\widetilde}
\def\wh{\widehat}
\newtheorem{thm}{Theorem}[section]
\newtheorem{rem}{Remark}[section]
\newtheorem{lem}[thm]{Lemma}
\newtheorem{cor}[thm]{Corollary}

\newenvironment{proof}{\noindent{\it Proof.}}{\hfill$\square$}

\numberwithin{equation}{section}
\numberwithin{figure}{section}
\numberwithin{table}{section}

\allowdisplaybreaks

\title{A Banach space mixed formulation for the unsteady Brinkman--Forchheimer equations}

\author{{\sc Sergio Caucao}\thanks{Department of Mathematics, University of Pittsburgh, Pittsburgh, PA 15260, USA, email: {\tt sac304@pitt.edu}. Supported in part by BECAS CHILE para postdoctorado en el extranjero (convocatoria 2018) and the Department of Mathematics, University of Pittsburgh.}
\quad
{\sc Ivan Yotov}\thanks{Department of Mathematics, University of Pittsburgh, Pittsburgh, PA 15260, USA, email: {\tt yotov@math.pitt.edu}. Supported in part by NSF grant DMS 1818775.}}

\date{\today}

\begin{document}

\maketitle

\begin{abstract}
\noindent We propose and analyze a mixed formulation for the
Brinkman--Forchheimer equations for unsteady flows.  Our approach is
based on the introduction of a pseudostress tensor related to the
velocity gradient, leading to a mixed formulation where the
pseudostress tensor and the velocity are the main unknowns of the
system.  We establish existence and uniqueness of a solution to the
weak formulation in a Banach space setting, employing classical
results on nonlinear monotone operators and a regularization
technique. We then present well-posedness and error analysis for
semidiscrete continuous-in-time and fully discrete finite element
approximations on simplicial grids with spatial discretization based
on the Raviart--Thomas spaces of degree $k$ for the pseudostress
tensor and discontinuous piecewise polynomial elements of degree $k$
for the velocity and backward Euler time discretization.  We provide
several numerical results to confirm the theoretical rates of
convergence and illustrate the performance and flexibility of the
method for a range of model parameters.
\end{abstract}

%\noindent
%{\bf Key words}: Brinkman--Forchheimer equations, stress-velocity formulation, mixed finite element methods, a priori error analysis
%
%\smallskip\noindent
%{\bf Mathematics subject classifications (2000)}: 65N30, 65N12, 65N15, 35Q79, 80A20, 76R05, 76D07

\maketitle

%************************************************************************
%************************************************************************

\section{Introduction}

The flow of fluids through porous media at higher Reynolds numbers has
a wide range of applications, including processes arising in chemical,
petroleum, and environmental engineering. Subsurface applications
include groundwater remediation and oil and gas extraction, where fast
flow may occur in fractured or vuggy aquifers or reservoirs, as well as
near injection and production wells. Darcy's law, which is widely used
to model flow in porous media, becomes unreliable for Reynolds numbers
greater than one. The Forchheimer model
\cite{forchheimer1901wasserbewegung} accounts for faster flows by
including a nonlinear inertial term. It can be obtained mathematically
by averaging the Navier-Stokes equations \cite{rm1992}. There have
been a number of numerical studies for the Forchheimer model, see,
e.g., \cite{GirWhe,pr2012,Rui-Pan-CCFD,Kim-Park,Park-2005}. Another extension to
Darcy's law is the Brinkman model \cite{brinkman1949calculation},
which describes Stokes flow through array of obstacles and can be
applied for flows through highly porous media. An advantage of the
Brinkman model is that it has two parameters related to the fluid
viscosity and the medium permeability, respectively, and varying them
allows for modeling flows ranging from the Stokes to the Darcy
regime. Because of this, the Brinkman equations have been often used
to model coupled Stokes and Darcy flows \cite{Mor-Show,Angot2010,Xie-Xu-Xue,LDQ}.

The Brinkman--Forchheimer model, see, e.g.,
\cite{PayStr,cku2006,dr2014,lsst2015,lst2017}, combines the advantages of the
two models and it has been used for fast flows in highly porous media.
In \cite{PayStr}, the authors prove continuous dependence of solutions of
the Brinkman--Forchheimer equations on the Brinkman and Forchheimer
coefficients in the $\L^2$-norm, which is later extended to the
$\H^1$-norm in \cite{cku2006}. Later on, existence and uniqueness of
weak solutions for a velocity-pressure formulation of the
Brinkman--Forchheimer model by means of a suitable regularization
combined with the Faedo--Galerkin approach, was proposed and analyzed
in \cite{dr2014}.  Moreover, the stability of the weak solution of the
corresponding stationary problem is proved.  More recently, a
perturbed compressible system that approximates the
Brinkman--Forchheimer equations is proposed and analyzed in
\cite{lsst2015}.  The corresponding time discretization of the
perturbed system is obtained by a semi-implicit Euler scheme and
lowest-order Raviart--Thomas element applied for spatial
discretization.  In turn, continuous dependence of the solution on the
Brinkman's and Forchheimer's coefficients as well as the initial data
an external forces is obtained.  Meanwhile, in \cite{lst2017} a
pressure stabilization method for the Brinkman--Forchheimer model is
proposed and analyzed.  The authors propose a time discretization
scheme that can be used with any consistent finite element space
approximation.  Second-order error estimate are also derived.

The goal of the present paper is to develop and analyze a new mixed
formulation of the Brinkman--Forchheimer problem and study its
numerical approximation by a mixed finite element method. Unlike
previous works, we introduce the pseudostress tensor as a new unknown,
which allows us to eliminate the pressure from the system, resulting
in a pseudostress-velocity mixed formulation. There are several
advantages of this new approach, including the direct and accurate
approximation of another unknown of physical interest, the
pseudostress tensor, in the $\bbH(\bdiv)$ space, thus enforcing
conservation of momentum in a physically compatible way. In addition,
our formulation alleviates the difficulty of choosing velocity-pressure
finite element spaces that can handle both the Stokes and Darcy limits
in the Brinkman equation. We can use stable Darcy-type mixed finite
element spaces for the pseudostress-velocity pair, such as the
Raviart--Thomas spaces. The numerical experiments indicate robustness
with respect to the Darcy parameter in both the Stokes and the Darcy
regimes. Furthermore, the pressure and the velocity gradient can be
recovered by a simple post-processing in terms of the pseudostress,
preserving the rates of convergence. As a result, these variables are
more accurately approximated, compared to velocity-pressure
formulations. This is illustrated in one of the numerical examples
with discontinuous spatially varying parameters, where the sharp
velocity gradient across the discontinuities is computed
accurately.

Employing techniques from \cite{s2010,aeny2019,cmo2018,cgo2018-pp}, we
combine the classical monotone operator theory and a suitable
regularization technique in Banach spaces to establish existence and
uniqueness of a solution of the continuous weak formulation. Stability
for the weak solution is obtained by an energy estimate. We then
consider semidiscrete continuous-in-time and fully discrete finite
element approximations. The pseudostress tensor and the velocity are
approximated using the Raviart--Thomas spaces of order $k \ge 0$ and
discontinuous piecewise polynomials of degree $\leq k$, respectively,
and time is discretized employing the backward Euler method.  The
well-posedness analysis of the discretization schemes follows the
framework for the continuous weak formulation, combined with discrete
inf-sup stability in the appropriate Banach spaces. We further perform
error analysis for the semidiscrete and fully discrete schemes,
establishing rates of convergence in space and time of the numerical
solution to the weak solution.

The rest of this work is organized as follows.  The remainder of this
section describes standard notation and functional spaces to be
employed throughout the paper.  In
Section~\ref{sec:continuous-formulation} we introduce the model
problem and derive its mixed variational formulation.  In
Section~\ref{sec:well-posedness-model} we establish the well-posedness
of the weak formulation.  The semidiscrete
continuous-in-time scheme is introduced and analyzed in
Section~\ref{sec:semidiscrete-approximation}.  Error estimates and
rates of convergence are also derived.  
In Section~\ref{sec:fully-discrete-formulation}, the fully discrete approximation is
developed and analyzed employing similar arguments to the semidiscrete formulation.
Finally, in Section~\ref{sec:numerical-results} we report numerical studies of the
accuracy of our mixed finite element method, confirming the
theoretical sub-optimal rates of convergence and suggesting optimal
rates of convergence.  In addition, we present computational
experiments illustrating the behavior of the method for a range of
parameter values, as well as its flexibility to handle spatially
varying parameters.

%************************************************************************

\subsection*{Preliminaries}

Let $\Omega\subset \R^n$, $n\in \{2,3\}$, denote a domain with Lipschitz boundary $\Gamma$. 
For $\rs\geq 0$ and $\rp\in[1,+\infty]$, we denote by $\L^\rp(\Omega)$ and $\W^{\rs,\rp}(\Omega)$ the usual Lebesgue and Sobolev spaces endowed with the norms $\|\cdot\|_{\L^\rp(\Omega)}$ and $\|\cdot\|_{\rs,\rp;\Omega}$, respectively.
Note that $\W^{0,\rp}(\Omega)=\L^\rp(\Omega)$. 
If $\rp = 2$, we write $\H^{\rs}(\Omega)$ in place of $\W^{\rs,2}(\Omega)$, and denote the corresponding Lebesgue and Sobolev norms by $\|\cdot\|_{0,\Omega}$ and $\|\cdot\|_{\rs,\Omega}$, respectively, and the seminorm by $|\cdot|_{\rs,\Omega}$. 
By $\bM$ and $\bbM$ we will denote the corresponding vectorial and tensorial counterparts of a generic scalar functional space $\M$. 
Moreover, given a separable Banach space $\V$ endowed with the norm $\|\cdot\|_{\V}$, we let $\L^{\rp}(0,T;\V)$ be the space of classes of functions $f : (0,T)\to \V$ that are Bochner measurable and such that $\|f\|_{\L^{\rp}(0,T;\V)} < \infty$, with
\begin{equation*}
\|f\|^{\rp}_{\L^{\rp}(0,T;\V)} \,:=\, \int^T_0 \|f(t)\|^{\rp}_{\V} \,dt,\quad
\|f\|_{\L^\infty(0,T;\V)} \,:=\, \mathop{\esssup}\limits_{t\in [0,T]} \|f(t)\|_{\V}.
\end{equation*}
In turn, for any vector field $\bv:=(v_i)_{i=1,n}$, we set the gradient and divergence operators, as
\begin{equation*}
\nabla\bv := \left(\frac{\partial\,v_i}{\partial\,x_j}\right)\qan
\div\bv := \sum^n_{j=1} \frac{\partial\,v_j}{\partial\,x_j}.
\end{equation*}
In addition, for any tensor fields $\btau = (\tau_{ij})_{i,j=1,n}$ and $\bzeta = (\zeta_{ij})_{i,j=1,n}$, we let $\bdiv\btau$ be the divergence operator $\div$ acting along the rows of $\btau$, and define the transpose, the trace, the tensor inner product, and the deviatoric tensor, respectively, as
\begin{equation*}
\btau^\rt := (\tau_{ji})_{i,j=1,n},\quad
\tr(\btau) := \sum^n_{i=1} \tau_{ii},\quad
\btau:\bzeta := \sum^n_{i,j=1} \tau_{ij}\,\zeta_{ij},\qan
\btau^\rd := \btau - \frac{1}{n}\,\tr(\btau)\,\bI,
\end{equation*}
where $\bI$ is the identity tensor in $\R^{n\times n}$. 
For simplicity, in what follows we denote
\begin{equation*}
(v,w)_\Omega \,:=\, \int_{\Omega} v\,w,\quad
(\bv,\bw)_\Omega \,:=\, \int_{\Omega} \bv\cdot\bw,\quad
(\btau,\bzeta)_\Omega \,:=\, \int_{\Omega} \btau:\bzeta.
\end{equation*}
When no confusion arises, $|\cdot|$ will denote the Euclidean norm in $\R^n$ or $\R^{n\times n}$. 
Additionally, we introduce the Hilbert space
\begin{equation*}
\bbH(\bdiv;\Omega) := \Big\{\btau\in\bbL^2(\Omega) :\quad \bdiv\btau\in\bL^ 2(\Omega)\Big\},
\end{equation*}
equipped with the usual norm $\|\btau\|^2_{\bdiv;\Omega} := \|\btau\|^2_{0,\Omega} + \|\bdiv\btau\|^2_{0,\Omega}$.
In addition, in the sequel we will make use of the well-known Young's inequality, for $a, b\geq 0,\, 1/\rp + 1/\rq = 1$, and $\delta >0$,
\begin{equation}\label{eq:Young-inequality}
a\,b \,\leq\, \frac{\delta^{\rp/2}}{\rp}\,a^\rp + \frac{1}{\rq\,\delta^{\rq/2}}\,b^\rq.
\end{equation}
Finally, we end this section by mentioning that, throughout the rest of the paper, we employ $\0$ to denote a generic null vector (or tensor), and use $C$ and $c$, with or without subscripts, bars, tildes or hats, to denote generic constants independent of the discretization parameters, which may take different values at different places.

%************************************************************************
%************************************************************************

\section{The continuous formulation}\label{sec:continuous-formulation}

In this section we introduce the model problem and derive the corresponding weak formulation.

%************************************************************************

\subsection{The model problem}\label{sec:model-problem}

In this work we are interested in approximating the solution of the unsteady Brinkman--Forchheimer equations (see for instance \cite{cku2006,dr2014,lsst2015,lst2017}).
More precisely, given the body force term $\f$ and a suitable initial data
$\bu_0$, the aforementioned system of equations is given by
\begin{equation}\label{eq:Brinkman-Forchheimer-1}
\begin{array}{c}
\ds \frac{\partial\,\bu}{\partial\,t} - \nu\,\Delta\bu + \alpha\,\bu + \tF\,|\bu|^{\rp-2}\bu + \nabla p \,=\, \f,\quad 
\div \bu \,=\, 0 \qin \Omega\times (0,T], \\ [2ex]
\ds \bu \,=\, \0 \qon \Gamma\times (0,T],\quad
\bu(0) \,=\, \bu_0 \qin \Omega,\quad
(p,1)_\Omega \,=\, 0,
\end{array}
\end{equation}
where the unknowns are the velocity field $\bu$ and the scalar pressure $p$.
In addition, the constant $\nu > 0$ is the Brinkman coefficient, $\alpha > 0$ is the Darcy coefficient, $\tF > 0$ is the Forchheimer coefficient and $\rp\in [3,4]$ is a given number.
Next, in order to derive our weak formulation, we first rewrite \eqref{eq:Brinkman-Forchheimer-1} as an equivalent first-order set of equations. To that end we introduce the pseudostress tensor
\begin{equation*}
\bsi := \nu\,\nabla\bu - p\,\bI
\end{equation*}
as a new unknown. Applying the trace operator to the above equation
and utilizing the incompressibility condition $\div\bu = 0$ in
$\Omega\times (0,T]$, we find that \eqref{eq:Brinkman-Forchheimer-1}
  can be rewritten, equivalently, as the set of equations with
  unknowns $\bsi$ and $\bu$, given by
  \begin{equation}\label{eq:Brinkman-Forchheimer-2}
\begin{array}{c}
\ds \frac{1}{\nu}\,\bsi^\rd \,=\, \nabla\bu,\quad 
\frac{\partial\,\bu}{\partial\,t} - \bdiv\bsi + \alpha\,\bu + \tF\,|\bu|^{\rp-2}\bu \,=\, \f,\quad
p \,=\, -\frac{1}{n}\,\tr(\bsi) \qin \Omega\times (0,T], \\ [2ex]
\ds \bu \,=\, \0 \qon \Gamma\times (0,T],\quad
\bu(0) \,=\, \bu_0 \qin \Omega,\quad
(\tr(\bsi),1)_\Omega \,=\, 0.
\end{array}
\end{equation}
Notice that the third equation in \eqref{eq:Brinkman-Forchheimer-2}
has allowed us to eliminate the pressure $p$ from the system and
provides a formula for its approximation through a post-processing
procedure, whereas the last equation takes care of the requirement
that $(p,1)_\Omega = 0$.

%************************************************************************

\subsection{The variational formulation}

In this section we derive our mixed variational formulation for the
system \eqref{eq:Brinkman-Forchheimer-2}.  To that end, we begin by
multiplying the first equation of \eqref{eq:Brinkman-Forchheimer-2} by
a tensor $\btau$, living in a suitable space, say $\bbX$, which will
be described next, integrating by parts, using the identity
$\bsi^\rd:\btau = \bsi^\rd:\btau^\rd$ and the Dirichlet boundary
condition $\bu = \0$ on $\Gamma\times (0,T]$, to obtain
\begin{equation}\label{eq:vf-1}
\frac{1}{\nu}\,(\bsi^\rd,\btau^\rd)_\Omega + (\bu,\bdiv\btau)_\Omega \,=\, 0 \quad \forall\,\btau\in \bbX.
\end{equation}
In turn, the second equation of \eqref{eq:Brinkman-Forchheimer-2} is imposed weakly as follows
\begin{equation}\label{eq:vf-2}
(\partial_t\,\bu,\bv)_\Omega - (\bdiv\bsi,\bv)_\Omega + \alpha\,(\bu,\bv)_\Omega + \tF\,(|\bu|^{\rp-2}\bu,\bv)_\Omega \,=\, (\f,\bv)_\Omega \quad \forall\,\bv\in \bM,
\end{equation}
where $\bM$ is a suitable space which together with $\bbX$ is
described below.  We first note that the first term in the
left-hand side in \eqref{eq:vf-1} is well defined if $\bsi, \btau\in
\bbL^2(\Omega)$.  In turn, if $\bu, \bv\in \bL^\rp(\Omega)$, with
$\rp\in [3,4]$, then the first, third and fourth terms in the
left-hand side in \eqref{eq:vf-2} are clearly well defined, which
forces both $\bdiv\bsi$ and $\bdiv\btau$ to live in
$\bL^{\rq}(\Omega)$, with $1/\rp + 1/\rq = 1$.  According to this,
we introduce the Banach space
\begin{equation*}
\bbH(\bdiv_{\rq};\Omega) := \Big\{ \btau\in \bbL^2(\Omega) :\quad \bdiv\btau\in \bL^\rq(\Omega) \Big\},
\end{equation*}
equipped with the norm
\begin{equation*}
\|\btau\|^2_{\bdiv_{\rq};\Omega} 
\,:=\, \|\btau\|^2_{0,\Omega} + \|\bdiv\btau\|^2_{\bL^{\rq}(\Omega)}.
\end{equation*}
Notice that $\bbH(\bdiv;\Omega)\subset \bbH(\bdiv_{\rq};\Omega)$.
In this way, we deduce that the equations \eqref{eq:vf-1}--\eqref{eq:vf-2} make sense if we choose the spaces
\begin{equation*}
\bbX \,:=\, \bbH(\bdiv_{\rq};\Omega) \qan
\bM \,:=\, \bL^{\rp}(\Omega),
\end{equation*}
with their respective norms: $\|\cdot\|_\bbX := \|\cdot\|_{\bdiv_{\rq};\Omega}$ and $\|\cdot\|_\bM := \|\cdot\|_{\bL^{\rp}(\Omega)}$.

Now, for convenience of the subsequent analysis and similarly as in \cite{cgo2018-pp} (see also \cite{Gatica,cgm2019-pp}) we consider the decomposition:
\begin{equation*}
\bbX \,=\, \bbX_0\oplus \R\,\bI,
\end{equation*}
where
\begin{equation*}
\bbX_0 \,:=\, \Big\{ \btau\in \bbH(\bdiv_{\rq};\Omega) :\quad (\tr(\btau),1)_\Omega = 0 \Big\};
\end{equation*}
that is, $\R\,\bI$ is a topological supplement for $\bbX_0$.
More precisely, each $\btau\in \bbX$ can be decomposed uniquely as:
\begin{equation*}
\btau = \btau_0 + c\,\bI \quad\mbox{with}\quad \btau_0\in \bbX_0 \qan 
c:= \frac{1}{n\,|\Omega|}\,(\tr(\btau),1)_\Omega\in \R.
\end{equation*}
Then, noticing that $\btau^\rd = \btau^\rd_0$ and $\bdiv\btau =
\bdiv\btau_0$ and employing the last equation of
\eqref{eq:Brinkman-Forchheimer-2}, we deduce that both $\bsi$ and
$\btau$ can be considered hereafter in $\bbX_0$.  Hence, the weak form
associated with the Brinkman--Forchheimer equation
\eqref{eq:Brinkman-Forchheimer-2} reads: given $\f\in
  \W^{1,1}(0,T;\bL^2(\Omega))$ and $\bu(0) = \bu_0\in \bM$, for $t\in
  (0,T]$, find $(\bsi(t), \bu(t))\in \L^\infty(0,T;\bbX_0)\times
    \W^{1,\infty}(0,T;\bL^2(\Omega))\cap \L^\infty(0,T;\bM)$, such that
\begin{equation}\label{eq:weak-Brinkman-Forchheimer}
\begin{array}{rlll}
\ds \cA\,\bsi(t) + \cB'\bu(t) & = & \0 & \mbox{in }\, \bbX'_0, \\ [2ex]
\ds \frac{\partial}{\partial\,t}\,\cE\,\bu(t) - \cB\,\bsi(t) + \cC\,\bu(t) & = & \bG(t) & \mbox{in }\, \bM',
\end{array}
\end{equation}
where the operators $\cA : \bbX_0\to \bbX'_0$, $\cB : \bbX_0\to \bM'$, $\cC : \bM\to \bM'$, and the functional $\bG\in \bM'$ are defined as follows:
\begin{equation}\label{eq:operators-A-B-C}
\begin{array}{c}
\ds [\cA(\bsi),\btau] \,:=\, \frac{1}{\nu}\,(\bsi^\rd,\btau^\rd)_\Omega,\quad
[\cB(\btau),\bv] \,:=\, (\bdiv\btau,\bv)_\Omega, \\ [3ex]
\ds [\cC(\bu),\bv] \,:=\, \alpha\,(\bu,\bv)_\Omega + \tF\,(|\bu|^{\rp-2}\bu,\bv)_\Omega,
\end{array}
\end{equation}
and
\begin{equation}\label{eq:functionals-G}
[\bG,\bv] := (\f,\bv)_\Omega.
\end{equation}
In addition, the operator $\cE : \bM\to \bM'$ is given by:
\begin{equation}\label{eq:operators-E}
[\cE(\bu),\bv] \,:=\, (\bu,\bv)_\Omega.
\end{equation}
In all the terms above, $[\cdot,\cdot]$ denotes the duality pairing induced by the corresponding operators.
In addition, we let $\cB': \bM\to \bbX'_0$ be the adjoint of $\cB$, which satisfy $[\cB'(\bv),\btau] = [\cB(\btau),\bv]$ for all $\btau\in \bbX_0$ and $\bv\in \bM$.

%************************************************************************
%************************************************************************

\section{Well-posedness of the model}\label{sec:well-posedness-model}

In this section we establish the solvability of \eqref{eq:weak-Brinkman-Forchheimer}.
To that end we first collect some previous results that will be used in the forthcoming analysis.

%************************************************************************

\subsection{Preliminaries}

Let us now discuss the stability properties of the operators involved.
We begin by observing that the operators $\cA, \cB, \cE$ and the functional $\bG$ are linear.
In turn, from \eqref{eq:operators-A-B-C}, \eqref{eq:functionals-G} and \eqref{eq:operators-E}, employing H\"older's and Cauchy--Schwarz inequalities, there hold
\begin{equation*}%\label{eq:operator-B-bounded}
\big|[\cB(\btau),\bv]\big| 
\,\leq\, \|\btau\|_{\bbX}\,\|\bv\|_{\bM} \quad \forall\,(\btau,\bv)\in \bbX_0\times \bM,
\end{equation*}
\begin{equation*}%\label{eq:functional-G-bounded}
\big|[\bG,\bv]\big| \,\leq\, \|\f\|_{\bM'}\|\bv\|_{\bM} \quad \forall\,\bv\in \bM,
\end{equation*}
and
\begin{equation*}%\label{eq:properties-operator-E2}
\big|[\cE(\bu),\bv]\big| \,\leq\, |\Omega|^{(\rp-2)/\rp}\,\|\bu\|_\bM\,\|\bv\|_\bM,\quad
[\cE(\bv),\bv] \,\geq\, \|\bv\|^2_{0,\Omega} \quad \forall\,\bu, \bv\in \bM,
\end{equation*}
which implies that $\cB, \bG$ are bounded and continuous, and $\cE$ is bounded, continuous and monotone.

Next, we summarize some properties of the operators $\cA$ and $\cC$.
\begin{lem}\label{lem:properties-operators-A-C}
The operators $\cA$ and $\cC$ are bounded, continuous, and monotone.
Moreover, $\cC$ is coercive.
\end{lem}
\begin{proof}
From the definition of $\cA$ and $\cC$ (cf. \eqref{eq:operators-A-B-C}), and employing Cauchy--Schwarz and H\"older's inequalities, for all $\bsi, \btau\in \bbX_0$ and for all $\bu,\bv\in \bM$, we deduce the following bounds:	
\begin{equation}\label{eq:properties-operator-A}
\big|[\cA(\bsi),\btau]\big| \,\leq\, \frac{1}{\nu}\,\|\bsi^\rd\|_{\bbX}\,\|\btau^\rd\|_{\bbX}
\,\leq\, \frac{1}{\nu}\,\|\bsi\|_{\bbX}\,\|\btau\|_{\bbX},\quad
[\cA(\btau),\btau] \,\geq\, \frac{1}{\nu}\,\|\btau^\rd\|^2_{0,\Omega}, 
\end{equation}
\begin{equation}\label{eq:properties-operator-C-1}
\big|[\cC(\bu),\bv]\big| \,\leq\, \Big(\alpha\,|\Omega|^{(\rp-2)/(2\rp)}\,\|\bu\|_{0,\Omega} + \tF\,\|\bu\|^{\rp-1}_{\bM}\Big)\,\|\bv\|_\bM,\quad
[\cC(\bv),\bv] \,\geq\, \alpha\,\|\bv\|^2_{0,\Omega} + \tF\,\|\bv\|^{\rp}_\bM,
\end{equation}
which imply that the operators are bounded and non-negative, and $\cC$ is coercive.
In addition, since $\cA$ is linear, its continuity and monotonicity follows from \eqref{eq:properties-operator-A}.
In turn, from the definition of $\cC$ (cf. \eqref{eq:operators-A-B-C}), it follows that
\begin{equation}\label{eq:properties-operator-C-2a}
\begin{array}{l}
\ds [\cC(\bv_1) - \cC(\bv_2),\bv] \,=\, \alpha\,(\bv_1 - \bv_2,\bv)_{\Omega} + \tF\,(|\bv_1|^{\rp-2}\bv_1 - |\bv_2|^{\rp-2}\bv_2,\bv)_{\Omega} \\ [2ex] 
\ds\quad\,\leq\, \alpha\,\|\bv_1 - \bv_2\|_{0,\Omega}\,\|\bv\|_{0,\Omega} + \tF\,\||\bv_1|^{\rp-2}\bv_1 - |\bv_2|^{\rp-2}\bv_2\|_{\bM'}\,\|\bv\|_{\bM} \quad \forall\,\bv_1, \bv_2, \bv\in \bM.
\end{array}
\end{equation}
Proceeding analogously to \cite[Proposition~5.3]{gm1975}, we deduce from \eqref{eq:properties-operator-C-2a} that there exists $c_{\rp} > 0$, depending only on $|\Omega|$ and $\rp$, such that
\begin{equation}\label{eq:properties-operator-C-2}
\begin{array}{l}
\ds \|\cC(\bv_1) - \cC(\bv_2)\|_{\bM'} 
\,\leq\, \alpha\,|\Omega|^{(\rp-2)/(2\rp)}\,\|\bv_1 - \bv_2\|_{0,\Omega} +  \tF\,c_{\rp}\,(\|\bv_1\|_\bM + \|\bv_2\|_\bM)^{\rp-2}\,\|\bv_1 - \bv_2\|_\bM \\ [2ex]
\ds\quad\,\leq\, \Big( \alpha\,|\Omega|^{(\rp-2)/\rp} +  \tF\,c_{\rp}\,(\|\bv_1\|_\bM + \|\bv_2\|_\bM)^{\rp-2} \Big)\,\|\bv_1 - \bv_2\|_\bM,
\end{array}
\end{equation}
concluding the continuity of $\cC$.
Finally, thanks to \cite[Lemma~5.1]{gm1975}, there exist $C_\rp > 0$, depending only on $|\Omega|$ and $\rp$, such that
\begin{equation*}
(|\bu|^{\rp-2}\bu - |\bv|^{\rp-2}\bv,\bu - \bv)_{\Omega} \,\geq\, C_{\rp}\,\|\bu - \bv\|^{\rp}_{\bM} \quad \forall\, \bu, \bv\in \bM,
\end{equation*}
which, together with the definition \eqref{eq:operators-A-B-C} of $\cC$, yields
\begin{equation}\label{eq:properties-operator-C-3}
[\cC(\bu) - \cC(\bv),\bu - \bv] \,\geq\, \alpha\,\|\bu - \bv\|^2_{0,\Omega} + \tF\,C_{\rp}\,\|\bu - \bv\|^{\rp}_{\bM} \quad \forall\, \bu, \bv\in \bM.
\end{equation}
Therefore, $\cC$ is monotone, which completes the proof.
\end{proof}

Now, we state the inf-sup condition associated with the operator $\cB$.
Since the operator $\cC$ is coercive, this result is not necessary to prove the well-posedness of the problem but it will be useful to obtain the stability bound.
\begin{lem}\label{lem:inf-sup-operator-B}
There exists a constant $\beta > 0$ such that
\begin{equation}\label{eq:inf-sup-operator-B}
\sup_{\0\neq \btau\in \bbX_0} \frac{[\cB(\btau),\bv]}{\|\btau\|_{\bbX}} \,\geq\, \beta\,\|\bv\|_{\bM} \quad \forall\,\bv\in \bM.
\end{equation}
\end{lem}
\begin{proof}
For the case $\rp = 4$ and $\rq = 4/3$ we refer the reader to \cite[Lemma~3.4]{cgo2018-pp}, which proof can be easily extended to the case $\rp\in [3,4]$ and $1/\rp + 1/\rq = 1$.
We omit further details.
\end{proof}

In addition, a key result that we will use to establish the existence of a solution to \eqref{eq:weak-Brinkman-Forchheimer} is the following theorem \cite[Theorem~IV.6.1(b)]{Showalter}.
\begin{thm}\label{thm:well-posed-parabolic-problem}
Let the linear, symmetric and monotone operator $\cN$ be given for the real vector space $E$ to its algebraic dual $E^*$, and let $E'_b$ be the Hilbert space which is the dual of $E$ with the seminorm
\begin{equation*}
|x|_b = \big(\cN\,x(x)\big)^{1/2} \quad x\in E.
\end{equation*}
Let $\cM\subset E\times E'_b$ be a relation with domain $\cD = \Big\{ x\in E \,:\, \cM(x) \neq \emptyset \Big\}$.
	
Assume $\cM$ is monotone and $Rg(\cN + \cM) = E'_b$.
Then, for each $u_0\in \cD$ and for each $f\in \W^{1,1}(0,T;E'_b)$, there is a solution $u$ of
\begin{equation*}
\frac{d}{dt}\big(\cN\,u(t)\big) + \cM\big(u(t)\big) \ni f(t) \quad 0 < t < T,
\end{equation*}
with
\begin{equation*}
\cN\,u\in \W^{1,\infty}(0,T;E'_b),\quad u(t)\in \cD,\quad \mbox{ for all }\, 0\leq t\leq T,\qan \cN\,u(0) = \cN\,u_0.
\end{equation*}
\end{thm}

We end this section recalling, for later use, that there exists a positive constant $C_\rd$, such that
\begin{equation}\label{eq:deviatoric-inequality}
C_\rd\,\|\btau\|^2_{0,\Omega} \,\leq\, \|\btau^\rd\|^2_{0,\Omega} + \|\bdiv\btau\|^2_{\bL^{\rq}(\Omega)} \quad \forall\,\btau\in \bbX_0.
\end{equation}
We remark here that the case $\rq = 4/3$ of \eqref{eq:deviatoric-inequality} is stated in \cite[Lemma~3.2]{cgo2018-pp} which proof can be easily extended to the general case $\rq\in (1,2)$.

%************************************************************************

\subsection{The resolvent system}

In this section we focus on proving the hypotheses of Theorem~\ref{thm:well-posed-parabolic-problem} to prove the well-posedness of \eqref{eq:weak-Brinkman-Forchheimer}.
To that end, we denote by $\bM_2$ the closure of the space $\bM$ with respect to the norm and inner product induced by the operator $\cE$ (cf. \eqref{eq:operators-E}), that is,
\begin{equation*}
\|\bv\|^2_{\bM_2} := (\bv,\bv)_\Omega,\quad
(\bv_1,\bv_2)_{\bM_2} := (\bv_1,\bv_2)_\Omega.
\end{equation*}
Notice that $\bM_2 = \bL^2(\Omega) \supset \bM$. 
To verify the range condition in Theorem~\ref{thm:well-posed-parabolic-problem}, let us consider the resolvent system associated with \eqref{eq:weak-Brinkman-Forchheimer}.
Find $(\bsi,\bu)\in \bbX_0\times \bM$ such that:
\begin{equation}\label{eq:resolvent-system}
\begin{array}{rlll}
\ds \cA\,\bsi + \cB'\bu & = & \0 & \mbox{in }\, \bbX'_0, \\ [2ex]
\ds -\,\cB\,\bsi + (\cE + \cC)\,\bu & = & \wh{\bG} & \mbox{in }\, \bM',
\end{array}
\end{equation}
where $\wh{\bG}\in \bM'_2 = \bL^2(\Omega)\subset \bM'$ is a functional given by $\wh{\bG}(\bv) := (\wh{\f},\bv)_\Omega$ for some $\wh{\f}\in \bM'_2$.
Next, in order to prove the well-posedness of \eqref{eq:resolvent-system}, we introduce an operator that will be used to regularize the problem.
Let $\cR_{\bsi} : \bbX_0\to \bbX'_0$ be defined by
\begin{equation}\label{eq:operator-R-bsi}
[\cR_{\bsi}(\bsi),\btau] \,:=\, \left(|\bdiv\bsi|^{\rq-2}\bdiv\bsi,\bdiv\btau\right)_\Omega \quad \forall\,\bsi, \btau\in \bbX_0.
\end{equation}

\begin{lem}\label{lem:R-sigma-properties}
The operator $\cR_{\bsi}$ is bounded, continuous, and monotone.
\end{lem}
\begin{proof}
From the definition of $\cR_\bsi$ (cf. \eqref{eq:operator-R-bsi}), and employing H\"older's inequality, we deduce the following bounds:
\begin{equation}\label{eq:R-bsi-bounded-coercive}
\big|[\cR_\bsi(\bsi),\btau]\big| \,\leq\, \|\bdiv\bsi\|^{\rq/\rp}_{\bL^{\rq}(\Omega)}\,\|\bdiv\btau\|_{\bL^{\rq}(\Omega)}, \quad
[\cR_\bsi(\btau),\btau] \,\geq\, \|\bdiv\btau\|^{\rq}_{\bL^{\rq}(\Omega)} \quad \forall\,\bsi, \btau\in \bbX_0,
\end{equation}
which imply that $\cR_\bsi$ is bounded and non-negative.
In turn, using again the definition of $\cR_\bsi$, we get
\begin{equation*}
[\cR_\bsi(\btau_1) - \cR_\bsi(\btau_2),\btau] 
\,=\, \left(|\bdiv\btau_1|^{\rq-2}\bdiv\btau_1 - |\bdiv\btau_2|^{\rq-2}\bdiv\btau_2,\bdiv\btau\right)_\Omega \quad \forall\,\btau_1, \btau_2,\btau\in \bbX_0.
\end{equation*}
Employing \cite[Lemma~5.4]{gm1975}, we obtain 
\begin{equation*}
\|\cR_\bsi(\btau_1) - \cR_\bsi(\btau_2)\|_{\bbX'} 
\,\leq\, C\,\|\bdiv\btau_1 - \bdiv\btau_2\|^{\rq/\rp}_{\bL^{\rq}(\Omega)} 
\,\leq\, C\,\|\btau_1 - \btau_2\|^{\rq/\rp}_{\bbX},
\end{equation*}
which implies the continuity of $\cR_\bsi$.
Finally, proceeding analogously to \cite[Proposition~5.2]{gm1975}, there exist a positive constant $C>0$ such that
\begin{equation*}%\label{eq:R-bsi-monotonicity}
\begin{array}{l}
\ds [\cR_\bsi(\btau_1) - \cR_\bsi(\btau_2),\btau_1 - \btau_2] 
\,=\, \left(|\bdiv\btau_1|^{\rq-2}\bdiv\btau_1 - |\bdiv\btau_2|^{\rq-2}\bdiv\btau_2, \bdiv(\btau_1 - \btau_2)\right)_\Omega \\ [2ex]
\ds\quad \geq \, C\,\frac{\|\bdiv(\btau_1 - \btau_2)\|^2_{\bL^{\rq}(\Omega)}}{(\|\bdiv\btau_1\|_{\bL^{\rq}(\Omega)} + \|\bdiv\btau_2\|_{\bL^{\rq}(\Omega)})^{2-\rq}} \quad \forall\, \btau_1, \btau_2\in \bbX_0.
\end{array}
\end{equation*}
Therefore $\cR_\bsi$ is monotone which concludes the proof.
\end{proof}

Now, a solution to \eqref{eq:resolvent-system} is established by taking a limit of the regularized solutions as the regularization parameter goes to zero.
\begin{lem}\label{lem:resolvent-system}
Given $\wh{\bG}\in \bM'_2$, there exists a solution $(\bsi,\bu)\in \bbX_0\times \bM$ of the resolvent system \eqref{eq:resolvent-system}.
\end{lem}
\begin{proof}
We proceed similarly to \cite[Lemma~4.6]{aeny2019} (see also \cite{s2010}).
For $\epsilon>0$, consider a regularization of \eqref{eq:resolvent-system} defined by: Given $\wh{\bG}\in \bM'_2\subset \bM'$, determine $(\bsi_\epsilon, \bu_\epsilon)\in \bbX_0\times \bM$ satisfying
\begin{equation}\label{eq:regularized-problem}
\begin{array}{rlll}
\ds (\epsilon\,\cR_\bsi + \cA)\,\bsi_\epsilon + \cB'\bu_\epsilon & = & \0 & \mbox{in }\, \bbX'_0, \\ [2ex]
\ds -\,\cB\,\bsi_\epsilon + (\cE + \cC)\,\bu_\epsilon & = & \wh{\bG} & \mbox{in }\, \bM'.
\end{array}
\end{equation}
Introduce the operator $\cJ : \bbX_0\times \bM \to (\bbX_0\times \bM)'$ defined as
\begin{equation*}
\cJ\left(\begin{array}{c}
\btau \\ \bv
\end{array}\right)
\,:=\, \left(\begin{array}{cc}
\epsilon\,\cR_\bsi + \cA & \cB' \\ -\,\cB & \cE + \cC
\end{array}\right)
\left(\begin{array}{c}
\btau \\ \bv
\end{array}\right).
\end{equation*}
Note that
\begin{equation*}%\label{eq:operator-cJ}
\left[\cJ\left(\begin{array}{c}
\btau_1 \\ \bv_1
\end{array}\right),
\left(\begin{array}{c}
\btau_2 \\ \bv_2
\end{array}\right)\right]
\,=\, [(\epsilon\,\cR_\bsi + \cA)(\btau_1),\btau_2] + [\cB(\btau_2),\bv_1] - [\cB(\btau_1),\bv_2] + [(\cE + \cC)(\bv_1),\bv_2],
\end{equation*}
and
\begin{equation*}
\begin{array}{l}
\ds \left[\cJ\left(\begin{array}{c}
\btau_1 \\ \bv_1
\end{array}\right) - 
\cJ\left(\begin{array}{c}
\btau_2 \\ \bv_2
\end{array}\right),
\left(\begin{array}{c}
\btau_1 \\ \bv_1
\end{array}\right) - 
\left(\begin{array}{c}
\btau_2 \\ \bv_2
\end{array}\right)\right] \\ [3ex]
\ds\quad\,=\, [(\epsilon\,\cR_\bsi + \cA)(\btau_1) - (\epsilon\,\cR_\bsi + \cA)(\btau_2),\btau_1 - \btau_2] + [(\cE + \cC)(\bv_1) - (\cE + \cC)(\bv_2),\bv_1 - \bv_2].
\end{array}
\end{equation*}
From Lemmas~\ref{lem:properties-operators-A-C} and \ref{lem:R-sigma-properties}, we have that $\cJ$ is a bounded, continuous, and monotone operator. Moreover, using the second bounds in \eqref{eq:properties-operator-A}, \eqref{eq:properties-operator-C-1} and \eqref{eq:R-bsi-bounded-coercive}, and inequality \eqref{eq:deviatoric-inequality}, we also have
\begin{equation}\label{eq:cJ-coercivity}
\begin{array}{l}
\ds \left[\cJ\left(\begin{array}{c}
\btau \\ \bv
\end{array}\right),
\left(\begin{array}{c}
\btau \\ \bv
\end{array}\right)\right]
\,=\, [(\epsilon\,\cR_\bsi + \cA)(\btau),\btau] + [(\cE + \cC)(\bv),\bv] \\ [3ex]
\ds\quad\,\geq\, \frac{1}{\nu}\,\|\btau^\rd\|^2_{0,\Omega} + \epsilon\,\|\bdiv\btau\|^{\rq}_{\bL^{\rq}(\Omega)} + \tF\,\|\bv\|^{\rp}_{\bM} + (1+\alpha)\,\|\bv\|^2_{\bM_2} \\ [2ex]
\ds\quad\,\geq\, C_{\rq}\,\min\Big\{ \frac{1}{\nu}\,\|\btau^\rd\|^{2-\rq}_{0,\Omega}, \epsilon \Big\}\,\|\btau\|^{\rq}_{\bbX} + \tF\,\|\bv\|^{\rp}_{\bM},
\end{array}
\end{equation}
where $C_{\rq} := \min\Big\{ 1, C^{\rq/2}_\rd \Big\}/2$.
It follows from \eqref{eq:cJ-coercivity} that $\cJ$ is coercive in $\bbX_0\times \bM$.
Thus, a straightforward application of the Browder--Minty theorem \cite[Theorem~10.49]{Renardy-Rogers} establishes the existence of a solution $(\bsi_\epsilon, \bu_\epsilon)\in \bbX_0\times \bM$ of \eqref{eq:regularized-problem}.

On the other hand, from the first inequality in \eqref{eq:cJ-coercivity} and \eqref{eq:regularized-problem}, we have
\begin{equation*}%\label{eq:bound-independently-eps-1}
\begin{array}{l}
\|\bsi^\rd_\epsilon\|^2_{0,\Omega} + \epsilon\,\|\bdiv\bsi_\epsilon\|^{\rq}_{\bL^\rq(\Omega)} + \|\bu_\epsilon\|^{\rp}_{\bM} + \|\bu_\epsilon\|^2_{\bM_2} 
\,\leq\, C\,\|\wh{\f}\|_{\bM'_2}\,\|\bu_\epsilon\|_{\bM_2}.
\end{array}
\end{equation*}
Then, employing Young's inequality with $\rp = \rq = 2$ and $\delta = 1$ (cf. \eqref{eq:Young-inequality}), we deduce that
\begin{equation}\label{eq:bound-independently-eps-2}
\begin{array}{l}
\|\bsi^\rd_\epsilon\|^2_{0,\Omega} + \epsilon\,\|\bdiv\bsi_\epsilon\|^{\rq}_{\bL^{\rq}(\Omega)} + \|\bu_\epsilon\|^{\rp}_{\bM} + \|\bu_\epsilon\|^2_{\bM_2} 
\,\leq\, C\,\|\wh{\f}\|^{2}_{\bM'_2},
\end{array}
\end{equation}
which implies that $\|\bsi^\rd_\epsilon\|_{0,\Omega}$ and $\|\bu_\epsilon\|_\bM$ are bounded independently of $\epsilon$.
In turn, recalling that $\bdiv(\bbX_0) = \bM'$, and employing the second equation in \eqref{eq:regularized-problem} and \eqref{eq:bound-independently-eps-2}, we get
\begin{equation*}
\|\bdiv\bsi_\epsilon\|_{\bL^{\rq}(\Omega)} 
\,\leq\, C\,\Big( \|\wh{\f}\|_{\bM'_2} + \|\bu_\epsilon\|^{\rp-1}_\bM + \|\bu_\epsilon\|_{\bM_2} \Big)
\,\leq\, C\,\Big( \|\wh{\f}\|_{\bM'_2} + \|\wh{\f}\|^{2/\rq}_{\bM'_2} \Big),
\end{equation*}
which combined with \eqref{eq:bound-independently-eps-2} and \eqref{eq:deviatoric-inequality} implies that $\|\bsi_\epsilon\|_{\bbX}$ is also bounded independently of $\epsilon$.

Since $\bbX_0$ and $\bM$ are reflexive Banach spaces, as $\epsilon\to 0$ we can extract weakly convergent subsequences $\big\{ \bsi_{\epsilon,n} \big\}^\infty_{n=1}, \big\{ \bu_{\epsilon,n} \big\}^\infty_{n=1}$ and $\big\{ \cC(\bu_{\epsilon,n}) \big\}^\infty_{n=1}$, such that $\bsi_{\epsilon,n} \rightharpoonup \bsi$ in $\bbX_0$, $\bu_{\epsilon,n}\rightharpoonup \bu$ in $\bM$, $\cC(\bu_{\epsilon,n})\rightharpoonup \bzeta$ in $\bM'$, and
\begin{equation*}
\begin{array}{rlll}
\ds \cA\,\bsi + \cB'\bu & = & \0 & \mbox{in }\, \bbX'_0, \\ [2ex]
\ds -\,\cB\,\bsi + \cE\,\bu + \bzeta & = & \wh{\bG} & \mbox{in }\, \bM'.
\end{array}
\end{equation*}
Moreover, from \eqref{eq:regularized-problem} we find that
\begin{equation}\label{eq:lim-sup-inequality}
\begin{array}{l}
\ds \lim\sup_{\epsilon\to 0} \,[\cC(\bu_\epsilon),\bu_\epsilon]
\,=\, \lim\sup_{\epsilon\to 0} \,\Big( -[(\epsilon\,\cR_\bsi + \cA)(\bsi_\epsilon),\bsi_\epsilon] + [(\wh{\bG} - \cE(\bu_\epsilon)),\bu_\epsilon] \Big) \\ [2ex]
\ds\quad \leq\, -[\cA(\bsi),\bsi] + [\wh{\bG},\bu] - [\cE(\bu),\bu] \,=\, [\bzeta,\bu].
\end{array}
\end{equation}
Since $\cC$ is monotone and continuous so then is a {\it type M} operator (cf. \cite[Lemma~II.2.1]{Showalter}), which together with \eqref{eq:lim-sup-inequality} yields $\cC(\bu) = \bzeta$. 
Hence, $\bsi$ and $\bu$ solve \eqref{eq:resolvent-system} concluding the proof.
\end{proof}

We end this section establishing a suitable initial condition result, which is necessary to apply Theorem~\ref{thm:well-posed-parabolic-problem} to our context.
\begin{lem}\label{lem:sol0-X-M}
Assume the initial condition $\bu_0\in \bM\cap \bH$, where
\begin{equation}\label{eq:initial-condition-u0}
\bH \,:=\, \Big\{ \bv\in \bH^1_0(\Omega) :\quad \Delta\bv\in \bL^{2}(\Omega) \qan \div\bv = 0 \qin \Omega \Big\}.
\end{equation}	
Then, there exists $\bsi_0\in \bbX_0$ such that
\begin{equation}\label{eq:sol0-X-M}
\left(\begin{array}{cc}
\cA & \cB' \\ -\cB & \cC
\end{array}\right)
\left(\begin{array}{c}
\bsi_0 \\ \bu_0
\end{array}\right)\in 
\left(\begin{array}{c}
\0 \\ \bM'_2
\end{array}\right).
\end{equation}
\end{lem}
\begin{proof}
Given $\bu_0\in \bM\cap \bH$ and define $\bsi_0 := \nu\,\nabla\bu_0$, it follows that
\begin{equation}\label{eq:sigma0-properties}
\frac{1}{\nu}\,\bsi^\rd_0 \,=\, \nabla\bu_0,\quad
\bdiv\bsi_0 \,=\, \nu\,\Delta\bu_0, \qan \tr(\bsi_0) \,=\, 0 \qin \Omega.
\end{equation}
Notice that $\bsi_0\in \bbH_0(\bdiv;\Omega)\subset \bbX_0$, with $\bbH_0(\bdiv;\Omega) := \bbH(\bdiv;\Omega)\cap \bbX_0$.
Next, integrating by parts the first equation in \eqref{eq:sigma0-properties} and proceeding similarly to \eqref{eq:vf-1}, we obtain
\begin{equation*}%\label{eq:first-row-initial-condition}
[\cA(\bsi_0),\btau] + [\cB(\btau),\bu_0] \,=\, 0 \quad \forall\,\btau\in \bbX_0.
\end{equation*}
Hence, given $\bu_0\in \bM\cap\bH$ (cf. \eqref{eq:initial-condition-u0}), taking $\bsi_0\in \bbX_0$ satisfying \eqref{eq:sigma0-properties}, and after minor algebraic manipulation we deduce that 
\begin{equation}\label{eq:initial-condition-problem}
\left(\begin{array}{cc}
\cA & \cB' \\ -\cB & \cC
\end{array}\right)
\left(\begin{array}{c}
\bsi_0 \\ \bu_0
\end{array}\right)
\,=\, \left(\begin{array}{c}
\0 \\ \bG_0
\end{array}\right),
\end{equation}
where
\begin{equation*}
[\bG_0,\bv] \,:=\, -\nu\,(\Delta\bu_0,\bv)_{\Omega} + [\cC(\bu_0),\bv],
\end{equation*}
which together with the additional regularity of $\bu_0$, and the continuous injection of $\bH^1(\Omega)$ into $\bL^{2(\rp-1)}$, with $\rp\in [3,4]$, implies that 
\begin{equation*}
\begin{array}{l}
\ds\big| [\bG_0,\bv] \big| \,\leq\, \Big( \nu\,\|\Delta\bu_0\|_{0,\Omega} + \alpha\,\|\bu_0\|_{0,\Omega} + \tF\,\|\bu_0\|^{\rp-1}_{\bL^{2(\rp-1)}(\Omega)} \Big)\,\|\bv\|_{\bM_2} \\ [2ex]
\ds\quad\,\leq\, C\,\Big( \|\Delta\bu_0\|_{0,\Omega} + \|\bu_0\|_{0,\Omega} + \|\bu_0\|^{\rp-1}_{1,\Omega} \Big)\,\|\bv\|_{\bM_2}.
\end{array}
\end{equation*}
Thus, $\bG_0\in \bM'_2$ so then \eqref{eq:sol0-X-M} holds, completing the proof.
\end{proof}

\begin{rem}
The assumption on the initial condition $\bu_0$ in
\eqref{eq:initial-condition-u0} is not necessary for all the results
that follow but we shall assume it from
now on for simplicity.  A similar assumption to $\bu_0$ is
also made in \cite[eq. (2.2)]{dr2014}.  Note also that
$(\bsi_0,\bu_0)$ satisfying \eqref{eq:sol0-X-M} is not unique.
\end{rem}

%************************************************************************

\subsection{The main result}

Now, we establish the main result of this section.
We begin by recalling the definition of the operators $\cA, \cB, \cC$, and $\cE$ (cf. \eqref{eq:operators-A-B-C}, \eqref{eq:operators-E}), and defining the operators
\begin{equation*}
\cN \,:=\, \left(\begin{array}{cc}
\0 & \0 \\ \0 & \cE
\end{array}\right),\quad
\cM \,:=\, \left(\begin{array}{cc}
\cA & \cB' \\ -\cB & \cC
\end{array}\right),
\end{equation*}
and the spaces,
\begin{equation*}
\bE'_b \,:=\, \Big\{ (\0,\bG)\in \bbX'_0\times \bM' :\quad \bG\in \bM'_2 \Big\},\quad
\cD \,:=\, \Big\{ (\bsi,\bu)\in \bbX_0\times \bM :\quad \cM\left(\bsi, \bu\right)\in \bE'_b \Big\},
\end{equation*}
where $\bM_2 = \bL^2(\Omega)$ and $\bM'_2 = \bL^2(\Omega)$.
Then, we can state the well-posedness of problem \eqref{eq:weak-Brinkman-Forchheimer}.
\begin{thm}\label{thm:well-posed-result}
  For each $\f\in \W^{1,1}(0,T;\bL^2(\Omega))$ and $\bu_0\in \bM\cap \bH$ (cf. \eqref{eq:initial-condition-u0}), there exists a unique solution to \eqref{eq:weak-Brinkman-Forchheimer} with
  $(\bsi, \bu)\in \L^\infty(0,T;\bbX_0)\times \W^{1,\infty}(0,T;\bL^2(\Omega))\cap \L^\infty(0,T;\bM)$ and $\bu(0) = \bu_0$.
\end{thm}
\begin{proof}
We begin by noting that $\cN$ is linear, symmetric and monotone.
In addition, since both $\cA$ and $\cC$ are monotone, then $\cM$ is monotone.
On the other hand, from Lemma~\ref{lem:resolvent-system} we know that for some $(\0,\wh{\bG})\in \bE'_b$ there is a $(\bsi,\bu)\in \bbX_0\times \bM$ such that $(\0,\wh{\bG}) = (\cN + \cM)(\bsi,\bu)$ which implies $Rg(\cN + \cM) = \bE'_b$.
Finally, considering $\bu_0\in \bM\cap \bH$ (cf. \eqref{eq:initial-condition-u0}), from a straightforward application of Lemma~\ref{lem:sol0-X-M} we are able to find $\bsi_0\in \bbX_0$ such that $(\bsi_0, \bu_0)\in \cD$.
Therefore, applying Theorem~\ref{thm:well-posed-parabolic-problem} to our context, we conclude the existence of a solution $(\bsi,\bu)$ to \eqref{eq:weak-Brinkman-Forchheimer}, with
$\bu\in \W^{1,\infty}(0,T;\bL^2(\Omega))$ and $\bu(0) = \bu_0$. 
In turn, employing \eqref{eq:weak-Brinkman-Forchheimer} with $(\btau, \bv) = (\bsi,\bu)$, we deduce that $\bsi\in \L^\infty(0,T;\bbX_0)$ and $\bu\in \L^\infty(0,T;\bM)$.

Now, assume that the solution of \eqref{eq:weak-Brinkman-Forchheimer} is not unique.
Let $(\bsi^i,\bu^i)$, with $i\in \{1,2\}$, be two solutions corresponding to the same data.
Then, taking \eqref{eq:weak-Brinkman-Forchheimer} with $(\btau,\bv) = (\bsi^1 - \bsi^2, \bu^1 - \bu^2)\in \bbX_0\times \bM$, we deduce that
\begin{equation*}
\frac{1}{2}\,\partial_t\,\|\bu^1 - \bu^2\|^2_{0,\Omega} + [\cA(\bsi^1 - \bsi^2),\bsi^1 - \bsi^2] + [\cC(\bu^1) - \cC(\bu^2),\bu^1 - \bu^2] \,=\, 0,
\end{equation*}
which together with the monotonicity bounds \eqref{eq:properties-operator-A} and \eqref{eq:properties-operator-C-3}, implies
\begin{equation*}
\frac{1}{2}\,\partial_t\,\|\bu^1 - \bu^2\|^2_{0,\Omega} + \frac{1}{\nu}\,\|(\bsi^1 - \bsi^2)^\rd\|^2_{0,\Omega} + \tF\,C_{\rp}\,\|\bu^1 - \bu^2\|^{\rp}_{\bM} \,\leq\, 0.
\end{equation*}
Integrating in time from $0$ to $t\in (0,T]$, and using $\bu^1(0) = \bu^2(0)$, we obtain
\begin{equation}\label{eq:uniquiness-inequality-1}
\|\bu^1(t) - \bu^2(t)\|^2_{0,\Omega} + C\,\int^t_0 \left(\|(\bsi^1 - \bsi^2)^\rd\|^2_{0,\Omega} + \|\bu^1 - \bu^2\|^{\rp}_{\bM}\right)\,ds \,\leq\, 0.
\end{equation}
Therefore, it follows from \eqref{eq:uniquiness-inequality-1} that $(\bsi^1(t))^\rd = (\bsi^2(t))^\rd$ and $\bu^1(t) = \bu^2(t)$ for all $t\in (0,T]$.
Next, from the second row of \eqref{eq:weak-Brinkman-Forchheimer}, we get 
\begin{equation*}
[\cB(\bsi^1 - \bsi^2),\bv] \,=\, (\partial_t\,(\bu^1 - \bu^2),\bv)_{\Omega} + [\cC(\bu^1) - \cC(\bu^2),\bv] 
\,=\, 0 \quad \forall\,\bv\in \bM,
\end{equation*}
which together with the property $\bdiv(\bbX_0) = \bM'$ allow us to deduce that
$\bdiv\bsi^1(t) = \bdiv\bsi^2(t)$ for all $t\in (0,T]$. Employing the inequality \eqref{eq:deviatoric-inequality}, we conclude that $\bsi^1(t) = \bsi^2(t)$ for all $t\in (0,T]$,
and therefore \eqref{eq:weak-Brinkman-Forchheimer} has a unique solution.
\end{proof}

We conclude this section with stability bounds for the solution of \eqref{eq:weak-Brinkman-Forchheimer}.
\begin{thm}\label{thm:stability-bound}
For the solution of \eqref{eq:weak-Brinkman-Forchheimer}, assuming sufficient regularity of the data, there exists a positive constant $C_{\cF}$ only depending on $|\Omega|, \nu, \alpha, \tF, \beta$, such that
\begin{equation}\label{eq:stability-bound}
\|\bsi\|_{\L^{2}(0,T;\bbX)} + \|\bu\|_{\L^{2}(0,T;\bM)} + \|\bu\|_{\L^\infty(0,T;\bL^2(\Omega))} 
\,\leq\, C_{\cF}\,\cF(\f,\bsi(0),\bu(0))
\end{equation}
with
\begin{equation}\label{eq:cF-definition}
\begin{array}{l}
\ds \cF(\f,\bsi(0),\bu(0)) 
\,:=\, \|\f\|^{\rp-1}_{\L^{2(\rp-1)}(0,T;\bL^2(\Omega))}
  + \|\f\|_{\L^2(0,T;\bL^2(\Omega))} \\ [2ex] 
  \ds\quad\,+\,\,\|\bsi^\rd(0)\|_{0,\Omega} + \|\bu(0)\|^{\rp/2}_{\bM} + \|\bu(0)\|^{\rp-1}_{0,\Omega}
  + \|\bu(0)\|_{0,\Omega}.
\end{array}
\end{equation}
Moreover, there exists a positive constant $C_{\cG}$ defined below in \eqref{eq:constant-CG-definition}, such that
\begin{equation}\label{eq:stability-bound-u-Linf-Lp}
\|\bu\|_{\L^{\infty}(0,T;\bM)} 
\,\leq\, C_{\cG}\,\cG(\f,\bsi(0),\bu(0))
\end{equation}
with
\begin{equation}\label{eq:cG-definition}
\cG(\f,\bsi(0),\bu(0)) 
\,:=\, \|\f\|^{2/\rp}_{\L^2(0,T;\bL^2(\Omega))} + \|\bsi^\rd(0)\|^{2/\rp}_{0,\Omega} + \|\bu(0)\|_{\bM} + \|\bu(0)\|^{2/\rp}_{0,\Omega},
\end{equation}
where $\bsi(0)$ is such that \eqref{eq:sol0-X-M} holds.
\end{thm}
\begin{proof}
We begin with choosing $(\btau,\bv) = (\bsi,\bu)$ in \eqref{eq:weak-Brinkman-Forchheimer} to get
\begin{equation*}
\frac{1}{2}\,\partial_t\,(\bu,\bu)_{\Omega} + [\cA(\bsi),\bsi] + [\cC(\bu),\bu] \,=\, (\f,\bu)_\Omega.
\end{equation*}
Next, we use the coercivity bounds in \eqref{eq:properties-operator-A}--\eqref{eq:properties-operator-C-1}, and Cauchy--Schwarz and Young's inequalities with $\rp = \rq = 2$ (cf. \eqref{eq:Young-inequality}), to obtain
\begin{equation}\label{eq:preliminary-stability-bound}
\frac{1}{2}\,\partial_t\,\|\bu\|^2_{0,\Omega} + \frac{1}{\nu}\,\|\bsi^\rd\|^2_{0,\Omega} + \tF\,\|\bu\|^{\rp}_{\bM} + \alpha\,\|\bu\|^2_{0,\Omega}
\,\leq\, \|\f\|_{\bM'}\,\|\bu\|_{\bM} 
\,\leq\, \frac{\delta}{2}\,\|\f\|^{2}_{0,\Omega} + \frac{1}{2\,\delta}\,\|\bu\|^2_{0,\Omega}.
\end{equation}
In turn, from the inf-sup condition of $\cB$ (cf. \eqref{eq:inf-sup-operator-B}), the first equation of \eqref{eq:weak-Brinkman-Forchheimer}, and the continuity bound of $\cA$ (cf. \eqref{eq:properties-operator-A}), we deduce that
\begin{equation*}
\beta\,\|\bu\|_{\bM} \,\leq\, \sup_{\btau\in \bbX_0} \frac{[\cB(\btau),\bu]}{\|\btau\|_{\bbX}} \,=\, -\,\sup_{\btau\in \bbX_0} \frac{[\cA(\bsi),\btau]}{\|\btau\|_{\bbX}} \,\leq\, \frac{1}{\nu}\,\|\bsi^\rd\|_{0,\Omega},
\end{equation*}
and then
\begin{equation*}
\frac{\beta^2\,\nu}{2}\,\|\bu\|^2_{\bM} \,\leq\, \frac{1}{2\,\nu}\,\|\bsi^\rd\|^2_{0,\Omega},
\end{equation*}
which combined with \eqref{eq:preliminary-stability-bound} and choosing $\delta = 1/\alpha$,
yields
\begin{equation}\label{eq:stability-bound-1}
  \frac{1}{2}\,\partial_t\,\|\bu\|^2_{0,\Omega} + \frac{1}{2\,\nu}\,\|\bsi^\rd\|^2_{0,\Omega}
  + \frac{\beta^2\,\nu}{2}\,\|\bu\|^{2}_{\bM} + \frac{\alpha}{2}\,\|\bu\|^2_{0,\Omega}
\,\leq\, \frac{1}{2\,\alpha}\,\|\f\|^{2}_{0,\Omega}.
\end{equation}
Notice that, in order to simplify the stability bound, we have neglected the term $\|\bu\|^{\rp}_{\bM}$ in the left-hand side of \eqref{eq:preliminary-stability-bound}.
Integrating \eqref{eq:stability-bound-1} from $0$ to $t\in (0,T]$, we obtain
\begin{equation}\label{eq:stability-bound-2}
\|\bu(t)\|^2_{0,\Omega} + \int^t_0 \Big( \|\bsi^\rd\|^2_{0,\Omega} + \|\bu\|^{2}_{\bM} + \|\bu\|^{2}_{0,\Omega} \Big)\, ds 
\,\leq\, C_1\,\Bigg( \int^t_0 \|\f\|^{2}_{0,\Omega}\, ds + \|\bu(0)\|^2_{0,\Omega} \Bigg),
\end{equation}
with $C_1>0$ depending only on $\nu, \alpha$, and $\beta$.

On the other hand, from the second equation of \eqref{eq:weak-Brinkman-Forchheimer}, we have the identity
\begin{equation*}
[\cB(\bsi),\bv] 
\,=\, -\,(\f,\bv)_{\Omega} + [\cC(\bu),\bv] + (\partial_t\,\bu,\bv)_{\Omega} \quad \forall\,\bv\in \bM.
\end{equation*}
Then, employing the property $\bdiv(\bbX_0) = \bM'$ in combination with \eqref{eq:properties-operator-C-1} and using similar arguments to \eqref{eq:stability-bound-2}, we deduce
\begin{equation}\label{eq:bound-int-div-bsi-1}
\begin{array}{l}
\ds \int^t_0 \|\bdiv\bsi\|^2_{\bL^{\rq}(\Omega)} \,ds 
\,\leq\, \wt{C}_2\,\int^t_0 \Big( \|\f\|^2_{0,\Omega} + \|\bu\|^{2(\rp-1)}_{\bM} + \|\bu\|^2_{0,\Omega} + \|\partial_t\,\bu\|^2_{0,\Omega} \Big)\,ds  \\ [3ex]
\ds\quad \,\leq\, C_2\,\Bigg( \int^t_0 \Big( \|\f\|^{2(\rp-1)}_{0,\Omega} + \|\f\|^2_{0,\Omega} \Big)\, ds + \|\bu(0)\|^{2(\rp-1)}_{0,\Omega} + \|\bu(0)\|^2_{0,\Omega} + \int^t_0 \|\partial_t\,\bu\|^2_{0,\Omega}\,ds \Bigg),
\end{array}
\end{equation}
with $C_2>0$ depending on $|\Omega|, \nu, \tF, \alpha$, and $\beta$.
Next, in order to bound the last term in \eqref{eq:bound-int-div-bsi-1}, we differentiate in time the first equation of \eqref{eq:weak-Brinkman-Forchheimer}, choose $(\btau, \bv) = (\bsi, \partial_t\,\bu)$, and employ Cauchy--Schwarz and Young's inequalities with $\rp = \rq = 2$ and $\delta = 1$ (cf. \eqref{eq:Young-inequality}), to obtain
\begin{equation*}
\frac{1}{2}\,\partial_t\,\Big( \frac{1}{\nu}\,\|\bsi^\rd\|^2_{0,\Omega} + \frac{2\,\tF}{\rp}\,\|\bu\|^{\rp}_{\bM} + \alpha\,\|\bu\|^2_{0,\Omega} \Big) + \|\partial_t\,\bu\|^2_{0,\Omega} 
\,\leq\, \frac{1}{2}\,\|\f\|^2_{0,\Omega} + \frac{1}{2}\,\|\partial_t\,\bu\|^2_{0,\Omega}. 
\end{equation*}
Integrating from $0$ to $t\in (0,T]$, we get
\begin{equation}\label{eq:bound-ut-int-ut}
\frac{2\,\tF}{\rp}\|\bu(t)\|^{\rp}_{\bM} + \int^t_0 \|\partial_t\,\bu\|^2_{0,\Omega}\,ds 
\,\leq\, C_3\,\Bigg( \int^t_0 \|\f\|^2_{0,\Omega}\,ds + \|\bsi^\rd(0)\|^2_{0,\Omega} 
+ \|\bu(0)\|^{\rp}_{\bM} + \|\bu(0)\|^2_{0,\Omega} \Bigg),
\end{equation}
with $C_3 := \max\big\{ 1,1/\nu,2\,\tF/\rp,\alpha \big\}$. 
Then, combining \eqref{eq:bound-ut-int-ut} with \eqref{eq:bound-int-div-bsi-1}, yields
\begin{equation*}%\label{eq:bound-int-div-bsi-2}
\begin{array}{l}
\ds \int^t_0 \|\bdiv\bsi\|^2_{\bL^{\rq}(\Omega)}\,ds \\ [3ex] 
\ds\quad\,\leq\, C_4\,\Bigg( \int^t_0 \Big( \|\f\|^{2(\rp-1)}_{0,\Omega} + \|\f\|^2_{0,\Omega} \Big)\,ds 
+ \|\bsi^\rd(0)\|^2_{0,\Omega} + \|\bu(0)\|^{\rp}_{\bM} + \|\bu(0)\|^{2(\rp-1)}_{0,\Omega} + \|\bu(0)\|^2_{0,\Omega} \Bigg),
\end{array}
\end{equation*}
which, combined with \eqref{eq:stability-bound-2} and \eqref{eq:deviatoric-inequality},
implies \eqref{eq:stability-bound}.
In addition, \eqref{eq:bound-ut-int-ut} implies \eqref{eq:stability-bound-u-Linf-Lp} with
\begin{equation}\label{eq:constant-CG-definition}
C_{\cG} \,:=\, \Big(\frac{\rp}{2\,\tF}\,\max\Big\{1,\frac{1}{\nu},\frac{2\,\tF}{\rp},\alpha \Big\}
\Big)^{1/\rp},
\end{equation}
concluding the proof.
\end{proof}

\begin{rem}
The stability bound \eqref{eq:stability-bound} can be derived
alternatively without the use of the inf-sup condition of $\cB$, but
in that case the expression of \eqref{eq:cF-definition} will be more
complicated.  We also note that \eqref{eq:stability-bound-u-Linf-Lp}
will be employed in the next section to deal with the nonlinear term
associated to the operator $\cC$ (cf. \eqref{eq:operators-A-B-C}),
which is necessary to obtain the corresponding error estimate.
\end{rem}

%************************************************************************
%************************************************************************

\section{Semidiscrete continuous-in-time approximation}\label{sec:semidiscrete-approximation}

In this section we introduce and analyze the semidiscrete continuous-in-time approximation of \eqref{eq:weak-Brinkman-Forchheimer}.
We analyze its solvability by employing the strategy developed in Section~\ref{sec:well-posedness-model}.
Finally, we perform error analysis and obtain rates of convergence of our discrete scheme.

%************************************************************************

\subsection{Existence and uniqueness of a solution}

Let $\cT_h$ be a shape-regular triangulation of $\Omega$ consisting of triangles (when $n=2$) or tetrahedra $K$ (when $n=3$) of diameter $h_K$, and define the mesh-size $h:=\max\big\{ h_K :\, K\in \cT_h \big\}$.
Then, for each $K\in \cT_h$, we let $\bRT_k(K)$ be the local Raviart--Thomas element of order $k\geq 0$, i.e., 
\begin{equation*}
\bRT_k(K) \,:=\, [\rP_k(K)]^n \oplus \rP_k(K)\,\bx,
\end{equation*}
where $\bx := (x_1,\dots,x_n)^\rt$ is a generic vector of $\R^n$ and $\rP_k(K)$ is the space of polynomials defined on $K$ of degree $\leq k$.
The finite element subspaces on $\Omega$ are defined as
\begin{equation*}
\begin{array}{l}
\ds \bbX_{h} \,:=\, \Big\{ \btau_h\in \bbX :\quad \bc^\rt\btau_h|_{K}\in \bRT_k(K) \quad \forall\,\bc\in \R^n \quad \forall\,K\in \cT_h \Big\},\quad \bbX_{0,h} \,:=\, \bbX_{h}\cap \bbX_0, \\ [3ex]
\ds \bM_h \,:=\, \Big\{ \bv_h\in \bM :\quad \bv_h|_{K}\in [\rP_k(K)]^n \quad \forall\,K\in \cT_h \Big\}.
\end{array}
\end{equation*}
The semidiscrete continuous-in-time problem associated with \eqref{eq:weak-Brinkman-Forchheimer} reads: for $t\in (0,T]$, find $(\bsi_h(t),\newline \bu_h(t))\in \L^\infty(0,T;\bbX_{0,h})\times \W^{1,\infty}(0,T;\bM_h)$, such that
\begin{equation}\label{eq:discrete-weak-Brinkman-Forchheimer}
\begin{array}{rlll}
\ds [\cA(\bsi_h),\btau_h] + [\cB(\btau_h),\bu_h] & = & 0 & \forall\, \btau_h\in \bbX_{0,h}, \\ [3ex]
\ds (\partial_t\,\bu_h,\bv_h)_{\Omega} - [\cB(\bsi_h),\bv_h] + [\cC(\bu_h),\bv_h] & = & (\f,\bv_h)_{\Omega} & \forall\,\bv_h\in \bM_h.
\end{array}
\end{equation}
As initial condition we take $(\bsi_h(0),\bu_h(0))$ to be a suitable approximations of
$(\bsi(0),\bu(0))$, the solution of \eqref{eq:initial-condition-problem}, that is, we chose $(\bsi_h(0),\bu_h(0))$ solving
\begin{equation}\label{eq:discrete-initial-condition-problem}
\begin{array}{rlll}
[\cA(\bsi_h(0)),\btau_h] + [\cB(\btau_h),\bu_h(0)] & = & 0 & \forall\,\btau_h\in \bbX_{0,h} \\ [3ex]
-\,[\cB(\bsi_h(0)),\bv_h] + [\cC(\bu_h(0)),\bv_h] & = & [\bG_0,\bv_h] & \forall\,\bv_h\in \bM_h,
\end{array}
\end{equation} 
with $\bG_0\in \bM'_2$ being the right-hand side of \eqref{eq:initial-condition-problem}.
This choice is necessary to guarantee that the discrete initial data is compatible in the sense of Lemma~\ref{lem:sol0-X-M}, which is needed for the application of Theorem~\ref{thm:well-posed-parabolic-problem}.
Notice that the well-posedness of problem \eqref{eq:discrete-initial-condition-problem}
follows from similar arguments to the proof of Lemma~\ref{lem:resolvent-system}.

Next, we recall for later use the discrete inf-sup condition associated to the operator $\cB$.
\begin{lem}\label{lem:discrete-inf-sup-operator-B}
There exists a constant $\wt{\beta} > 0$ such that
\begin{equation}\label{eq:discrete-inf-sup-operator-B}
\sup_{\0\neq \btau_h\in \bbX_{0,h}} \frac{[\cB(\btau_h),\bv_h]}{\|\btau_h\|_{\bbX}} 
\,\geq\, \wt{\beta}\,\|\bv_h\|_{\bM} \quad \forall\,\bv_h\in \bM_h.
\end{equation}
\end{lem}
\begin{proof}
For the case $\rp = 4$ and $\rq = 4/3$ we refer the reader to \cite[Lemma~4.4]{cgo2018-pp}, which proof can be easily extended to the case $\rp\in [3,4]$ and $1/\rp + 1/\rq = 1$.
We omit further details.
\end{proof}

We can establish the following well-posedness result.
\begin{thm}\label{thm:well-posed-discrete-result}
For each $\f\in \W^{1,1}(0,T;\bL^2(\Omega))$ and $(\bsi_h(0),\bu_h(0))$ satisfying \eqref{eq:discrete-initial-condition-problem}, there exists a unique solution to \eqref{eq:discrete-weak-Brinkman-Forchheimer} with $(\bsi_h(t), \bu_h(t))\in \L^\infty(0,T;\bbX_{0,h})\times \W^{1,\infty}(0,T;\bM_h)$.
Moreover, assuming sufficient regularity of the data, there exist a positive constant $\wt{C}_{\cF}$ depending only on $|\Omega|, \nu, \alpha, \tF, \wt{\beta}$, such that 
\begin{equation}\label{eq:discrete-stability-bound}
\|\bsi_h\|_{\L^{2}(0,T;\bbX)} 
+ \|\bu_h\|_{\L^{2}(0,T;\bM)} 
+ \|\bu_h\|_{\L^\infty(0,T;\bL^2(\Omega))}
\,\leq\, \wt{C}_{\cF}\,\cF(\f,\bsi_h(0),\bu_h(0)),
\end{equation}
and with $C_{\cG}$ defined in \eqref{eq:constant-CG-definition}, there holds
\begin{equation}\label{eq:discrete-stability-bound-u-Linf-Lp}
\|\bu_h\|_{\L^{\infty}(0,T;\bM)} 
\,\leq\, C_{\cG}\,\cG(\f,\bsi_h(0),\bu_h(0)),
\end{equation}
where $\cF$ and $\cG$ are defined in \eqref{eq:cF-definition} and \eqref{eq:cG-definition}, respectively.
\end{thm}
\begin{proof}
From the fact that $\bbX_{0,h}\subset \bbX_0, \bM_h\subset \bM$, and $\bdiv(\bbX_{0,h}) = \bM_h$, considering $(\bsi_h(0),\bu_h(0))$ satisfying \eqref{eq:discrete-initial-condition-problem}, and employing the discrete inf-sup condition of $\cB$ (cf. \eqref{eq:discrete-inf-sup-operator-B}), the proof is identical to the proofs of Theorems~\ref{thm:well-posed-result} and \ref{thm:stability-bound}.
\end{proof}

%************************************************************************

\subsection{Error analysis}

We proceed with establishing rates of convergence. Let
$\be_{\bsi} = \bsi - \bsi_h$ and $\be_{\bu} = \bu - \bu_h$.
As usual, we shall then decompose these errors into
\begin{equation}\label{eq:error-decomposition}
\be_{\bsi} = \bdelta_{\bsi} + \bbeta_{\bsi},\quad
\be_{\bu} = \bdelta_{\bu} + \bbeta_{\bu},
\end{equation}
with
\begin{equation*}%\label{eq:delta-eta-definition}
\bdelta_{\bsi} = \bsi - \bPi^k_h(\bsi),\quad 
\bbeta_{\bsi} = \bPi^k_h(\bsi) - \bsi_h,\quad
\bdelta_{\bu} = \bu - \bP^k_h(\bu),\quad 
\bbeta_{\bu} = \bP^k_h(\bu) - \bu_h,
\end{equation*}
where $\bP^k_h : \bM \to \bM_h$ is the vector version of the usual orthogonal projector with respect to $\L^2$-inner product, satisfying:
\begin{equation}\label{eq:Pk-projection-property}
(\bu - \bP^k_h(\bu),\bv_h)_{\Omega} \,=\, 0 \quad \forall\,\bv_h\in \bM_h,
\end{equation}
and, given $\rs\geq 2\,n/(n+2)$, $\bPi^k_h : \bbH_{\rs} \to \bbX_h$ is the Raviart--Thomas
interpolation operator, where
\begin{equation}\label{Hs-defn}
\bbH_{\rs} \,:=\, \Big\{ \btau\in \bbX :\quad \btau|_{K}\in \bbW^{1,\rs}(K)\quad \forall\,K\in \cT_h \Big\},
\end{equation}
satisfying \cite[Lemma~1.41]{Ern-Guermond}:
\begin{equation}\label{eq:Pik-interpolator}
\bdiv(\bPi^k_h(\btau)) \,=\, \bP^k_h(\bdiv(\btau)) \quad \forall\,\btau\in \bbH_{\rs}.
\end{equation}

Subtracting the discrete problem \eqref{eq:discrete-weak-Brinkman-Forchheimer} from the continuous one \eqref{eq:weak-Brinkman-Forchheimer}, we obtain the following error system:
\begin{equation}\label{eq:error-system}
\begin{array}{rll}
[\cA(\bsi - \bsi_h),\btau_h] + [\cB(\btau_h),\bu - \bu_h] & = & 0 \quad \forall\,\btau_h\in \bbX_{0,h}, \\ [3ex]
(\partial_t\,(\bu - \bu_h),\bv_h)_{\Omega} - [\cB(\bsi - \bsi_h),\bv_h] + [\cC(\bu) - \cC(\bu_h),\bv_h] & = & 0 \quad \forall\,\bv_h\in \bM_h.
\end{array}
\end{equation}

We now establish the main result of this section.
\begin{thm}\label{thm:error-estimates}	
Let $(\bsi(t),\bu(t))\in \L^{\infty}(0,T;\bbX_0)\times \W^{1,\infty}(0,T;\bL^2(\Omega))\cap \L^\infty(0,T;\bM)$ and $(\bsi_h(t), \bu_h(t))\in \L^{\infty}(0,T;\bbX_{0,h})\times \W^{1,\infty}(0,T;\bM_h)$ be the unique solutions of the continuous and semidiscrete problems \eqref{eq:weak-Brinkman-Forchheimer} and \eqref{eq:discrete-weak-Brinkman-Forchheimer}, respectively.
Assume that $\bsi\in \bbX_0\cap \bbH_{\rs}$.
Then there exist $\wt{C}, \wh{C} > 0$ depending only on
$|\Omega|, \nu, \alpha, \tF, \wt{\beta}, \rp$, and data, such that
\begin{equation}\label{eq:error-estimate}
\begin{array}{l}
\ds \|\be^{\rd}_{\bsi}\|_{\L^2(0,T;\bbL^2(\Omega))} + \|\be_{\bu}\|_{\L^{2}(0,T;\bM)} + \|\be_{\bu}\|_{\L^\infty(0,T;\bL^2(\Omega))} \\ [2ex]
\ds\quad\,\leq\, \wt{C}\,\Big( \|\bdelta_{\bsi}\|_{\L^2(0,T;\bbX)} + \|\bdelta_{\bu}\|^{\rp/2}_{\L^2(0,T;\bM)} + \|\bdelta_{\bu}\|^{\rq/2}_{\L^2(0,T;\bM)} + \|\bdelta_{\bu}\|_{\L^2(0,T;\bM)} + \|\bdelta_{\bu}\|_{\L^\infty(0,T;\bM)} \\ [2ex] 
\ds\quad\,+\,\,\|\bdelta_{\bsi}(0)\|_{\bbX} + \|\bdelta_{\bu}(0)\|^{\rp/2}_{\bM} + \|\bdelta_{\bu}(0)\|^{\rq/2}_{\bM} + \|\bdelta_{\bu}(0)\|_{\bM} \Big)
\end{array}
\end{equation}
%Moreover, assuming that $\bu\in \bL^{2\rp}(\Omega)$, with $\rp\in [3,4]$, there exists $C > 0$ depending only on $|\Omega|, \nu, \alpha, \tF$, $\wt{\beta}, \rp, \|\bu\|_{\L^\infty(0,T;\bL^{2\rp}(\Omega))}$, and data, such that
and
\begin{equation}\label{eq:full-error-estimate}
\begin{array}{l}
\ds \|\be_{\bsi}\|_{\L^2(0,T;\bbX)} \\ [2ex]
\ds\quad\,\leq\, \wh{C}\,h^{-n(\rp-2)/(2\rp)} \Big( \|\partial_t\,\bdelta_{\bsi}\|_{\L^2(0,T;\bbX)} + \|\bdelta_{\bsi}\|_{\L^2(0,T;\bbX)} + \|\bdelta_{\bu}\|^{\rp/2}_{\L^2(0,T;\bM)} + \|\bdelta_{\bu}\|^{\rq/2}_{\L^2(0,T;\bM)} \\ [2ex] 
\ds\quad\,+\,\,\|\bdelta_{\bu}\|_{\L^2(0,T;\bM)} + \|\bdelta_{\bu}\|_{\L^\infty(0,T;\bM)} + \|\bdelta_{\bsi}(0)\|_{\bbX} + \|\bdelta_{\bu}(0)\|^{\rp/2}_{\bM} + \|\bdelta_{\bu}(0)\|^{\rq/2}_{\bM} + \|\bdelta_{\bu}(0)\|_{\bM} \Big).
\end{array}
\end{equation}
\end{thm}
% \|\bu\|^{\rp-2}_{\L^\infty(0,T;\bL^{2\rp}(\Omega))}
%
\begin{proof}
We first note that $\bsi\in \bbX_0\cap \bbH_{\rs}$ (cf. \eqref{Hs-defn}), so $\bPi^k_h(\bsi)$ is well-defined.
Then, adding and subtracting suitable terms in \eqref{eq:error-system} with $(\btau_h,\bv_h) = (\bbeta_{\bsi},\bbeta_{\bu})\in \bbX_{0,h}\times \bM_h$, and employing the lower bounds of $\cA$ and $\cC$ (cf. \eqref{eq:properties-operator-A}, \eqref{eq:properties-operator-C-3}), we deduce
\begin{equation}\label{eq:error-bound-1}
\begin{array}{l}
\ds \frac{1}{2}\,\partial_t\,(\bbeta_\bu,\bbeta_\bu)_{\Omega} + \frac{1}{\nu}\,\|\bbeta^\rd_{\bsi}\|^2_{0,\Omega} + \tF\,C_{\rp}\,\|\bbeta_{\bu}\|^{\rp}_{\bM} + \alpha\,\|\bbeta_{\bu}\|^2_{0,\Omega} \\ [2ex]
\ds\quad \,=\, -\,[\cA(\bdelta_\bsi),\bbeta_\bsi] - [\cC(\bu) - \cC(\bP^k_h(\bu)),\bbeta_\bu] - (\partial_t\,\bdelta_\bu,\bbeta_\bu)_{\Omega} + [\cB(\bdelta_\bsi),\bbeta_\bu] - [\cB(\bbeta_\bsi),\bdelta_\bu]. 
\end{array}
\end{equation}
Notice that $\bdiv(\bbeta_{\bsi})\in \bM_h$. 
Then, from the projection properties \eqref{eq:Pk-projection-property}--\eqref{eq:Pik-interpolator}, the last three terms in \eqref{eq:error-bound-1} are zero.
In addition, using the continuity bounds of $\cA$ and $\cC$ (cf. \eqref{eq:properties-operator-A}, \eqref{eq:properties-operator-C-2a}--\eqref{eq:properties-operator-C-2}), H\"older and Young's inequalities with $\delta = (\rp\,\tF\,C_{\rp}/2)^{2/\rp}$ (cf. \eqref{eq:Young-inequality}), and inequality \cite[Lemma~2.2]{Adams-Fournier}:
\begin{equation}\label{eq:ab-m-power-inequality}
(a + b)^{m} \leq 2^{m-1}(a^{m} + b^{m}) \quad\forall\,a,b\geq 0,\,\,\forall\, m\geq 1,
\end{equation}
with $m = \rp-2\in [1,2]$, we deduce that the right-hand side of \eqref{eq:error-bound-1} is bounded as follows:
\begin{equation}\label{eq:error-bound-1b}
\begin{array}{l}
\ds [\cA(\bdelta_{\bsi}),\bbeta_{\bsi}] + [\cC(\bu) - \cC(\bP^k_h(\bu)),\bbeta_{\bu}] \\ [2ex]
\ds\quad\,\leq\, \frac{1}{\nu}\,\|\bdelta_{\bsi}\|_{0,\Omega}\,\|\bbeta^\rd_{\bsi}\|_{0,\Omega} + \alpha\,\|\bdelta_{\bu}\|_{0,\Omega}\,\|\bbeta_{\bu}\|_{0,\Omega} + \tF\,c_{\rp}\,\big( \|\bdelta_{\bu}\|_{\bM} + 2\,\|\bu\|_{\bM} \big)^{\rp-2}\,\|\bdelta_{\bu}\|_{\bM}\,\|\bbeta_{\bu}\|_{\bM} \\ [2ex]
\ds\quad\,\leq\, \frac{1}{2\,\nu}\,\|\bdelta_{\bsi}\|^2_{0,\Omega} + \frac{\alpha}{2}\,\|\bdelta_{\bu}\|^2_{0,\Omega} + C(\rp,\rq)\,\big( \|\bdelta_{\bu}\|^{\rp}_{\bM} + \|\bu\|^{\rq(\rp-2)}_{\bM}\,\|\bdelta_{\bu}\|^{\rq}_{\bM} \big) \\ [2ex] 
\ds\quad\,+\,\,\frac{1}{2}\,\Big( \frac{1}{\nu}\,\|\bbeta^\rd_{\bsi}\|^2_{0,\Omega} + \tF\,C_{\rp}\,\|\bbeta_{\bu}\|^{\rp}_{\bM} + \alpha\,\|\bbeta_{\bu}\|^2_{0,\Omega} \Big),
\end{array}
\end{equation}
where $C(\rp,\rq) := 2^{2\rq(\rp-2)+\rq/\rp-1}\,\tF\,c^{\rq}_{\rp}/(\rq\,(\rp\,C_\rp)^{\rq/\rp})$.
Thus, combining \eqref{eq:error-bound-1} with \eqref{eq:error-bound-1b}, we obtain
\begin{equation}\label{eq:error-bound-1c}
\begin{array}{l}
\ds \partial_t\,\|\bbeta_{\bu}\|^2_{0,\Omega} + \frac{1}{\nu}\,\|\bbeta^\rd_{\bsi}\|^2_{0,\Omega} + \tF\,C_{\rp}\,\|\bbeta_{\bu}\|^{\rp}_{\bM} + \alpha\,\|\bbeta_{\bu}\|^2_{0,\Omega} \\ [2ex] 
\ds\quad\,\leq\, \frac{1}{\nu}\,\|\bdelta_{\bsi}\|^2_{0,\Omega} + \alpha\,\|\bdelta_{\bu}\|^2_{0,\Omega} + 2\,C(\rp,\rq)\,\big( \|\bdelta_{\bu}\|^{\rp}_{\bM} + \|\bu\|^{\rq(\rp-2)}_{\bM}\,\|\bdelta_{\bu}\|^{\rq}_{\bM} \big).
\end{array}
\end{equation}
Then, integrating \eqref{eq:error-bound-1c} from $0$ to $t\in (0,T]$ and recalling that $\|\bu\|_{\L^{\infty}(0,T;\bM)}$ is bounded by data (cf. \eqref{eq:stability-bound-u-Linf-Lp}), we find that
\begin{equation}\label{eq:error-bound-2}
\begin{array}{l}
\ds \|\bbeta_{\bu}(t)\|^2_{0,\Omega} + \int^t_0 \Big( \frac{1}{\nu}\,\|\bbeta^\rd_{\bsi}\|^2_{0,\Omega} + \tF\,C_{\rp}\,\|\bbeta_{\bu}\|^{\rp}_{\bM} + \alpha\,\|\bbeta_{\bu}\|^2_{0,\Omega} \Big)\,ds  \\ [3ex]
\ds\quad\,\leq\, C_1\,\Bigg( \int^t_0 \Big(\|\bdelta_{\bsi}\|^2_{0,\Omega} + \|\bdelta_{\bu}\|^{\rp}_{\bM} + \|\bdelta_{\bu}\|^{\rq}_{\bM} + \|\bdelta_{\bu}\|^2_{0,\Omega} \Big)\,ds + \|\bbeta_{\bu}(0)\|^2_{0,\Omega} \Bigg),
\end{array}
\end{equation}
with $C_1 > 0$ depending only on $|\Omega|, \nu, \alpha, \tF, \rp$, and data. 
Next, from the discrete inf-sup condition of $\cB$ (cf. \eqref{eq:discrete-inf-sup-operator-B}) and the first equation of \eqref{eq:error-system}, we get
\begin{equation*}
\wt{\beta}\,\|\bbeta_{\bu}\|_{\bM} 
\,\leq\, \frac{1}{\nu}\,\|\bbeta^{\rd}_{\bsi}\|_{0,\Omega} + \big( \frac{1}{\nu}\,\|\bdelta^{\rd}_{\bsi}\|_{0,\Omega} + \|\bdelta_{\bu}\|_{\bM} \big),
\end{equation*}
and then
\begin{equation}\label{eq:discrete-inf-sup-bbeta-bdelta}
\frac{\wt{\beta}^2\,\nu}{4}\,\|\bbeta_{\bu}\|^2_{\bM} 
\,\leq\, \frac{1}{2\,\nu}\,\|\bbeta^{\rd}_{\bsi}\|^2_{0,\Omega} + \frac{1}{\nu}\,\|\bdelta_{\bsi}\|^2_{0,\Omega} + \nu\,\|\bdelta_{\bu}\|^2_{\bM}.
\end{equation}
Hence, integrating \eqref{eq:discrete-inf-sup-bbeta-bdelta} from $0$ to $t\in (0,T]$, adding with \eqref{eq:error-bound-2}, and neglecting $\|\bbeta_{\bu}\|^{\rp}_{\bM}$ to obtain a simplified error estimate, yields
\begin{equation}\label{eq:new-error-bound-2}
\begin{array}{l}
\ds \|\bbeta_{\bu}(t)\|^2_{0,\Omega} + \int^t_0 \Big( \|\bbeta^\rd_{\bsi}\|^2_{0,\Omega} + \|\bbeta_{\bu}\|^{2}_{\bM} + \|\bbeta_{\bu}\|^2_{0,\Omega} \Big)\,ds \\ [3ex] 
\ds\quad\,\leq\, C_2\,\Bigg(\int^t_0 \Big( \|\bdelta_{\bsi}\|^2_{0,\Omega} + \|\bdelta_{\bu}\|^{\rp}_{\bM} + \|\bdelta_{\bu}\|^{\rq}_{\bM} + \|\bdelta_{\bu}\|^2_{\bM} \Big)\,ds  + \|\bbeta_{\bu}(0)\|^2_{0,\Omega} \Bigg),
\end{array}
\end{equation}
with $C_2 > 0$ depending only on $|\Omega|, \nu, \alpha, \tF, \wt{\beta}$, and data.

Next, in order to bound the last term in \eqref{eq:new-error-bound-2}, we subtract the continuous and discrete initial condition problems \eqref{eq:initial-condition-problem} and \eqref{eq:discrete-initial-condition-problem}, to obtain the error system:
\begin{equation*}
\begin{array}{rll}
[\cA(\bsi(0) - \bsi_h(0)),\btau_h] + [\cB(\btau_h),\bu(0) - \bu_h(0)] & = & 0 \quad \forall\,\btau_h\in \bbX_{0,h}, \\ [2ex]
-\,[\cB(\bsi(0) - \bsi_h(0)),\bv_h] + [\cC(\bu(0)) - \cC(\bu_h(0)),\bv_h] & = & 0 \quad \forall\,\bv_h\in \bM_h.
\end{array}
\end{equation*} 
Then, using similar arguments above (cf. \eqref{eq:error-bound-1c}), allow us to get
\begin{equation}\label{eq:data-0-bound}
\|\bbeta^\rd_{\bsi}(0)\|^2_{0,\Omega} + \|\bbeta_{\bu}(0)\|^{\rp}_{\bM} + \|\bbeta_{\bu}(0)\|^2_{0,\Omega} 
\,\leq\, C_{0}\,\Big( \|\bdelta_{\bsi}(0)\|^2_{0,\Omega} + \|\bdelta_{\bu}(0)\|^{\rp}_{\bM} + \|\bdelta_{\bu}(0)\|^{\rq}_{\bM} + \|\bdelta_{\bu}(0)\|^2_{\bM} \Big),
\end{equation}
with $C_{0}>0$ depending only on $|\Omega|, \nu, \tF, \alpha$.
Thus, combining \eqref{eq:data-0-bound} with \eqref{eq:new-error-bound-2}, and using the error decompositions \eqref{eq:error-decomposition}, yields
\begin{equation}\label{eq:new-error-bound-3}
\begin{array}{l}
\ds \|\be_{\bu}(t)\|^2_{0,\Omega} + \int^t_0 \Big(\|\be^{\rd}_{\bsi}\|^2_{0,\Omega} + \|\be_{\bu}\|^{2}_{\bM}\Big)\,ds
\,\leq\, C\,\Bigg( \|\bdelta_{\bu}(t)\|^{2}_{0,\Omega} + \int^t_0 \Big( \|\bdelta_{\bsi}\|^2_{0,\Omega} + \|\bdelta_{\bu}\|^{\rp}_{\bM} \Big)\,ds \\ [3ex] 
\ds\quad\,+\,\,\int^t_0\Big( \|\bdelta_{\bu}\|^{\rq}_{\bM} + \|\bdelta_{\bu}\|^2_{\bM} \Big)\,ds
+ \|\bdelta_{\bsi}(0)\|^2_{0,\Omega} + \|\bdelta_{\bu}(0)\|^{\rp}_{\bM} + \|\bdelta_{\bu}(0)\|^{\rq}_{\bM} + \|\bdelta_{\bu}(0)\|^2_{\bM} \Bigg),
\end{array}
\end{equation}
which implies \eqref{eq:error-estimate}.

On the other hand, in order to obtain \eqref{eq:full-error-estimate}, we need to bound the norm $\|\bdiv(\bbeta_{\bsi})\|_{\bL^\rq(\Omega)}$.
To that end, from the second equation of \eqref{eq:error-system} we have the identity
\begin{equation*}
[\cB(\bbeta_{\bsi}),\bv_h] \,=\, (\partial_t\,\bbeta_{\bu},\bv_h)_{\Omega} + [\cC(\bu) - \cC(\bu_h),\bv_h] + (\partial_t\,\bdelta_{\bu},\bv_h)_{\Omega} - [\cB(\bdelta_{\bsi}),\bv_h] \quad \forall\,\bv_h\in \bM_h,
\end{equation*}
where the last two terms are zero thanks to the projection properties \eqref{eq:Pk-projection-property}--\eqref{eq:Pik-interpolator},
which together with the continuity bound of $\cC$ (cf. \eqref{eq:properties-operator-C-2}) and the fact that $\bdiv(\bbX_0) = \bM'$, implies
\begin{equation*}
\|\bdiv(\bbeta_{\bsi})\|_{\bL^{\rq}(\Omega)} 
\,\leq\, |\Omega|^{(\rp-2)/(2\rp)}\,\|\partial_t\,\bbeta_{\bu}\|_{0,\Omega} + \Big( \alpha\,|\Omega|^{(\rp-2)/\rp} + \tF\,c_\rp\,\big( \|\bu\|_{\bM} + \|\bu_h\|_{\bM} \big)^{\rp-2}\, \Big)\, \|\be_{\bu}\|_{\bM}.
\end{equation*}
Then, taking square in the above inequality, integrating from $0$ to $t\in (0,T]$, recalling that both $\|\bu\|_{\L^\infty(0,T;\bM)}$ and $\|\bu_h\|_{\L^\infty(0,T;\bM)}$ are bounded by data (cf. \eqref{eq:stability-bound-u-Linf-Lp}, \eqref{eq:discrete-stability-bound-u-Linf-Lp}), and employing \eqref{eq:new-error-bound-3}, we deduce that
\begin{equation}\label{eq:error-bound-3}
\begin{array}{l}
\ds \int^t_0 \|\bdiv(\bbeta_{\bsi})\|^{2}_{\bL^{\rq}(\Omega)}\,ds
\,\leq\, 2\,|\Omega|^{(\rp-2)/\rp}\,\int^t_0 \|\partial_t\,\bbeta_{\bu}\|^2_{0,\Omega}\,ds + C_3\,\Bigg( \|\bdelta_{\bu}(t)\|^{2}_{0,\Omega} + \int^t_0 \|\bdelta_{\bsi}\|^2_{0,\Omega}\,ds \\ [3ex] 
\ds\quad\,+\,\,\int^t_0\Big( \|\bdelta_{\bu}\|^{\rp}_{\bM} + \|\bdelta_{\bu}\|^{\rq}_{\bM} + \|\bdelta_{\bu}\|^2_{\bM} \Big)\,ds
+ \|\bdelta_{\bsi}(0)\|^2_{0,\Omega} + \|\bdelta_{\bu}(0)\|^{\rp}_{\bM} + \|\bdelta_{\bu}(0)\|^{\rq}_{\bM} + \|\bdelta_{\bu}(0)\|^2_{\bM} \Bigg),
\end{array}
\end{equation}
with $C_3 > 0$ depending only on $|\Omega|, \nu, \alpha, \tF, \wt{\beta}$, and data.
Then, combining \eqref{eq:new-error-bound-2} and \eqref{eq:error-bound-3}, and employing \eqref{eq:deviatoric-inequality}, we get
\begin{equation}\label{eq:error-bound-4}
\begin{array}{l}
\ds \int^t_0 \|\bbeta_{\bsi}\|^2_{\bbX}\,ds
\,\leq\, \frac{2\,|\Omega|^{(\rp-2)/\rp}}{\wt{C}_{\rd}}\,\int^t_0 \|\partial_t\,\bbeta_{\bu}\|^2_{0,\Omega}\,ds + \frac{(C_2 + C_3)}{\wt{C}_{\rd}}\,\Bigg( \|\bdelta_{\bu}(t)\|^{2}_{0,\Omega} + \int^t_0 \|\bdelta_{\bsi}\|^2_{0,\Omega}\,ds \\ [3ex] 
\ds\quad\,+\,\,\int^t_0\Big( \|\bdelta_{\bu}\|^{\rp}_{\bM} + \|\bdelta_{\bu}\|^{\rq}_{\bM} + \|\bdelta_{\bu}\|^2_{\bM} \Big)\,ds
+ \|\bdelta_{\bsi}(0)\|^2_{0,\Omega} + \|\bdelta_{\bu}(0)\|^{\rp}_{\bM} + \|\bdelta_{\bu}(0)\|^{\rq}_{\bM} + \|\bdelta_{\bu}(0)\|^2_{\bM} \Bigg),
\end{array}
\end{equation}
where $\wt{C}_{\rd} := \min\big\{1,C_{\rd}\big\}/2$.
Next, in order to bound the first term in the right-hand side of \eqref{eq:error-bound-4}, we differentiate in time the first equation of \eqref{eq:error-system}, and choose $(\btau_h, \bv_h) = (\bbeta_{\bsi}, \partial_t\,\bbeta_{\bu})$ in \eqref{eq:error-system}, to find that
\begin{equation*}
\begin{array}{l}
\ds \frac{1}{2}\,\partial_t\,\Big( \frac{1}{\nu}\,\|\bbeta^\rd_{\bsi}\|^2_{0,\Omega} + \alpha\,\|\bbeta_{\bu}\|^2_{0,\Omega} \Big) + \|\partial_t\,\bbeta_{\bu}\|^2_{0,\Omega} \\ [2ex] 
\ds\quad\,=\,-\,[\cA(\partial_t\,\bdelta_{\bsi}),\bbeta_{\bsi}] - [\cC(\bu) - \cC(\bu_h),\partial_t\,\bbeta_{\bu}] - (\partial_t\,\bdelta_{\bu},\partial_t\,\bbeta_{\bu})_{\Omega} + [\cB(\bdelta_{\bsi}),\partial_t\,\bbeta_{\bu}] - [\cB(\bbeta_{\bsi}),\partial_t\,\bdelta_{\bu}],
\end{array}
\end{equation*}
where the last three terms are zero thanks to the projection properties \eqref{eq:Pk-projection-property}--\eqref{eq:Pik-interpolator}. 
%Thus, assuming now $\bu\in \bL^{2\rp}(\Omega)$, with $\rp\in [3,4]$, proceeding as in \eqref{eq:properties-operator-C-2}, using \cite[Proposition~5.3]{gm1975}, \eqref{eq:ab-m-power-inequality}, Cauchy--Schwarz and Young's inequalities, and the fact that $\|\bu_h\|_{\bL^{2\rp}(\Omega)} \leq c h^{-n/(2\rp)}\,\|\bu_h\|_{\bM}$, which follows from the global inverse inequality  \cite[Corollary~1.141]{Ern-Guermond}, we obtain
Thus, employing Cauchy--Schwarz, H\"older and Young's inequalities, the continuity bound of $\cC$ (cf. \eqref{eq:properties-operator-C-2}), \eqref{eq:ab-m-power-inequality}, and the fact that $\|\partial_t\,\bbeta_{\bu}\|_{\bM} \leq c h^{-n(\rp - 2)/(2\rp)}\,\|\partial_t\,\bbeta_{\bu}\|_{0,\Omega}$, with $\bbeta_{\bu}\in \bM_h$, which follows from the global inverse inequality  \cite[Corollary~1.141]{Ern-Guermond}, we obtain
\begin{equation*}
\begin{array}{l}
\ds \frac{1}{2}\,\partial_t\,\Big( \frac{1}{\nu}\,\|\bbeta^\rd_{\bsi}\|^2_{0,\Omega} + \alpha\,\|\bbeta_{\bu}\|^2_{0,\Omega} \Big) + \|\partial_t\,\bbeta_{\bu}\|^2_{0,\Omega} \\ [2ex]
\ds\quad\,\leq\, \frac{1}{\nu}\,\|\partial_t\,\bdelta_{\bsi}\|_{0,\Omega}\,\|\bbeta_{\bsi}\|_{0,\Omega} + \|\cC(\bu) - \cC(\bu_h)\|_{\bM'}\,\|\partial_t\,\bbeta_{\bu}\|_{\bM}  \\ [2ex] 
\ds\quad\,\leq\, \frac{1}{\nu}\,\|\partial_t\,\bdelta_{\bsi}\|_{0,\Omega}\,\|\bbeta_{\bsi}\|_{0,\Omega} +  \wt{C}_4\,h^{-n(\rp - 2)/(2\rp)}\,\Big( 1 + \|\bu\|^{\rp-2}_{\bM} + \|\bu_h\|^{\rp-2}_{\bM} \Big)\,\|\be_{\bu}\|_{\bM}\,\|\partial_t\bbeta_{\bu}\|_{0,\Omega}  \\ [2ex]
\ds\quad\,\leq\, C_4\,h^{-n(\rp-2)/\rp}\,\Big( 1 + \|\bu\|^{2(\rp-2)}_{\bM} + \|\bu_h\|^{2(\rp-2)}_{\bM} \Big)\,\Big( \|\partial_t\,\bdelta_{\bsi}\|^2_{0,\Omega} + \|\be_{\bu}\|^2_{\bM} \Big) + \frac{1}{2}\,\|\partial_t\,\bbeta_{\bu}\|^2_{0,\Omega} + \frac{\delta}{2}\,\|\bbeta_{\bsi}\|^2_{\bbX},
\end{array}
\end{equation*}
with $C_4 > 0$ depending on $|\Omega|, \nu, \alpha, \tF$.
Thus, integrating from $0$ to $t\in (0,T]$, we find in particular that
\begin{equation*}
\begin{array}{l}
\ds \int^t_0 \|\partial_t\,\bbeta_{\bu}\|^2_{0,\Omega}\,ds
\,\leq\, 2\,C_4\,h^{-n(\rp-2)/\rp}\,\Big( 1 + \|\bu\|^{2(\rp-2)}_{\L^\infty(0,T;\bM)} + \|\bu_h\|^{2(\rp-2)}_{\L^\infty(0,T;\bM)} \Big) \\ [3ex] 
\ds\quad\,\times\,\,\Bigg( \|\bdelta_{\bu}(t)\|^{2}_{0,\Omega} + \int^t_0 \Big( \|\partial_t\,\bdelta_{\bsi}\|^2_{0,\Omega} + \|\bdelta_{\bsi}\|^2_{0,\Omega} + \|\bdelta_{\bu}\|^{\rp}_{\bM} + \|\bdelta_{\bu}\|^{\rq}_{\bM} + \|\bdelta_{\bu}\|^2_{\bM} \Big)\,ds \\ [3ex] 
\ds\quad\,+\,\,\|\bdelta_{\bsi}(0)\|^2_{0,\Omega} + \|\bdelta_{\bu}(0)\|^{\rp}_{\bM} + \|\bdelta_{\bu}(0)\|^{\rq}_{\bM} + \|\bdelta_{\bu}(0)\|^2_{\bM} \Bigg) + \delta \int^t_0 \|\bbeta_{\bsi}\|^2_{\bbX}\,ds,
\end{array}
\end{equation*}
which combined with \eqref{eq:error-bound-4}, the fact that both $\|\bu\|_{\L^\infty(0,T;\bM)}$ and $\|\bu_h\|_{\L^\infty(0,T;\bM)}$ are bounded by data (cf. \eqref{eq:stability-bound-u-Linf-Lp}, \eqref{eq:discrete-stability-bound-u-Linf-Lp}), and taking $\delta = \wt{C}_{\rd}/(4\,|\Omega|^{(\rp-2)/\rp})$, allow us to deduce
\begin{equation}\label{eq:error-bound-5}
\begin{array}{l}
\ds \int^t_0 \|\bbeta_{\bsi}\|^2_{\bbX}  
\,\leq\, C_5\,h^{-n(\rp-2)/\rp}\,\Bigg( \|\bdelta_{\bu}(t)\|^{2}_{0,\Omega} + \int^t_0 \Big( \|\partial_t\,\bdelta_{\bsi}\|^2_{0,\Omega} + \|\bdelta_{\bsi}\|^2_{0,\Omega} + \|\bdelta_{\bu}\|^{\rp}_{\bM} \Big)\,ds \\ [3ex]
\ds\quad\,+\,\,\int^t_0 \Big( \|\bdelta_{\bu}\|^{\rq}_{\bM} + \|\bdelta_{\bu}\|^2_{\bM} \Big)\,ds + \|\bdelta_{\bsi}(0)\|^2_{0,\Omega} + \|\bdelta_{\bu}(0)\|^{\rp}_{\bM} + \|\bdelta_{\bu}(0)\|^{\rq}_{\bM} + \|\bdelta_{\bu}(0)\|^2_{\bM} \Bigg),
\end{array}
\end{equation}
with $C_5 > 0$ depending only on $|\Omega|, \nu, \alpha, \tF, \wt{\beta}$, and data.
Thus, using the error decomposition \eqref{eq:error-decomposition} in combination with \eqref{eq:error-bound-5}, we obtain \eqref{eq:full-error-estimate} and complete the proof.
\end{proof}

\begin{rem}
The error bounds provided in Theorem~\ref{thm:error-estimates} can be
derived alternatively without the use of the discrete inf-sup
condition of $\cB$ (cf. \eqref{eq:discrete-inf-sup-operator-B}), but
in that case the expression in the right-hand side of
\eqref{eq:error-estimate}--\eqref{eq:full-error-estimate} will be more
complicated.  We also note that in the steady state
  case of \eqref{eq:Brinkman-Forchheimer-2} the error estimate
  \eqref{eq:full-error-estimate} does not include the term
  $h^{-n(\rp-2)/(2\rp)}$ because the global inverse inequality is not
  necessary to bound $\|\bdiv(\bbeta_{\bsi})\|_{\bL^\rq(\Omega)}$.
\end{rem}

Now, in order to obtain theoretical rates of convergence for the discrete scheme \eqref{eq:discrete-weak-Brinkman-Forchheimer}, we recall the approximation properties of the subspaces involved (see \cite[Section~5.5]{cgm2019-pp} and \cite[Section~4.2.1]{cgo2018-pp}).
Note that each one of them is named after the unknown to which it is applied later on.

\bigskip

\noindent $(\mathbf{AP}^{\bsi}_h)$ For each $l\in (0,k+1]$ and for each
$\btau\in \bbH^{l}(\Omega)\cap \bbX_0$ with $\bdiv\btau\in \bW^{l,\rq}(\Omega)$, there holds
\begin{equation*}
\|\btau - \bPi^k_h(\btau)\|_{\bbX} 
\,\leq\, C\,h^{l}\,\Big\{ \|\btau\|_{\bbH^{l}(\Omega)} + \|\bdiv\btau\|_{\bW^{l,\rq}(\Omega)} \Big\}.
\end{equation*}

\noindent $(\mathbf{AP}^{\bu}_h)$ For each $l\in [0,k+1]$ and for each $\bv\in \bW^{l,\rp}(\Omega)$, there holds
\begin{equation*}
\|\bv - \bP^k_h(\bv)\|_{\bM} 
\,\leq\, C\,h^{l}\,\|\bv\|_{\bW^{l,\rp}(\Omega)}.
\end{equation*}

It follows that there exist positive constants $C(\partial_t\,\bsi), C(\bsi)$ and $C(\bu)$, depending on the extra regularity assumptions for $\bsi$ and $\bu$, respectively, and whose explicit expressions are obtained from the right-hand side of the foregoing approximation properties, such that
\begin{equation*}
\|\partial_t\,\bdelta_{\bsi}\|_{\bbX} \,\leq\, C(\partial_t\,\bsi)\,h^{l}, \quad
\|\bdelta_{\bsi}\|_{\bbX} \,\leq\, C(\bsi)\,h^{l}, \qan
\|\bdelta_{\bu}\|_{\bM} \,\leq\, C(\bu)\,h^{l}.
\end{equation*}
Then, we establish the theoretical rate of convergence of the discrete scheme \eqref{eq:discrete-weak-Brinkman-Forchheimer}.
Notice that, at least sub-optimal rates of convergences of order $\cO(h^{l\,\rq/2})$ and $\cO(h^{l\,\rq/2 - n(\rp-2)/(2\rp)})$ are confirmed.
\begin{thm}\label{thm:rate-of-convergence}
Let $(\bsi(t),\bu(t))\in \L^{\infty}(0,T;\bbX_0)\times \W^{1,\infty}(0,T;\bL^2(\Omega))\cap \L^\infty(0,T;\bM)$ and $(\bsi_h(t), \bu_h(t))\in \L^{\infty}(0,T;\bbX_{0,h})\times \W^{1,\infty}(0,T;\bM_h)$ be the unique solutions of the continuous and semidiscrete problems \eqref{eq:weak-Brinkman-Forchheimer} and \eqref{eq:discrete-weak-Brinkman-Forchheimer}, respectively.
Assume further that there exists $l\in (0,k+1]$, such that  $\bsi\in \bbH^{l}(\Omega), \bdiv\bsi\in \bW^{l,\rq}(\Omega)$ and $\bu\in \bW^{l,\rp}(\Omega)$, with $\rp\in [3,4]$ and $1/\rp + 1/\rq = 1$.
Then, there exist $C_1(\bsi,\bu), C_2(\bsi,\bu) > 0$ depending only on
$C(\partial_t\,\bsi), C(\bsi), C(\bu), |\Omega|, \nu, \alpha, \tF, \wt{\beta}, \rp$, and data, such that
\begin{equation*}%\label{eq:rate-of-convergence-1}
\|\be^{\rd}_{\bsi}\|_{\L^2(0,T;\bbL^2(\Omega))} + \|\be_{\bu}\|_{\L^{2}(0,T;\bM)} + \|\be_{\bu}\|_{\L^\infty(0,T;\bL^2(\Omega))} 
\,\leq\, C_1(\bsi,\bu)\,\Big( h^{l\,\rq/2} + h^{l} + h^{l\,\rp/2} \Big)
\end{equation*}
and
\begin{equation*}%\label{eq:rate-of-convergence-2}
\|\be_{\bsi}\|_{\L^2(0,T;\bbX)} \,\leq\, C_2(\bsi,\bu)\,h^{-n(\rp-2)/(2\rp)}\,\Big( h^{l\,\rq/2} + h^{l} + h^{l\,\rp/2} \Big).
\end{equation*}
\end{thm}
\begin{proof}
It suffices to apply Theorem~\ref{thm:error-estimates} and the approximation properties of the discrete subspaces. 
We omit further details.
\end{proof}

%************************************************************************

\subsection{Computing other variables of interest}

In this section we introduce suitable approximations for other variables of interest, such that the pressure $p$ and the velocity gradient $\bGu := \nabla\bu$, both written in terms of the solution of the semidiscrete continuous-in-time problem \eqref{eq:discrete-weak-Brinkman-Forchheimer}.
In fact, observing that at the continuous level there hold
\begin{equation*}
p \,=\, -\frac{1}{n}\,\tr(\bsi) \qan
\bGu \,=\, \frac{1}{\nu}\,\bsi^{\rd},
\end{equation*}
provided the discrete solution $(\bsi_h,\bu_h)\in \bbX_{0,h}\times \bM_h$ of problem \eqref{eq:discrete-weak-Brinkman-Forchheimer}, we propose the following approximations for the aforementioned variables:
\begin{equation}\label{eq:discrete-ph-Gh}
p_h \,=\, -\frac{1}{n}\,\tr(\bsi_h) \qan
\bGu_h \,=\, \frac{1}{\nu}\,\bsi^{\rd}_h.
\end{equation}
Then, we define the corresponding errors $\be_p = p - p_h$ and $\be_{\bGu} = \bGu - \bGu_h$.
The following result, whose proof follows directly from Theorem~\ref{thm:rate-of-convergence}, establishes the corresponding approximation result for this post-processing procedure.
\begin{cor}\label{cor:rate-of-convergence-p-Gu}
Let the assumptions of Theorem~\ref{thm:rate-of-convergence} hold.
Let $p_h$ and $\bGu_h$ given by \eqref{eq:discrete-ph-Gh}.
Then, there exist $\wt{C}_1(\bsi,\bu), \wt{C}_2(\bsi,\bu) > 0$ depending only on $C(\partial_t\,\bsi), C(\bsi), C(\bu), |\Omega|, \nu, \alpha, \tF, \wt{\beta}, \rp$, and data, such that
\begin{equation*}
\|\be_{\bGu}\|_{\L^2(0,T;\bbL^2(\Omega))}
\,\leq\, \wt{C}_1(\bsi,\bu)\,\Big( h^{l\,\rq/2} + h^{l} + h^{l\,\rp/2} \Big)
\end{equation*}
and
\begin{equation*}
\|\be_p\|_{\L^2(0,T;\L^2(\Omega))} 
\,\leq\, \wt{C}_2(\bsi,\bu)\,h^{-n(\rp-2)/(2\rp)}\,\Big( h^{l\,\rq/2} + h^{l} + h^{l\,\rp/2} \Big).
\end{equation*}
\end{cor}

%************************************************************************
%************************************************************************

\section{Fully discrete approximation}\label{sec:fully-discrete-formulation}

In this section we introduce and analyze a fully discrete approximation of \eqref{eq:weak-Brinkman-Forchheimer} (cf. \eqref{eq:discrete-weak-Brinkman-Forchheimer}).
To that end, for the time discretization we employ the backward Euler method.
Let $\Delta\,t$ be the time step, $T = N\,\Delta\,t$, and let $t_n = n\,\Delta\,t$, $n=0,\dots,N$.
Let $d_{t}\,u^n := (\Delta\,t)^{-1}(u^n - u^{n-1})$ be the first order (backward) discrete time derivative, where $u^n := u(t_n)$.
Then the fully discrete method reads:
given $(\bsi^0_h,\bu^0_h) = (\bsi_h(0),\bu_h(0))$ satisfying \eqref{eq:discrete-initial-condition-problem}, find $(\bsi^n_h, \bu^n_h)\in \bbX_{0,h}\times \bM_h$, $n=1,\dots,N$, such that
\begin{equation}\label{eq:fully-discrete-weak-Brinkman-Forchheimer}
\begin{array}{rlll}
\ds [\cA(\bsi^n_h),\btau_h] + [\cB(\btau_h),\bu^n_h] & = & 0 & \forall\, \btau_h\in \bbX_{0,h}, \\ [3ex]
\ds (d_{t}\,\bu^n_h,\bv_h)_{\Omega} - [\cB(\bsi^n_h),\bv_h] + [\cC(\bu^n_h),\bv_h] & = & (\f^n,\bv_h)_{\Omega} & \forall\,\bv_h\in \bM_h.
\end{array}
\end{equation}
Moreover, given a separable Banach space $\V$ endowed with the norm $\|\cdot\|_{\V}$, we introduce the discrete-in-time norms
\begin{equation*}
\|u\|^{\rp}_{\ell^{\rp}(0,T;\V)} \,:=\, \Delta\,t\,\sum^N_{n=1} \|u^n\|^{\rp}_{\V},\quad
\|u\|_{\ell^\infty(0,T;\V)} \,:=\, \max_{0\leq n \leq N}\,\|u^n\|_{\V}.
\end{equation*}

Next, we state the main results for method \eqref{eq:fully-discrete-weak-Brinkman-Forchheimer}.
\begin{thm}\label{thm:bound-stability-fully-discrete}
  For each $(\bsi^0_h,\bu^0_h) = (\bsi_h(0),\bu_h(0))$ satisfying
  \eqref{eq:discrete-initial-condition-problem} and 
$\f^n\in \bL^2(\Omega)$, $n=1,\dots,N$, there exists a unique
solution $(\bsi^n_h, \bu^n_h)\in \bbX_{0,h}\times \bM_h$
to \eqref{eq:fully-discrete-weak-Brinkman-Forchheimer}.
Moreover, assuming sufficient regularity of the data, there exists a
constant $C_{\hat\cF} > 0$ depending only on $|\Omega|, \nu, \alpha,
\tF, \wt{\beta}$, such that
\begin{equation}\label{eq:stability-bound-fd}
\|\bsi_h\|_{\ell^{2}(0,T;\bbX)} + \|\bu_h\|_{\ell^{2}(0,T;\bM)} + \|\bu_h\|_{\ell^\infty(0,T;\bL^2(\Omega))}
+ \Delta\,t\,\|d_{t}\,\bu_h\|_{\ell^2(0,T;\bL^2(\Omega))} 
\,\leq\, C_{\hat\cF}\,\hat\cF(\f,\bsi^0_h,\bu^0_h)
\end{equation}
with
\begin{equation*}
\begin{array}{l}
\ds \hat\cF(\f,\bsi^0_h,\bu^0_h) 
\,:=\, \|\f\|^{\rp-1}_{\ell^{2(\rp-1)}(0,T;\bL^2(\Omega))}
+ \|\f\|_{\ell^2(0,T;\bL^2(\Omega))} \\ [2ex] 
\ds\quad\,+\,\,\|(\bsi^0_h)^\rd\|_{0,\Omega} + \|\bu^0_h\|^{\rp/2}_{\bM} + \|\bu^0_h\|^{\rp-1}_{0,\Omega}
+ \|\bu^0_h\|_{0,\Omega}.
\end{array}
\end{equation*}
Moreover, with the constant $C_{\cG}$ from \eqref{eq:constant-CG-definition}, there holds
\begin{equation}\label{eq:stability-bound-u-linf-lp-fd}
\|\bu_h\|_{\ell^{\infty}(0,T;\bM)} 
\,\leq\, C_{\cG}\,\hat\cG(\f,\bsi^0_h,\bu^0_h)
\end{equation}
with
\begin{equation*}
\hat\cG(\f,\bsi^0_h,\bu^0_h) 
\,:=\, \|\f\|^{2/\rp}_{\ell^2(0,T;\bL^2(\Omega))} + \|(\bsi^0_h)^\rd\|^{2/\rp}_{0,\Omega} + \|\bu^0_h\|_{\bM} + \|\bu^0_h\|^{2/\rp}_{0,\Omega}.
\end{equation*}
\end{thm}
\begin{proof}
First, we note that at each time step the well-posedness of the fully discrete problem \eqref{eq:fully-discrete-weak-Brinkman-Forchheimer}, with $n=1,\dots,N$, follows from similar arguments to the proof of Lemma~\ref{lem:resolvent-system}.	

The proof of \eqref{eq:stability-bound-fd} and
\eqref{eq:stability-bound-u-linf-lp-fd} follows the structure of the proof of Theorem~\ref{thm:stability-bound}.	
We choose $(\btau_h,\bv_h) = (\bsi^n_h,\bu^n_h)$ in \eqref{eq:fully-discrete-weak-Brinkman-Forchheimer} and use the identity
\begin{equation}\label{eq:identity-discrete-time-derivative}
(d_{t}\,u^n_h,u^n_h)_{\Omega} 
\,=\, \frac{1}{2}\,d_{t}\,\|u^n_h\|^2_{0,\Omega} + \frac{1}{2}\,\Delta\,t\,\|d_{t}\,u^n_h\|^2_{0,\Omega},
\end{equation}
in combination with the coercivity bounds in \eqref{eq:properties-operator-A}--\eqref{eq:properties-operator-C-1} and Cauchy--Schwarz and Young's inequalities, to obtain the analogous estimate of \eqref{eq:preliminary-stability-bound}, that is
\begin{equation}\label{eq:preliminary-stability-bound-fd}
\frac{1}{2}\,d_{t}\,\|\bu^n_h\|^2_{0,\Omega} 
+ \frac{1}{2}\,\Delta\,t\,\|d_{t}\,\bu^n_h\|^2_{0,\Omega} 
+ \frac{1}{\nu}\,\|(\bsi^n_h)^\rd\|^2_{0,\Omega} 
+ \tF\,\|\bu^n_h\|^{\rp}_{\bM} + \alpha\,\|\bu^n_h\|^2_{0,\Omega}
\,\leq\, \frac{\delta}{2}\,\|\f^n\|^{2}_{0,\Omega} 
+ \frac{1}{2\,\delta}\,\|\bu^n_h\|^2_{0,\Omega}.
\end{equation}
In turn, from the discrete inf-sup condition of $\cB$ (cf. \eqref{eq:discrete-inf-sup-operator-B}), the first equation of \eqref{eq:fully-discrete-weak-Brinkman-Forchheimer}, and the continuity bound of $\cA$ (cf. \eqref{eq:properties-operator-A}), we deduce that
\begin{equation*}
\frac{\wt{\beta}^2\,\nu}{2}\,\|\bu^n_h\|^2_{\bM} 
\,\leq\, \frac{1}{2\,\nu}\,\|(\bsi^n_h)^\rd\|^2_{0,\Omega},
\end{equation*}
which combined with \eqref{eq:preliminary-stability-bound-fd} and choosing $\delta = 1/\alpha$,
yields
\begin{equation}\label{eq:stability-bound-1-fd}
\frac{1}{2}\,d_{t}\,\|\bu^n_h\|^2_{0,\Omega} 
+ \frac{1}{2}\,\Delta\,t\,\|d_{t}\,\bu^n_h\|^2_{0,\Omega} 
+ \frac{1}{2\,\nu}\,\|(\bsi^n_h)^\rd\|^2_{0,\Omega} 
+ \frac{\wt{\beta}\,\nu}{2}\,\|\bu^n_h\|^{2}_{\bM} 
+ \frac{\alpha}{2}\,\|\bu^n_h\|^2_{0,\Omega}
\,\leq\, \frac{1}{2\,\alpha}\,\|\f^n\|^{2}_{0,\Omega}. 
\end{equation}
Notice that, in order to simplify the stability bound, we have neglected the term $\|\bu^n_h\|^{\rp}_{\bM}$ in the left-hand side of \eqref{eq:preliminary-stability-bound-fd}.
Thus, summing up over the time index $n=1,\dots,N$ in \eqref{eq:stability-bound-1-fd} and multiplying by $\Delta\,t$, we get
\begin{equation}\label{eq:stability-bound-2-fd}
\begin{array}{l}
\ds \|\bu^N_h\|^2_{0,\Omega} 
+ (\Delta\,t)^2\,\sum^N_{n=1} \|d_{t}\,\bu^n_h\|^2_{0,\Omega} 
+ \Delta\,t\,\sum^N_{n=1} \Big( \|(\bsi^n_h)^\rd\|^2_{0,\Omega} + \|\bu^n_h\|^{2}_{\bM} + \|\bu^n_h\|^{2}_{0,\Omega} \Big)  \\ [3ex]
\ds\quad\,\leq\, C_1\,\Bigg( \Delta\,t\,\sum^N_{n=1} \|\f^n\|^{2}_{0,\Omega} + \|\bu^0_h\|^2_{0,\Omega} \Bigg),
\end{array}
\end{equation}
with $C_1>0$ depending only on $\nu, \alpha$, and $\wt{\beta}$.
	
On the other hand, from the second equation of \eqref{eq:fully-discrete-weak-Brinkman-Forchheimer}, we have the identity
\begin{equation*}
[\cB(\bsi^n_h),\bv_h] 
\,=\, -\,(\f^n,\bv_h)_{\Omega} + [\cC(\bu^n_h),\bv_h] + (d_{t}\,\bu^n_h,\bv_h)_{\Omega} \quad \forall\,\bv_h\in \bM_h.
\end{equation*}
Then, employing the property $\bdiv(\bbX_0) = \bM'$ in combination with \eqref{eq:properties-operator-C-1} and using similar arguments to \eqref{eq:stability-bound-2-fd}, we deduce the analogous estimate of \eqref{eq:bound-int-div-bsi-1}, that is
\begin{equation}\label{eq:bound-int-div-bsi-1-fd}
\begin{array}{l}
\ds \Delta\,t\,\sum^N_{n=1} \|\bdiv\bsi^n_h\|^2_{\bL^{\rq}(\Omega)} \\ [3ex]
\ds\quad\,\leq\, C_2\,\Bigg( \Delta\,t \sum^N_{n=1} \Big( \|\f^n\|^{2(\rp-1)}_{0,\Omega} + \|\f^n\|^2_{0,\Omega} \Big) 
+ \|\bu^0_h\|^{2(\rp-1)}_{0,\Omega} + \|\bu^0_h\|^2_{0,\Omega} 
+ \Delta\,t \sum^N_{n=1} \|d_{t}\,\bu^n_h\|^2_{0,\Omega} \Bigg),
\end{array}
\end{equation}
with $C_2>0$ depending on $|\Omega|, \nu, \tF, \alpha$, and $\wt{\beta}$.
Next, in order to bound the last term in \eqref{eq:bound-int-div-bsi-1-fd}, we apply some algebraic manipulation in \eqref{eq:fully-discrete-weak-Brinkman-Forchheimer}, choose $(\btau_h, \bv_h) = (\bsi^n_h, d_{t}\,\bu^n_h)$, and employ \eqref{eq:identity-discrete-time-derivative}, Cauchy--Schwarz and Young's inequalities, to obtain
\begin{equation}\label{eq:aux-1-fully-discrete}
\begin{array}{l}
\ds \frac{1}{2\,\nu}\,d_{t}\,\|(\bsi^n_h)^\rd\|^2_{0,\Omega} 
+ \tF\,(\Delta\,t)^{-1}\,(|\bu^n_h|^{\rp-2}\bu^n_h,\bu^n_h - \bu^{n-1}_h)_{\Omega}
+ \frac{\alpha}{2}\,d_{t}\,\|\bu^n_h\|^2_{0,\Omega}  \\ [3ex]
\ds\quad\,+\,\,\frac{\Delta\,t}{2}\,\Big( \frac{1}{\nu}\,\|d_{t}\,(\bsi^n_h)^\rd\|^2_{0,\Omega} + \alpha\,\|d_{t}\,\bu^n_h\|^2_{0,\Omega} \Big)
+ \|d_{t}\,\bu^n_h\|^2_{0,\Omega} 
\,\leq\, \frac{1}{2}\,\|\f^n\|^2_{0,\Omega} 
+ \frac{1}{2}\,\|d_{t}\,\bu^n_h\|^2_{0,\Omega}.
\end{array}
\end{equation}
In turn, employing H\"older and Young's inequalities, we have
$$
\big|(|\bu^n_h|^{\rp-2}\bu^n_h,\bu^{n-1}_h)_{\Omega}\big|
\leq \frac{\rp-1}{\rp}\|\bu^n_h\|^{\rp}_\bM + \frac{1}{\rp}\,\|\bu^{n-1}_h\|^{\rp}_{\bM},
$$
which implies
\begin{equation}\label{eq:aux-2-fully-discrete}
(\Delta\,t)^{-1}\,(|\bu^n_h|^{\rp-2}\bu^n_h,\bu^n_h - \bu^{n-1}_h)_{\Omega} 
\,\geq\, \frac{(\Delta\,t)^{-1}}{\rp}\,\Big(\|\bu^n_h\|^{\rp}_{\bM} - \|\bu^{n-1}_h\|^{\rp}_{\bM} \Big) 
\,=\, \frac{1}{\rp}\,d_{t}\,\|\bu^n_h\|^{\rp}_{\bM}.
\end{equation}
Thus, combining \eqref{eq:aux-1-fully-discrete} with \eqref{eq:aux-2-fully-discrete}, summing up over the time index $n=1,\dots,N$ and multiplying by $\Delta\,t$, we get 
\begin{equation}\label{eq:bound-ut-int-ut-fd}
\|\bu^N_h\|^{\rp}_{\bM} 
+ \Delta\,t\,\sum^N_{n=1} \|d_{t}\,\bu^n_h\|^2_{0,\Omega} 
\,\leq\, C_3\,\Bigg( \Delta\,t\,\sum^N_{n=1}\, \|\f^n\|^2_{0,\Omega} 
+ \|(\bsi^0_h)^\rd\|^2_{0,\Omega} + \|\bu^0_h\|^{\rp}_{\bM} + \|\bu^0_h\|^{2}_{0,\Omega} \Bigg),
\end{equation}
with $C_3 := \max\big\{ 1,1/\nu,2\,\tF/\rp,\alpha \big\}$ as in \eqref{eq:bound-ut-int-ut}. 
Then, combining \eqref{eq:bound-ut-int-ut-fd} with \eqref{eq:bound-int-div-bsi-1-fd} yields
\begin{equation*}%\label{eq:bound-int-div-bsi-2}
\begin{array}{l}
\ds \Delta\,t\,\sum^N_{n=1} \|\bdiv\bsi^n_h\|^2_{\bL^{\rq}(\Omega)} \\ [3ex] 
\ds\quad\,\leq\, C_4\,\Bigg( \Delta\,t\,\sum^N_{n=1} \Big( \|\f^n\|^{2(\rp-1)}_{0,\Omega} + \|\f^n\|^2_{0,\Omega} \Big) + \|(\bsi^0_h)^\rd\|^2_{0,\Omega} + \|\bu^0_h\|^{\rp}_{\bM} + \|\bu^0_h\|^{2(\rp-1)}_{0,\Omega} + \|\bu^0_h\|^{2}_{0,\Omega} \Bigg),
\end{array}
\end{equation*}
which, combined with \eqref{eq:stability-bound-2-fd} and \eqref{eq:deviatoric-inequality},
implies \eqref{eq:stability-bound-fd}.
In addition, \eqref{eq:bound-ut-int-ut-fd} implies \eqref{eq:stability-bound-u-linf-lp-fd} with the same constant $C_{\cG}$ as in \eqref{eq:constant-CG-definition}, concluding the proof.
\end{proof}

Now, we proceed with establishing rates of convergence for the fully discrete scheme \eqref{eq:fully-discrete-weak-Brinkman-Forchheimer}.
To that end, we subtract the fully discrete problem \eqref{eq:fully-discrete-weak-Brinkman-Forchheimer} from the continuous counterparts \eqref{eq:weak-Brinkman-Forchheimer} at each time step $n = 1,\dots, N$, to obtain the following error system:
\begin{equation}\label{eq:error-system-fully-discrete}
\begin{array}{rll}
[\cA(\bsi^n - \bsi^n_h),\btau_h] + [\cB(\btau_h),\bu^n - \bu^n_h] & = & 0, \\ [3ex]
(d_{t}\,(\bu^n - \bu^n_h),\bv_h)_{\Omega} - [\cB(\bsi^n - \bsi^n_h),\bv_h] + [\cC(\bu^n) - \cC(\bu^n_h),\bv_h] & = & (r_n(\bu),\bv_h)_{\Omega},
\end{array}
\end{equation}
for all $(\btau_h, \bv_h)\in \bbX_{0,h}\times \bM_h$, where $r_n$ denotes the difference between the time derivative and its discrete analog, that is
\begin{equation*}
r_n(\bu) \,=\, d_{t}\,\bu^n - \partial_t\,\bu(t_n).
\end{equation*}
In addition, we recall from \cite[Lemma~4]{byz2015} that for sufficiently smooth $\bu$, there holds
\begin{equation}\label{eq:bound-second-time-derivative}
\Delta\,t\,\sum^N_{n=1} \|r_n(\bu)\|^2_{0,\Omega} 
\,\leq\, C(\partial_{tt}\,\bu)\,(\Delta\,t)^2,\quad\mbox{with}\quad C(\partial_{tt}\,\bu) \,:=\, C\,\|\partial_{tt}\,\bu\|^2_{\L^2(0,T;\bL^2(\Omega))}.
\end{equation}
Then, the proof of the theoretical rate of convergence of the fully discrete scheme \eqref{eq:fully-discrete-weak-Brinkman-Forchheimer} follows the structure of the proof of Theorems~\ref{thm:error-estimates} and \ref{thm:rate-of-convergence}, using discrete-in-time arguments as in the proof of Theorem~\ref{thm:bound-stability-fully-discrete} and the estimate \eqref{eq:bound-second-time-derivative}.
\begin{thm}\label{thm:rate-of-convergence-fully-discrete}
Let the assumptions of Theorem~\ref{thm:rate-of-convergence} hold.
Then for the solution of the fully discrete problem \eqref{eq:fully-discrete-weak-Brinkman-Forchheimer} there exist $\wh{C}_1(\bsi,\bu), \wh{C}_2(\bsi,\bu) > 0$ depending only on
$C(\partial_{tt}\,\bu), C(\bsi), C(\bu), |\Omega|, \nu, \alpha, \tF$, $\wt{\beta}, \rp$, and data, such that
\begin{equation*}
\begin{array}{l}
\ds \|\be^{\rd}_{\bsi}\|_{\ell^2(0,T;\bbL^2(\Omega))} + \|\be_{\bu}\|_{\ell^{2}(0,T;\bM)} + \|\be_{\bu}\|_{\ell^\infty(0,T;\bL^2(\Omega))} 
+ \Delta\,t\,\|d_{t}\,\be_{\bu}\|_{\ell^2(0,T;\bL^2(\Omega))} \\ [2ex]
\ds\quad\,\leq\, \wh{C}_1(\bsi,\bu)\,\Big( h^{l\,\rq/2} + h^{l} + h^{l\,\rp/2} + \Delta\,t \Big)
\end{array}
\end{equation*}
and
\begin{equation*}
\|\be_{\bsi}\|_{\ell^2(0,T;\bbX)} 
\,\leq\, \wh{C}_2(\bsi,\bu)\,h^{-n(\rp-2)/(2\rp)}\,\Big( h^{l\,\rq/2} + h^{l} + h^{l\,\rp/2} + \Delta\,t \Big).
\end{equation*}
\end{thm}

Consequently, and similarly to Corollary~\ref{cor:rate-of-convergence-p-Gu} we establishes the corresponding approximation result for the post-processed variables.
\begin{cor}\label{cor:rate-of-convergence-p-Gu-fd}
Let the assumptions of Theorem~\ref{thm:rate-of-convergence} hold.
Let $p_h$ and $\bGu_h$ be given by \eqref{eq:discrete-ph-Gh} at each time step $n = 1,\dots, N$.
Then, there exist $\ov{C}_1(\bsi,\bu), \ov{C}_2(\bsi,\bu) > 0$ depending only on $C(\partial_{tt}\,\bu), C(\bsi), C(\bu)$, $|\Omega|, \nu, \alpha, \tF, \wt{\beta}, \rp$, and data, such that
\begin{equation*}
\|\be_{\bGu}\|_{\ell^2(0,T;\bbL^2(\Omega))}
\,\leq\, \ov{C}_1(\bsi,\bu)\,\Big( h^{l\,\rq/2} + h^{l} + h^{l\,\rp/2} + \Delta\,t \Big)
\end{equation*}
and
\begin{equation*}
\|\be_p\|_{\ell^2(0,T;\L^2(\Omega))} 
\,\leq\, \ov{C}_2(\bsi,\bu)\,h^{-n(\rp-2)/(2\rp)}\,\Big( h^{l\,\rq/2} + h^{l} + h^{l\,\rp/2} + \Delta\,t \Big).
\end{equation*}
\end{cor}

%************************************************************************
%************************************************************************

\section{Numerical results}\label{sec:numerical-results}

In this section we present numerical results that illustrate the behavior of the fully discrete
method \eqref{eq:fully-discrete-weak-Brinkman-Forchheimer}.
Our implementation is based on a {\tt FreeFem++} code \cite{freefem}, in conjunction with the direct
linear solver {\tt UMFPACK} \cite{umfpack}.
For spatial discretization we use the Raviart--Thomas spaces $\bbRT_k-\bP_k$ with $k = 0,1$.
We handle the nonlinearity using Newton's method. The iterations are terminated once the
relative difference between the coefficient vectors for two consecutive iterates is sufficiently small, i.e.,
\begin{equation*}
\frac{\|\coeff^{m+1} - \coeff^m\|_{\ell^2}}{\|\coeff^{m+1}\|_{\ell^2}} \,\leq\, \tol,
\end{equation*}
where $\|\cdot\|_{\ell^2}$ is the standard $\ell^2$-norm in $\mathbb{R}^{\DOF}$, with $\DOF$ denoting the total number of degrees of
freedom defining the finite element subspaces $\bbX_h$ and $\bM_h$, and $\tol$ is a fixed tolerance chosen as $\tol=1\textrm{E}-06$.

The examples considered in this section are described next.
In all of them, and for the sake of simplicity, we choose $\nu = 1, \rp = 3$, and $\rq = 3/2$.
In addition, the condition $(\tr(\bsi_h),1)_{\Omega} = 0$ is implemented using a scalar Lagrange multiplier (adding one row and one column to the matrix system that solves \eqref{eq:discrete-weak-Brinkman-Forchheimer} for $\bsi_h$ and $\bu_h$). 

Examples~1 and 2 are used to corroborate the rate of convergence in two and three dimensional domains, respectively. 
The total simulation time for these examples is $T=0.01$\,s and the time step is $\Delta\,t = 10^{-3}$\,s.
The time step is sufficiently small, so that the time discretization error does not affect the convergence rates.
On the other hand, Examples~3 and 4 are used to analyze the behavior of the method
when different Darcy and Forchheimer coefficients are considered in different scenarios.
For these cases, the total simulation time and the time step are considered as $T = 1$\,s and $\Delta\,t = 10^{-2}$\,s, respectively.

%************************************************************************

\subsection*{Example 1: Two-dimensional smooth exact solution}

In this test we study the convergence for the space discretization using an analytical solution.
The domain is the square $\Omega = (0,1)^2$.
We consider $\alpha = 1, \tF = 10$, and the data $\f$ is adjusted so that the exact solution is given by the smooth functions
\begin{equation*}
\bu = \exp(t)
\left(\begin{array}{c}
\sin(\pi\,x)\cos(\pi\,y) \\ -\cos(\pi\,x)\sin(\pi\,y)
\end{array}\right),\quad
p = \exp(t)\,\cos(\pi\,x)\,\sin\left(\frac{\pi\,y}{2}\right).
\end{equation*}
The model problem is then complemented with the appropriate Dirichlet boundary condition and initial data.

In Figure~\ref{fig:example1} we display the solution obtained with the
mixed finite element method using a second order approximation and
$627,360\,\DOF$ at time $T = 0.01$.  Note that we are able to compute
not only the original unknowns, but also the pressure field
and velocity gradient through the formulas in \eqref{eq:discrete-ph-Gh}.
Tables \ref{tab:example1-k0} and \ref{tab:example1-k1} show the
convergence history for a sequence of quasi-uniform mesh refinements,
including the average number of Newton iterations. 
The results confirm that the sub-optimal spatial rates of convergence
$\cO(h^{3\,(k+1)/4})$ and $\cO(h^{3\,(k+1)/4 - 1/3})$ provided by
Theorem~\ref{thm:rate-of-convergence-fully-discrete} and
Corollary~\ref{cor:rate-of-convergence-p-Gu-fd} (see also Theorem~\ref{thm:rate-of-convergence} and
Corollary~\ref{cor:rate-of-convergence-p-Gu}) are attained for $k=0,1$.
Moreover, the numerical results suggest optimal rate of convergence
$\cO(h^{k+1})$. The Newton's method exhibits behavior independent of
the mesh size, converging in three or four iterations in all cases.

%%%%%%%%%%%    Figure1 - example1 - k=1  %%%%%%%%%%%
\begin{figure}[ht]
\begin{center}		
\includegraphics[width=5cm]{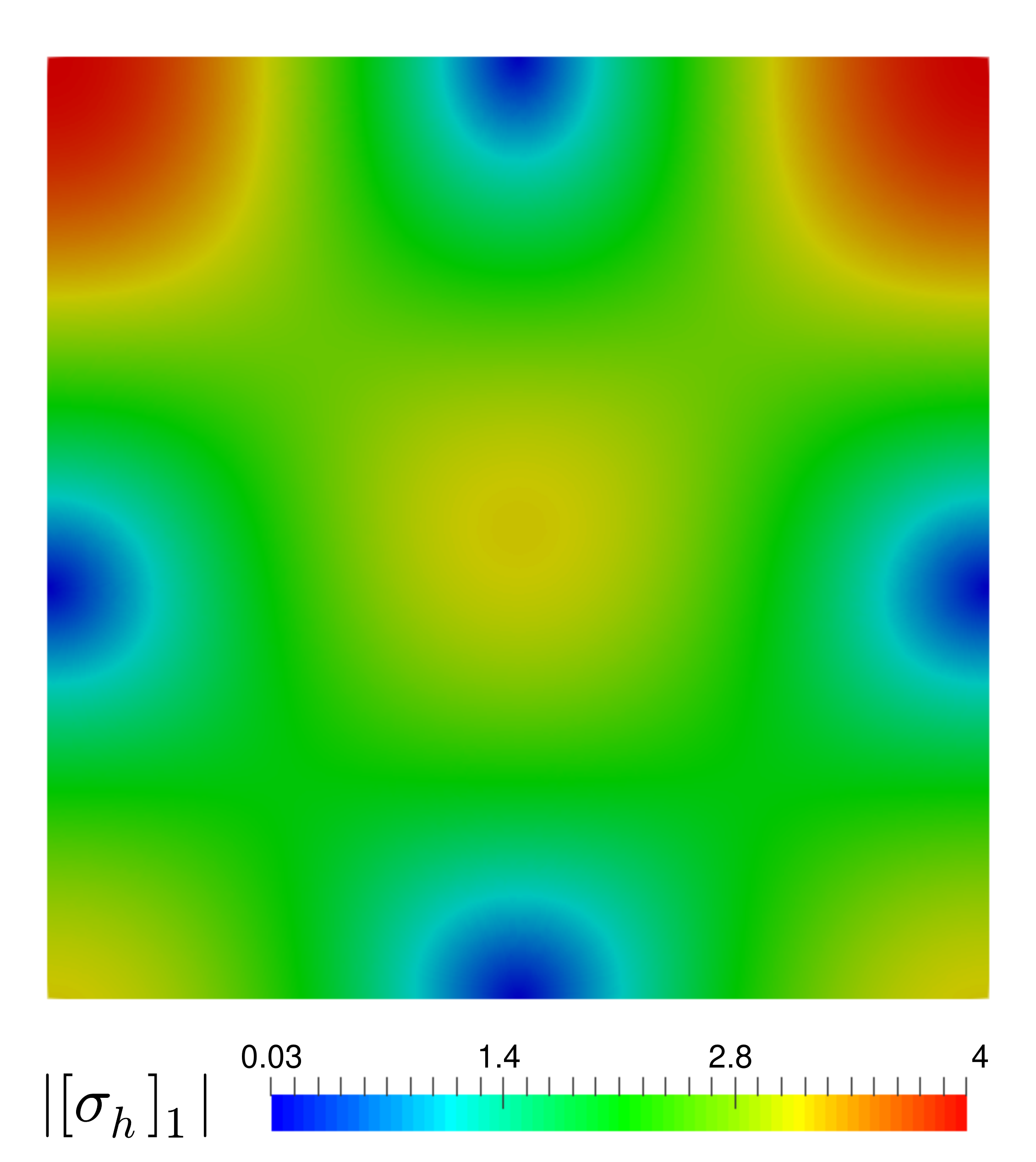}
\includegraphics[width=5cm]{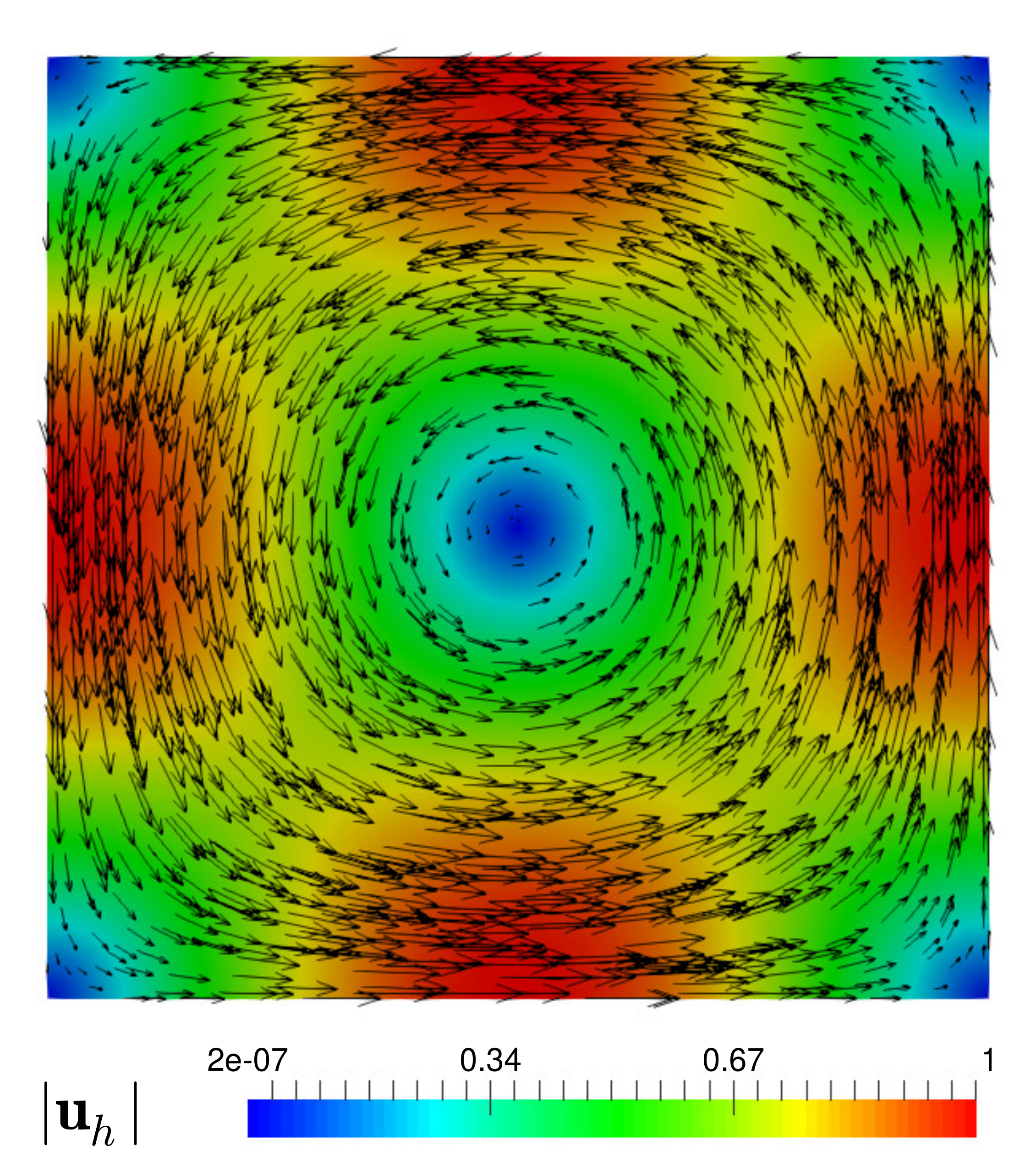}	
\includegraphics[width=5cm]{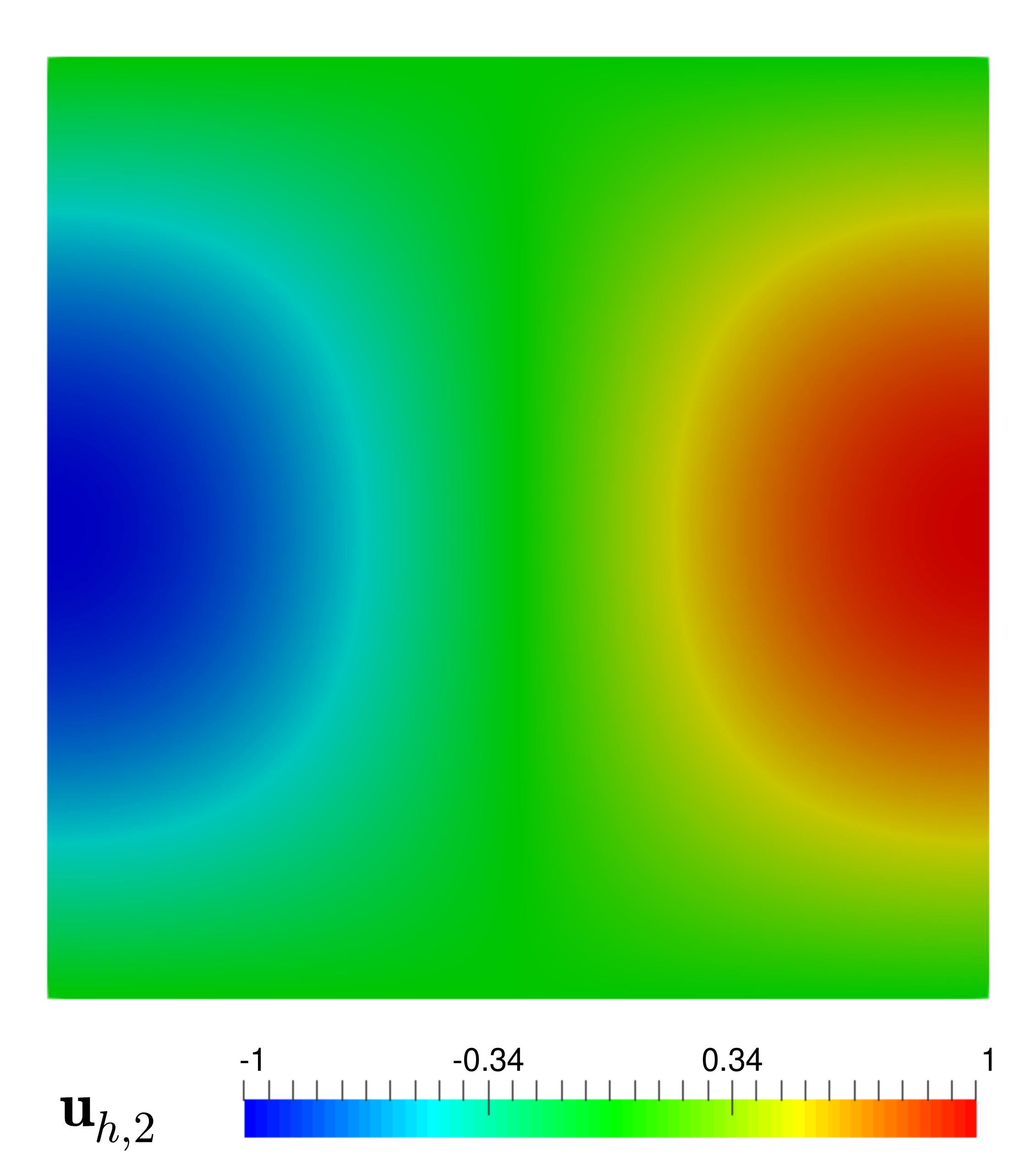}
	
\includegraphics[width=5cm]{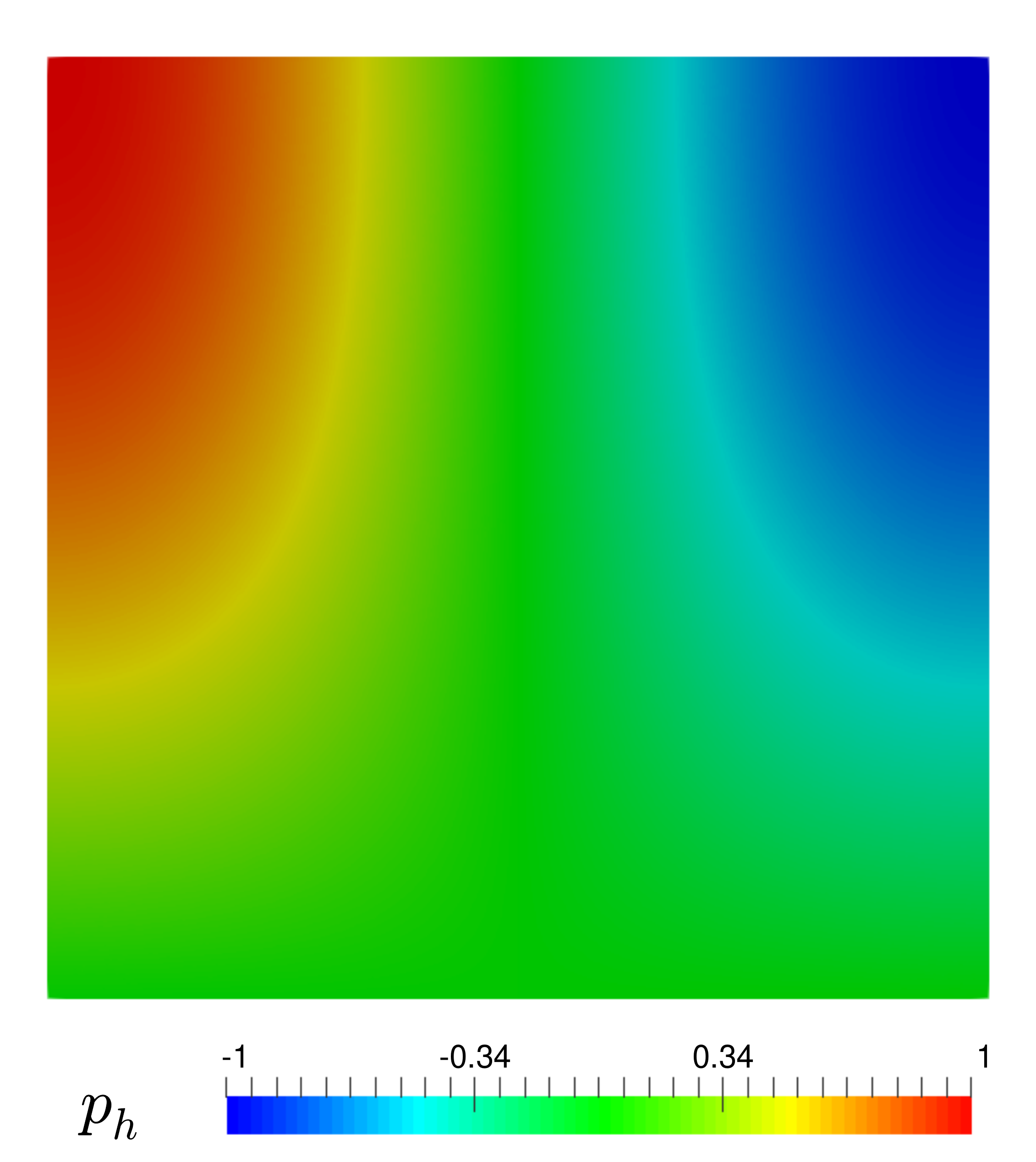}
\includegraphics[width=5cm]{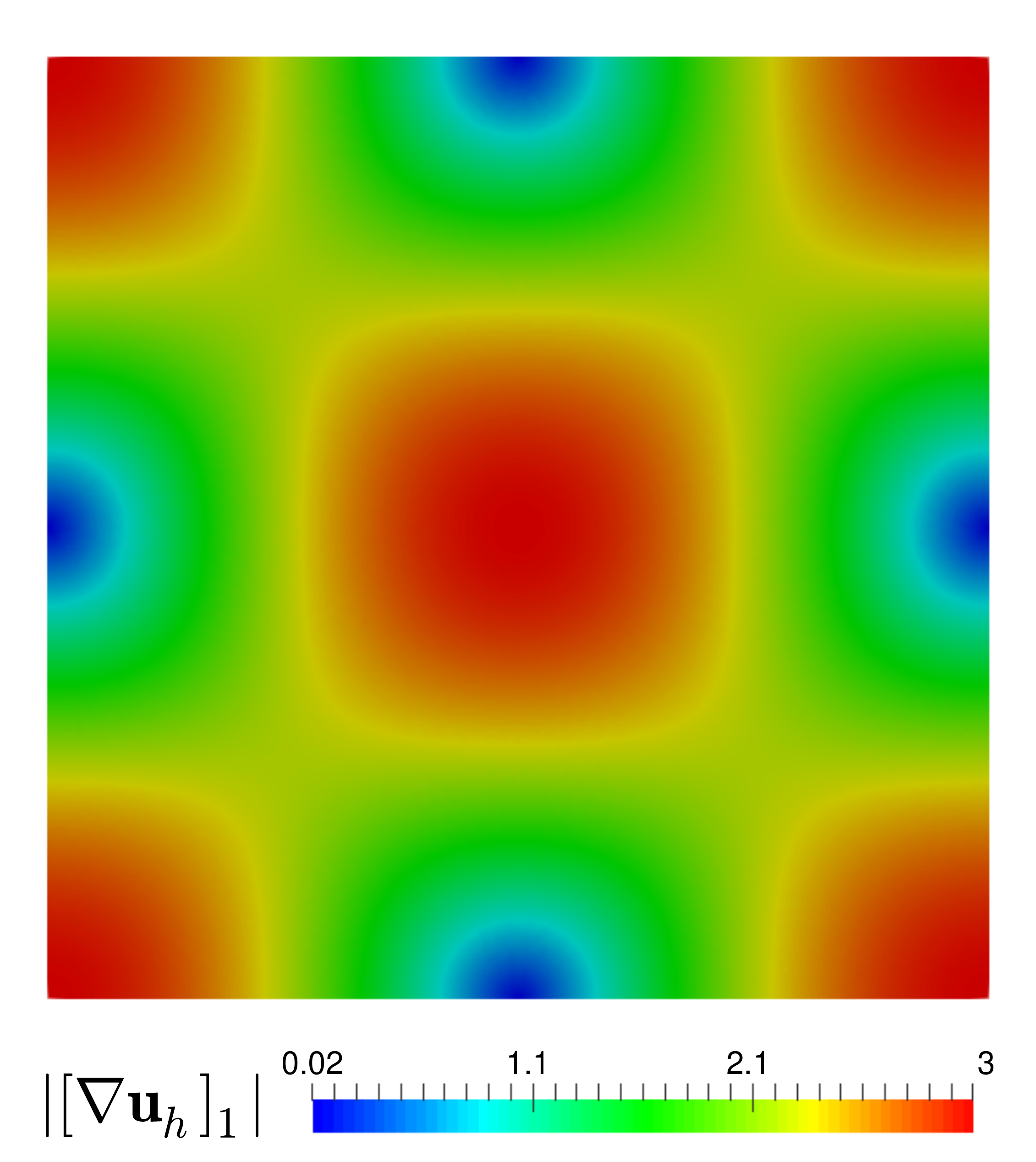}
\includegraphics[width=5cm]{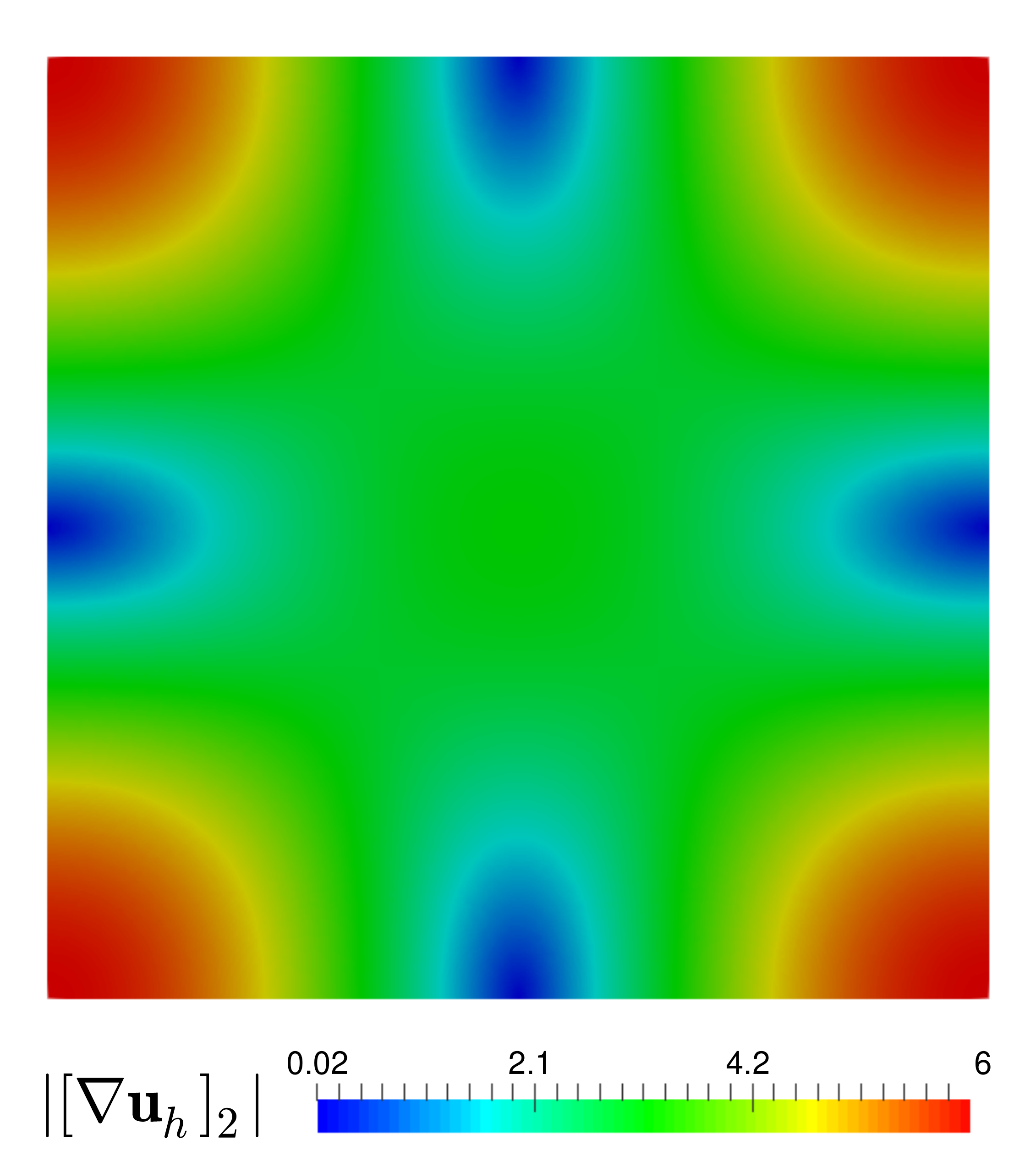}
\caption{Example 1, Computed magnitude of a component of the pseudostress tensor and
  the velocity, and a velocity component (top plots); pressure field and magnitude of the velocity gradient components (bottom plots).}\label{fig:example1}
\end{center}
\end{figure}

%%%%%%%%%%%    Table1 - example1 - k=0  %%%%%%%%%%%
\begin{table}[ht]
\begin{center}
		
\begin{tabular}{r|c||cc|cc|cc}
\hline
& & \multicolumn{2}{|c|}{$\|\be_{\bsi}\|_{\ell^2(0,T;\bbX)}$} & \multicolumn{2}{|c|}{$\|\be_{\bu}\|_{\ell^2(0,T;\bM)}$} & \multicolumn{2}{|c}{$\|\be_{\bu}\|_{\ell^{\infty}(0,T;\bL^2(\Omega))}$} \\ 
$\DOF$    & $h$ & $\error$ & $\rate$ & $\error$ & $\rate$ & $\error$ & $\rate$ \\  \hline\hline
196    & 0.3727 & 4.3E-01 &   --   & 2.3E-02 &   --   & 2.0E-01 &   --   \\ 
792    & 0.1964 & 1.8E-01 & 1.3822 & 9.9E-03 & 1.3081 & 8.7E-02 & 1.3077 \\ 
3084   & 0.0970 & 8.7E-02 & 1.0326 & 5.0E-03 & 0.9775 & 4.4E-02 & 0.9773 \\
12208  & 0.0478 & 4.1E-02 & 1.0721 & 2.4E-03 & 1.0245 & 2.1E-02 & 1.0184 \\
48626  & 0.0245 & 2.0E-02 & 1.0419 & 1.2E-03 & 1.0150 & 1.1E-02 & 1.0188 \\
196242 & 0.0128 & 1.0E-02 & 1.0826 & 6.1E-04 & 1.0761 & 5.3E-03 & 1.0755 \\ 
\hline 
\end{tabular}
		
\medskip
		
\begin{tabular}{cc|cc|c}
\hline
\multicolumn{2}{c}{$\|\be_p\|_{\ell^2(0,T;\L^2(\Omega))}$} & \multicolumn{2}{|c|}{$\|\be_{\bGu}\|_{\ell^2(0,T;\bbL^2(\Omega))}$} & \\ 
$\error$ & $\rate$ & $\error$ & $\rate$ & $\iter$ \\  \hline\hline
4.5E-02 &   --   & 7.2E-02 &  --    & 3 \\ 
1.8E-02 & 1.3903 & 3.4E-02 & 1.1696 & 3 \\ 
8.6E-03 & 1.0837 & 1.7E-02 & 0.9690 & 3 \\
3.5E-03 & 1.2611 & 8.6E-03 & 0.9710 & 3 \\
1.8E-03 & 1.0284 & 4.3E-03 & 1.0421 & 3 \\
8.5E-04 & 1.1326 & 2.1E-03 & 1.0753 & 3 \\ 
\hline 
\end{tabular}
		
\caption{{\sc Example 1}, Number of degrees of freedom, mesh sizes, errors, rates of convergences, and average number of Newton iterations for the mixed $\bbRT_0-\bP_0$ approximation of the Brinkman--Forchheimer model.}\label{tab:example1-k0}
\end{center}
\end{table}

%%%%%%%%%%%    Table2 - example1 - k=1  %%%%%%%%%%%
\begin{table}[ht!]
\begin{center}
		
\begin{tabular}{r|c||cc|cc|cc}
\hline
& & \multicolumn{2}{|c|}{$\|\be_{\bsi}\|_{\ell^2(0,T;\bbX)}$} & \multicolumn{2}{|c|}{$\|\be_{\bu}\|_{\ell^2(0,T;\bM)}$} & \multicolumn{2}{|c}{$\|\be_{\bu}\|_{\ell^{\infty}(0,T;\bL^2(\Omega))}$} \\ 
$\DOF$    & $h$ & $\error$ & $\rate$ & $\error$ & $\rate$ & $\error$ & $\rate$ \\  \hline\hline
608    & 0.3727 & 6.2E-01 &   --   & 5.1E-03 &   --   & 5.7E-02 &   --   \\ 
2496   & 0.1964 & 1.8E-01 & 1.8958 & 1.2E-03 & 2.2442 & 1.4E-02 & 2.2166 \\ 
9792   & 0.0970 & 5.1E-02 & 1.8033 & 3.0E-04 & 1.9672 & 3.5E-03 & 1.9686 \\
38912  & 0.0478 & 1.4E-02 & 1.8696 & 7.6E-05 & 1.9515 & 8.8E-04 & 1.9396 \\
155296 & 0.0245 & 3.5E-03 & 2.0277 & 1.9E-05 & 2.0669 & 2.2E-04 & 2.0693 \\
627360 & 0.0128 & 8.4E-04 & 2.1899 & 4.5E-06 & 2.1837 & 5.4E-05 & 2.1671 \\ 
\hline 
\end{tabular}
		
\medskip
		
\begin{tabular}{cc|cc|c}
\hline
\multicolumn{2}{c}{$\|\be_p\|_{\ell^2(0,T;\L^2(\Omega))}$} & \multicolumn{2}{|c|}{$\|\be_{\bGu}\|_{\ell^2(0,T;\bbL^2(\Omega))}$} & \\ 
$\error$ & $\rate$ & $\error$ & $\rate$ & $\iter$ \\  \hline\hline
6.3E-02 &   --   & 3.4E-02 &  --    & 3 \\ 
1.1E-02 & 2.6684 & 8.3E-03 & 2.2220 & 3 \\ 
1.8E-03 & 2.5820 & 2.0E-03 & 1.9999 & 3 \\
3.9E-04 & 2.1791 & 5.1E-04 & 1.9384 & 3 \\
6.4E-05 & 2.7237 & 1.3E-04 & 2.1013 & 3 \\
1.2E-05 & 2.5604 & 3.0E-05 & 2.2057 & 3 \\ 
\hline 
\end{tabular}
		
\caption{{\sc Example 1}, Number of degrees of freedom, mesh sizes, errors, rates of convergences, and average number of Newton iterations for the mixed $\bbRT_1-\bP^{\dc}_1$ approximation of the Brinkman--Forchheimer model.}\label{tab:example1-k1}
\end{center}
\end{table}

%************************************************************************

\subsection*{Example 2: Three-dimensional smooth exact solution}

In our second example, we consider the cube domain $\Omega = (0,1)^3$.
The solution is given by
\begin{equation*}
\bu = \exp(t)
\left(\begin{array}{c}
\sin(\pi\,x)\cos(\pi\,y)\cos(\pi\,z) \\ 
-2\,\cos(\pi\,x)\sin(\pi\,y)\cos(\pi\,z) \\
\cos(\pi\,x)\cos(\pi\,y)\sin(\pi\,z)
\end{array}\right),\quad
p = \exp(t)\,(x - 0.5)^3\sin(y+z).
\end{equation*}
Similarly to the first example, we consider the parameters $\alpha=1, \tF=10$, and the right-hand side function $\f$ is computed from \eqref{eq:Brinkman-Forchheimer-1} using the above solution.

The numerical solution at time $T = 0.01$ is shown in
Figure~\ref{fig:example2}. The convergence history for a set of
quasi-uniform mesh refinements using $k = 0$ is shown in
Table~\ref{tab:example2-k0}. Again, the mixed finite element method
converges optimally with order $\cO(h)$, which is better than the theoretical
suboptimal rates of convergence $\cO(h^{3/4})$ and $\cO(h^{1/4})$
provided by Theorem~\ref{thm:rate-of-convergence-fully-discrete} and
Corollary~\ref{cor:rate-of-convergence-p-Gu-fd} (see also Theorem~\ref{thm:rate-of-convergence} and
Corollary~\ref{cor:rate-of-convergence-p-Gu}).

%%%%%%%%%%%    Figure2 - example2 - k=0  %%%%%%%%%%%
\begin{figure}[ht]
\begin{center}	
\includegraphics[width=5cm]{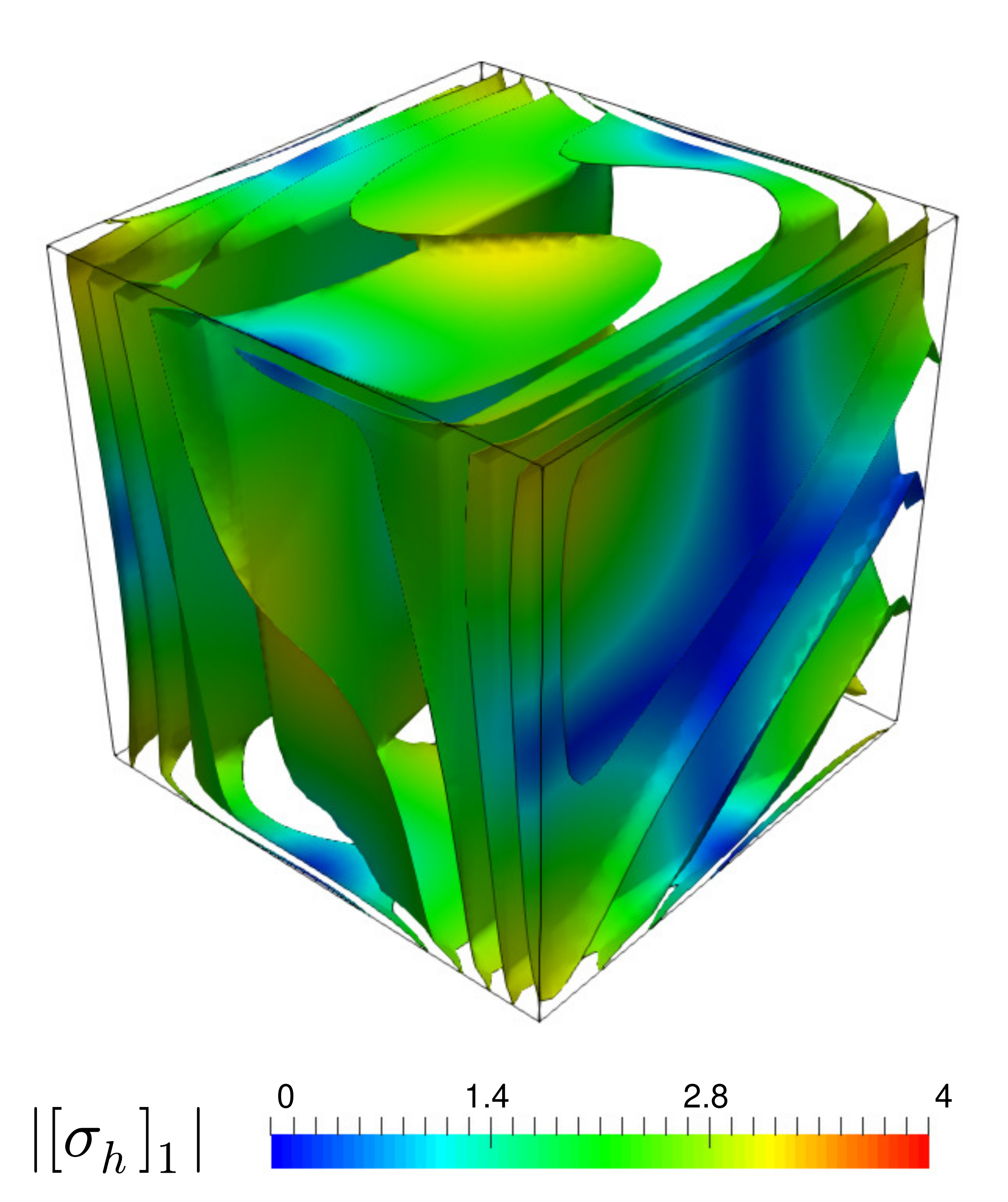}
\includegraphics[width=5cm]{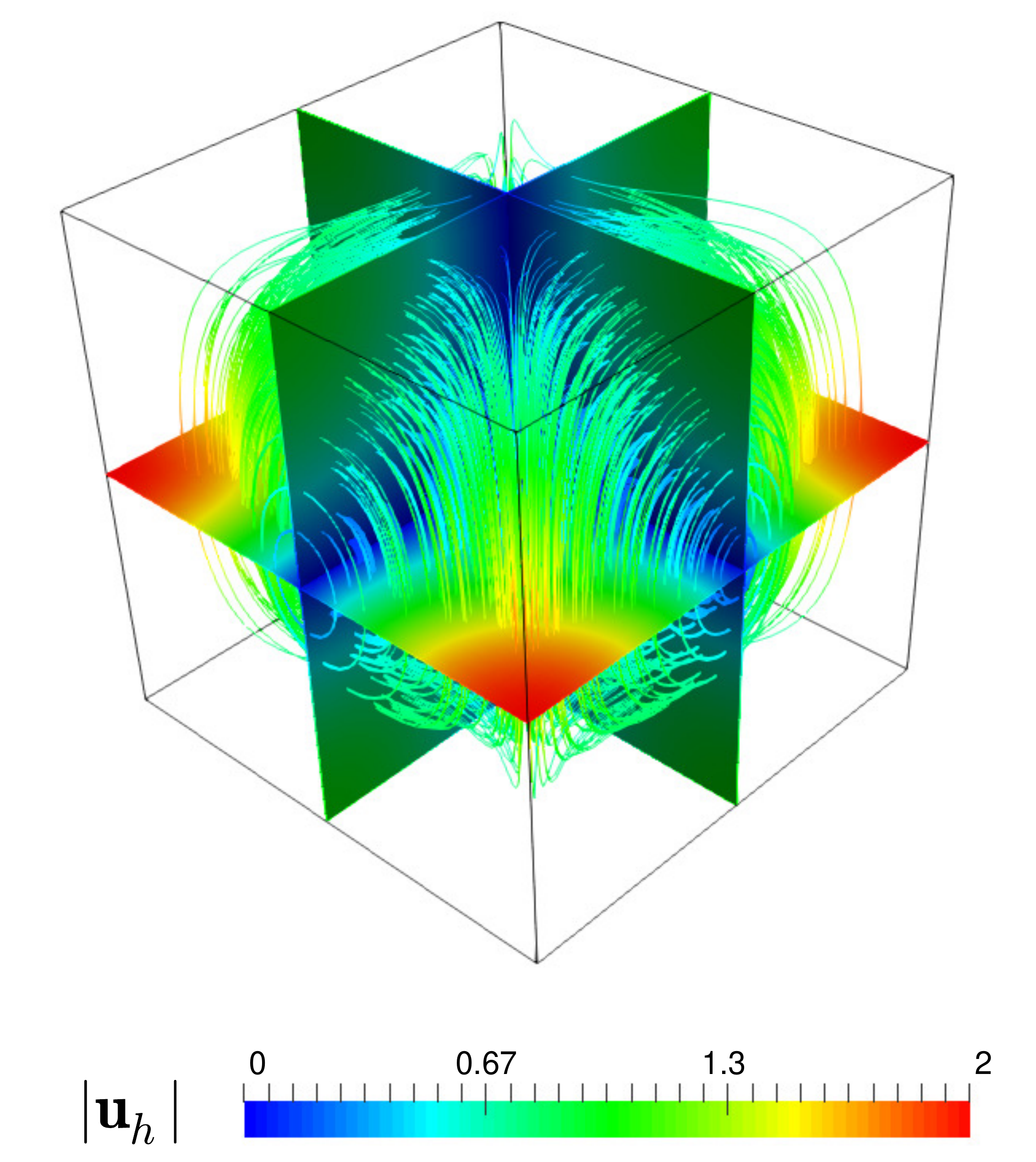}	
\includegraphics[width=5cm]{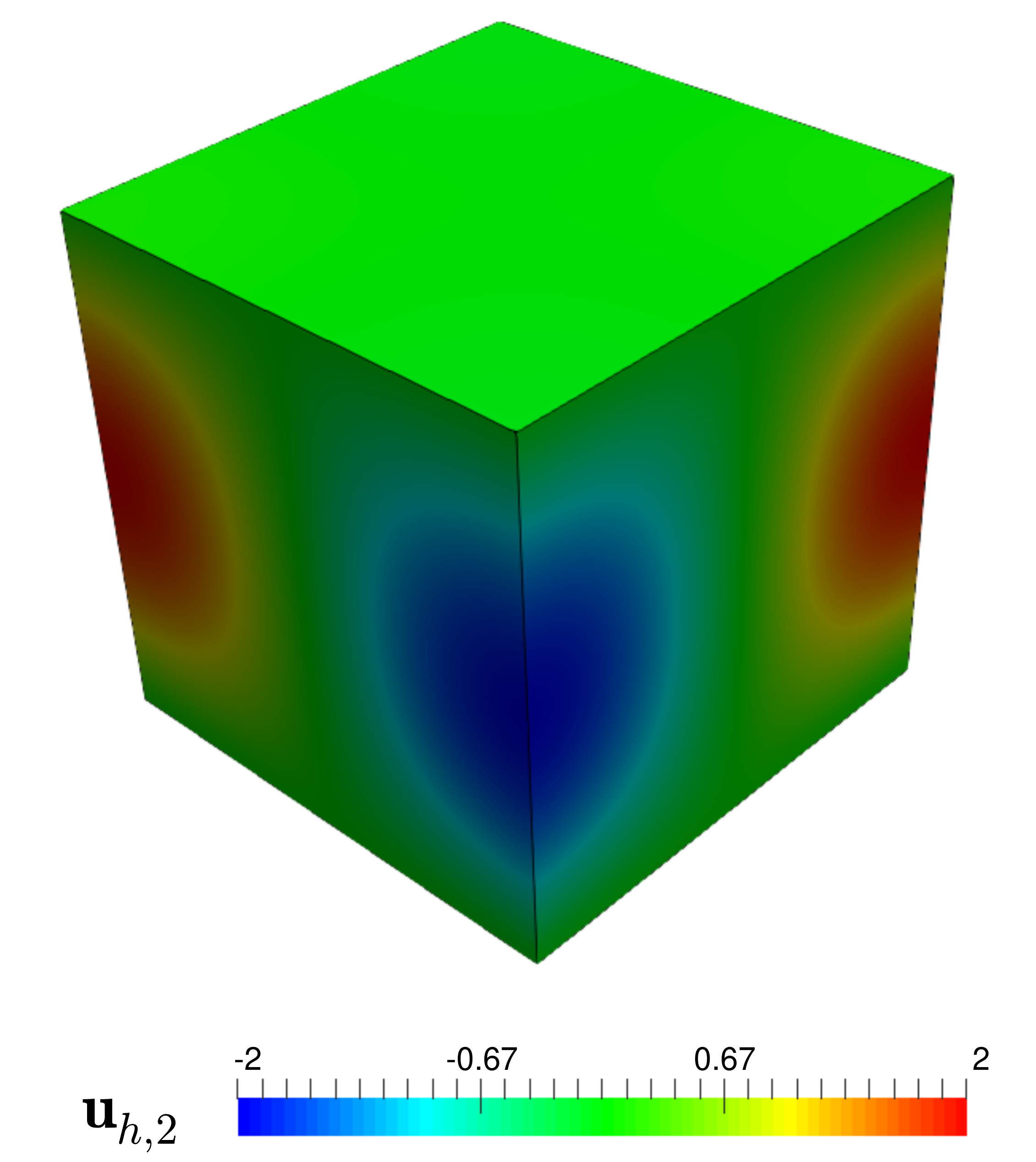}	
		
\includegraphics[width=5cm]{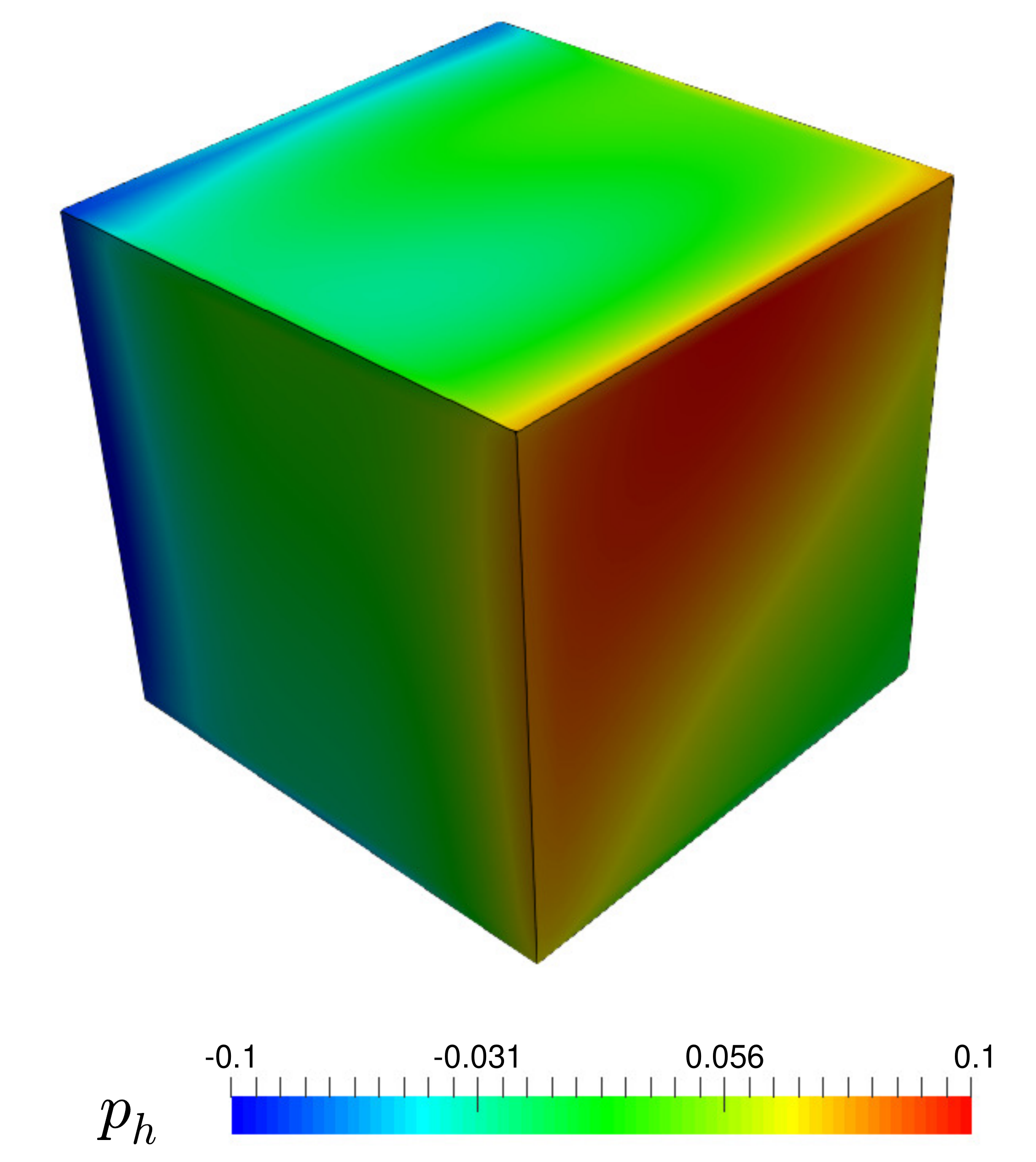}
\includegraphics[width=5cm]{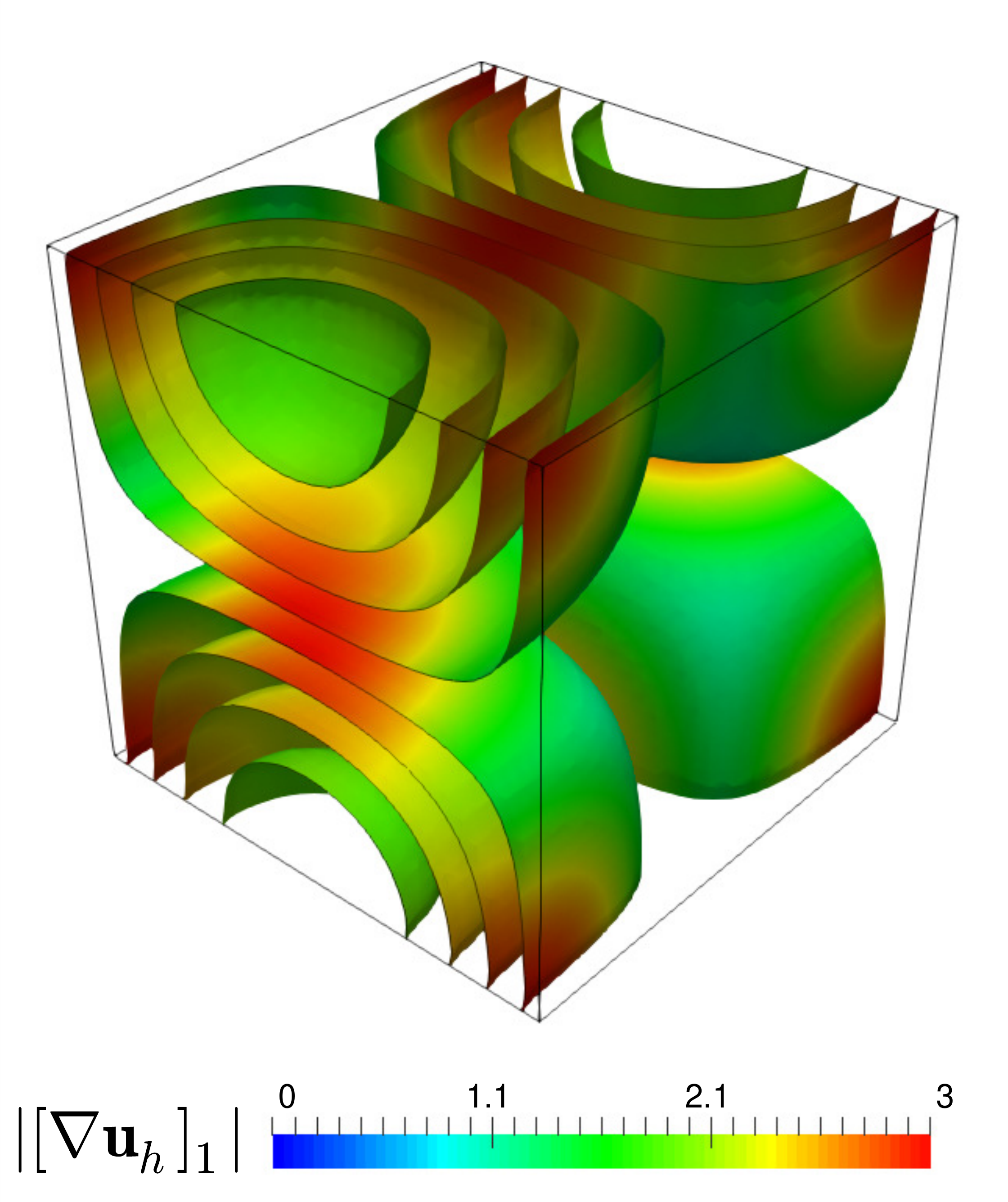}
\includegraphics[width=5cm]{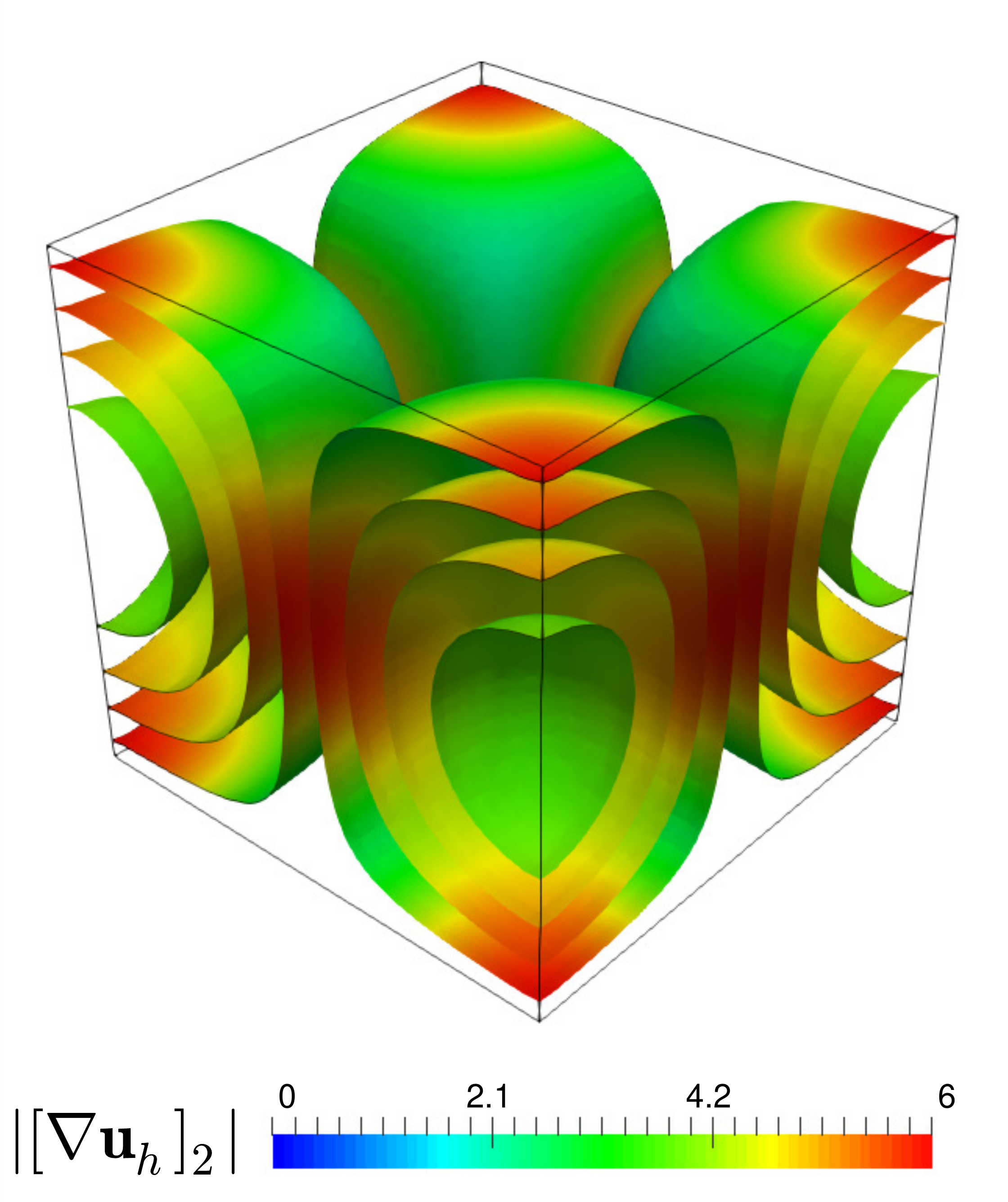}
\caption{Example 2, Computed magnitude of a component of the pseudostress tensor and
  the velocity, and a velocity component (top plots); pressure field and magnitude of the velocity gradient components (bottom plots).} \label{fig:example2}
\end{center}
\end{figure}

%%%%%%%%%%%    Table3 - example2 - k=0  %%%%%%%%%%%
\begin{table}[ht!]
\begin{center}
		
\begin{tabular}{r|c||cc|cc|cc}
\hline
& & \multicolumn{2}{|c|}{$\|\be_{\bsi}\|_{\ell^2(0,T;\bbX)}$} & \multicolumn{2}{|c|}{$\|\be_{\bu}\|_{\ell^2(0,T;\bM)}$} & \multicolumn{2}{|c}{$\|\be_{\bu}\|_{\ell^{\infty}(0,T;\bL^2(\Omega))}$} \\ 
$\DOF$    & $h$ & $\error$ & $\rate$ & $\error$ & $\rate$ & $\error$ & $\rate$ \\  \hline\hline
504     & 0.7071 & 1.8E-00 &   --   & 5.1E-02 &   --   & 4.5E-01 &   --   \\ 
3744    & 0.3536 & 8.1E-01 & 1.1148 & 2.7E-02 & 0.9181 & 2.4E-01 & 0.9353 \\ 
28800   & 0.1768 & 3.7E-01 & 1.1120 & 1.4E-02 & 0.9726 & 1.2E-01 & 0.9758 \\
225792  & 0.0884 & 1.8E-01 & 1.0863 & 7.0E-03 & 0.9927 & 6.1E-02 & 0.9940 \\
1787904 & 0.0442 & 8.5E-02 & 1.0486 & 3.5E-03 & 0.9981 & 3.0E-02 & 0.9985 \\
\hline 
\end{tabular}
		
\medskip
		
\begin{tabular}{cc|cc|c}
\hline
\multicolumn{2}{c}{$\|\be_p\|_{\ell^2(0,T;\L^2(\Omega))}$} & \multicolumn{2}{|c|}{$\|\be_{\bGu}\|_{\ell^2(0,T;\bbL^2(\Omega))}$} & \\ 
$\error$ & $\rate$ & $\error$ & $\rate$ & $\iter$ \\  \hline\hline
1.4E-01 &   --   & 2.3E-01 &  --    & 4 \\ 
7.2E-02 & 0.9785 & 1.2E-01 & 0.9868 & 4 \\ 
3.3E-02 & 1.1336 & 5.9E-02 & 0.9897 & 3 \\
1.3E-02 & 1.3651 & 2.9E-02 & 0.9957 & 3 \\
4.7E-03 & 1.4490 & 1.5E-02 & 0.9985 & 3 \\
\hline 
\end{tabular}
		
\caption{{\sc Example 2}, Number of degrees of freedom, mesh sizes, errors, rates of convergences, and average number of Newton iterations for the mixed $\bbRT_0-\bP_0$ approximation of the Brinkman--Forchheimer model.}\label{tab:example2-k0}
\end{center}
\end{table}

%************************************************************************

\subsection*{Example 3: Lid-driven cavity flow with different Darcy and Forchheimer coefficients}

In our third example, inspired by \cite[Section~V]{lst2017}, we study the behavior of the model for lid-driven cavity flow in the square domain $\Omega = (0,1)^2$ with different values of the parameters $\alpha$ and $\tF$. The body force term is chosen as $\f = \0$,
the initial condition is zero, and the boundaries conditions are 
\begin{equation*}
\bu = (1,0) \qon \Gamma_{\mathrm{top}},\quad
\bu = \0 \qon \partial\Omega\setminus\Gamma_{\mathrm{top}}.
\end{equation*}

In Figure~\ref{fig:example3a} we display the approximated magnitude of
the velocity under two scenarios.  In the first one we consider fixed
$\alpha = 1$ and different values of $\tF$, i.e., $\tF\in \big\{ 1,
10, 100, 1000 \big\}$, whereas in the second case we take fixed $\tF =
1$ and different $\alpha\in \big\{ 10, 100, 1000 \big\}$.  The curves
corresponding to the values of the first velocity component along the
vertical line through the cavity center for the different values of
$\alpha$ and $\tF$ are shown in Figure~\ref{fig:example3b}. The
results illustrate the ability of the method to handle a wide range of
parameters. In addition, we observe a reasonable behavior of the
Brinkman--Forchheimer model when varying the parameters. In
particular, keeping $\alpha$ fixed and increasing $\tF$ results in
slower flow, since the nonlinear inertial term starts dominating the
viscous term in the Stokes flow. Keeping $\tF$ fixed and increasing
$\alpha$ has an even stronger effect of slowing down the flow, since
in this case the model approaches Darcy flow with smaller and smaller
permeability. We note that the case $\tF=1$, $\alpha = 1000$ is
especially challenging numerically, due to the sharp boundary layer
near the top boundary, but the method handles it without any
difficulty.

%%%%%%%%%%%    Figure3 - example3 - uh-magnitude  %%%%%%%%%%%
\begin{figure}[ht]
\begin{flushright}
\includegraphics[width=4.15cm]{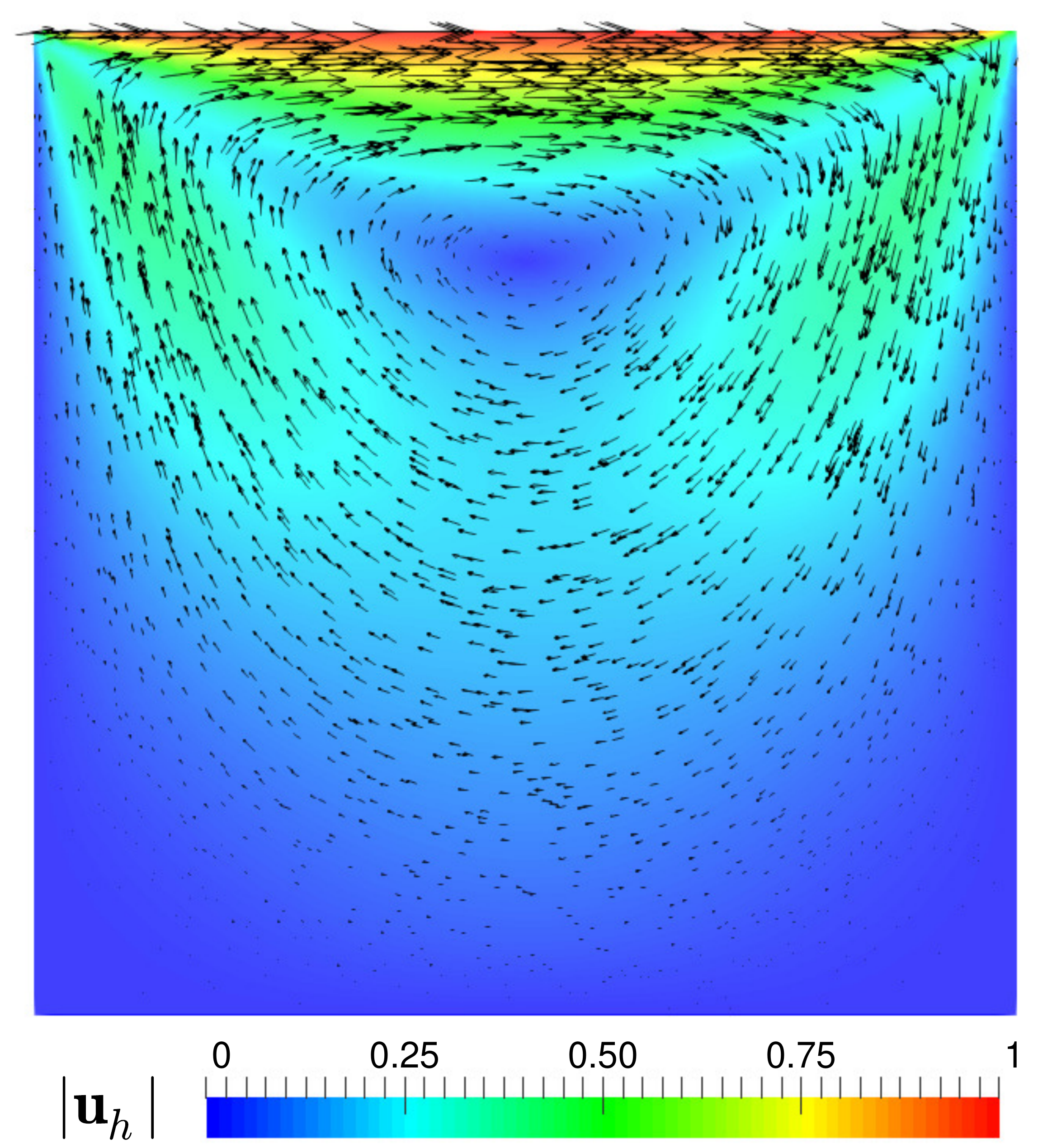}
\includegraphics[width=4.15cm]{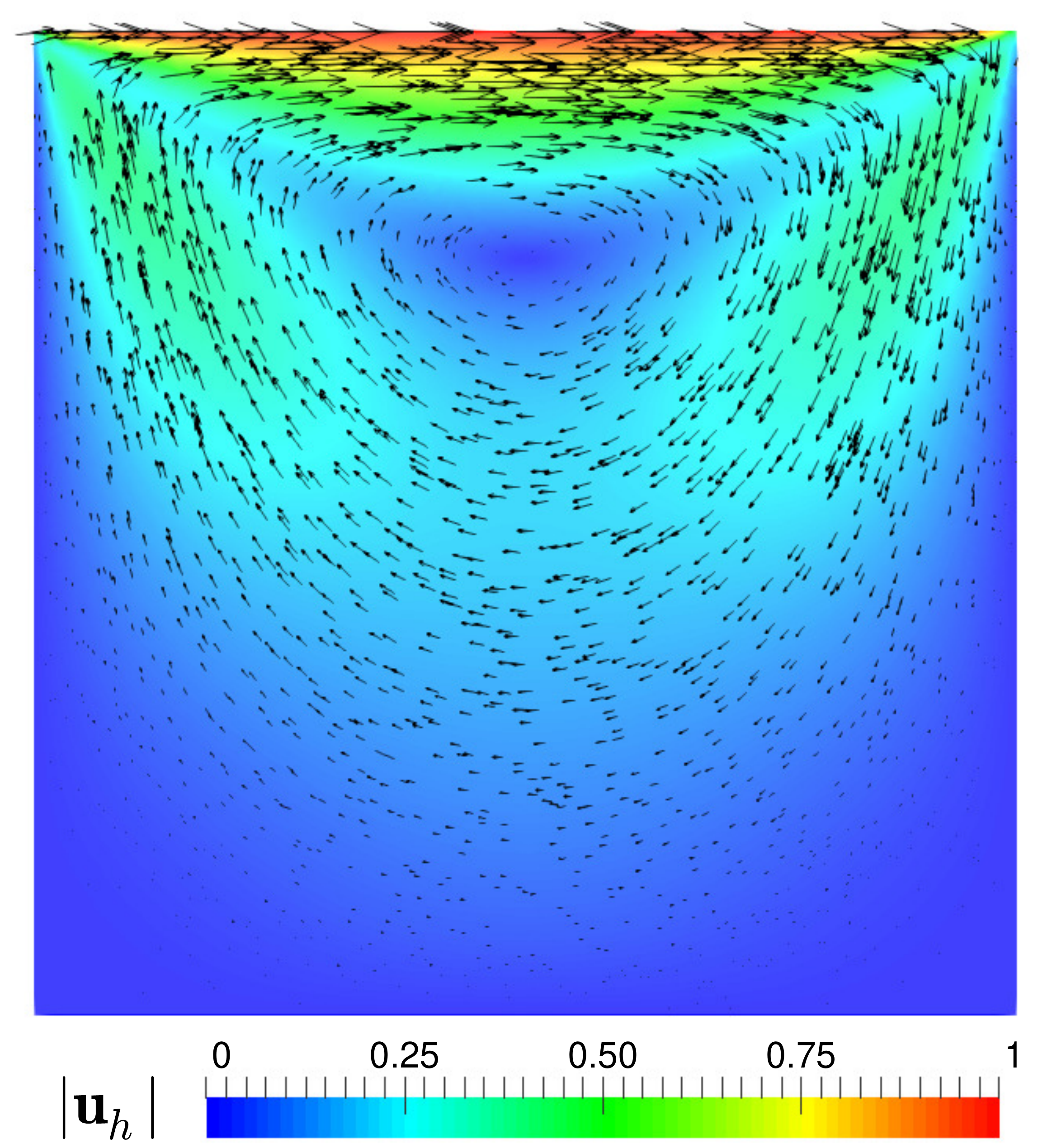}
\includegraphics[width=4.15cm]{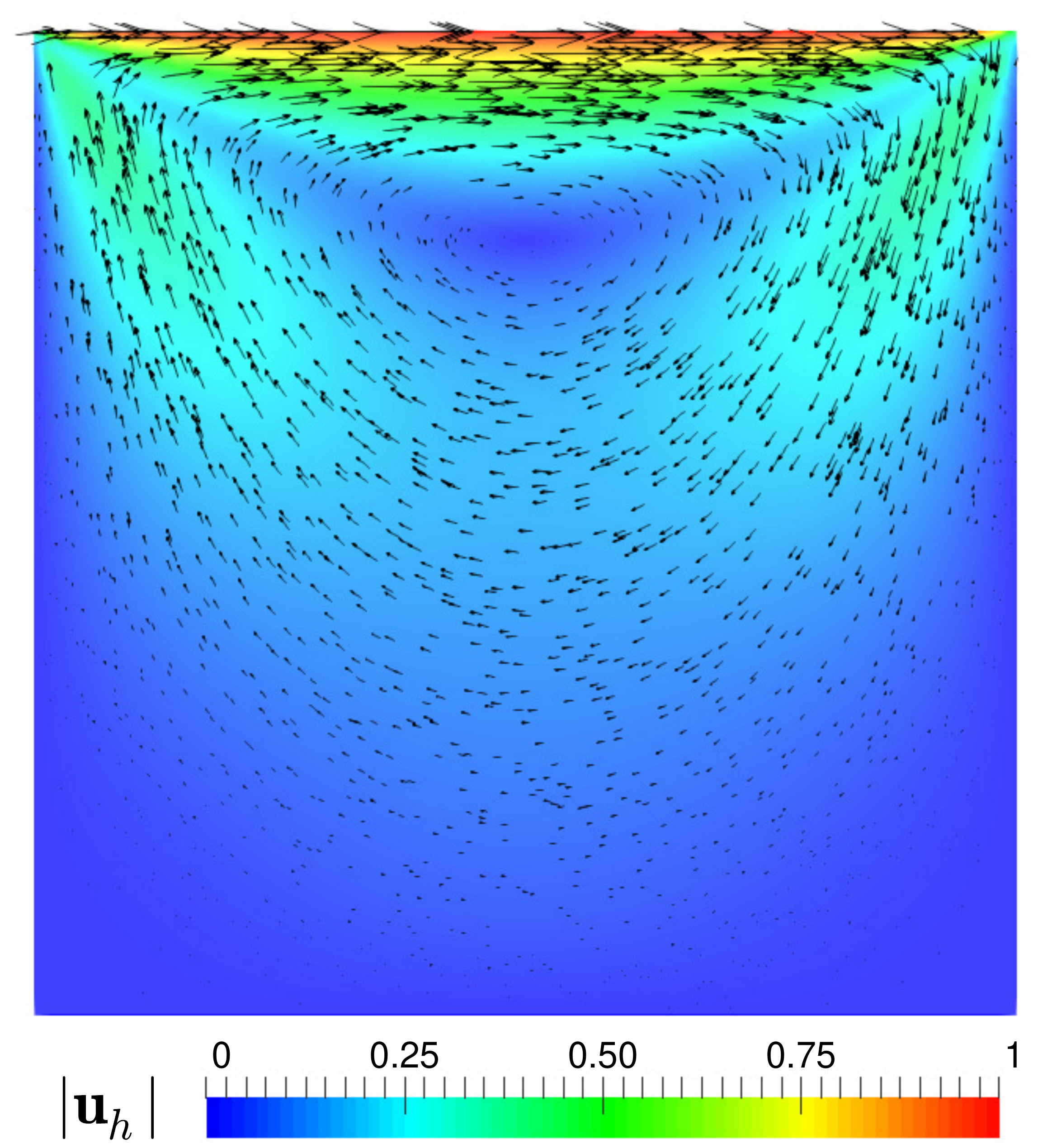}
\includegraphics[width=4.15cm]{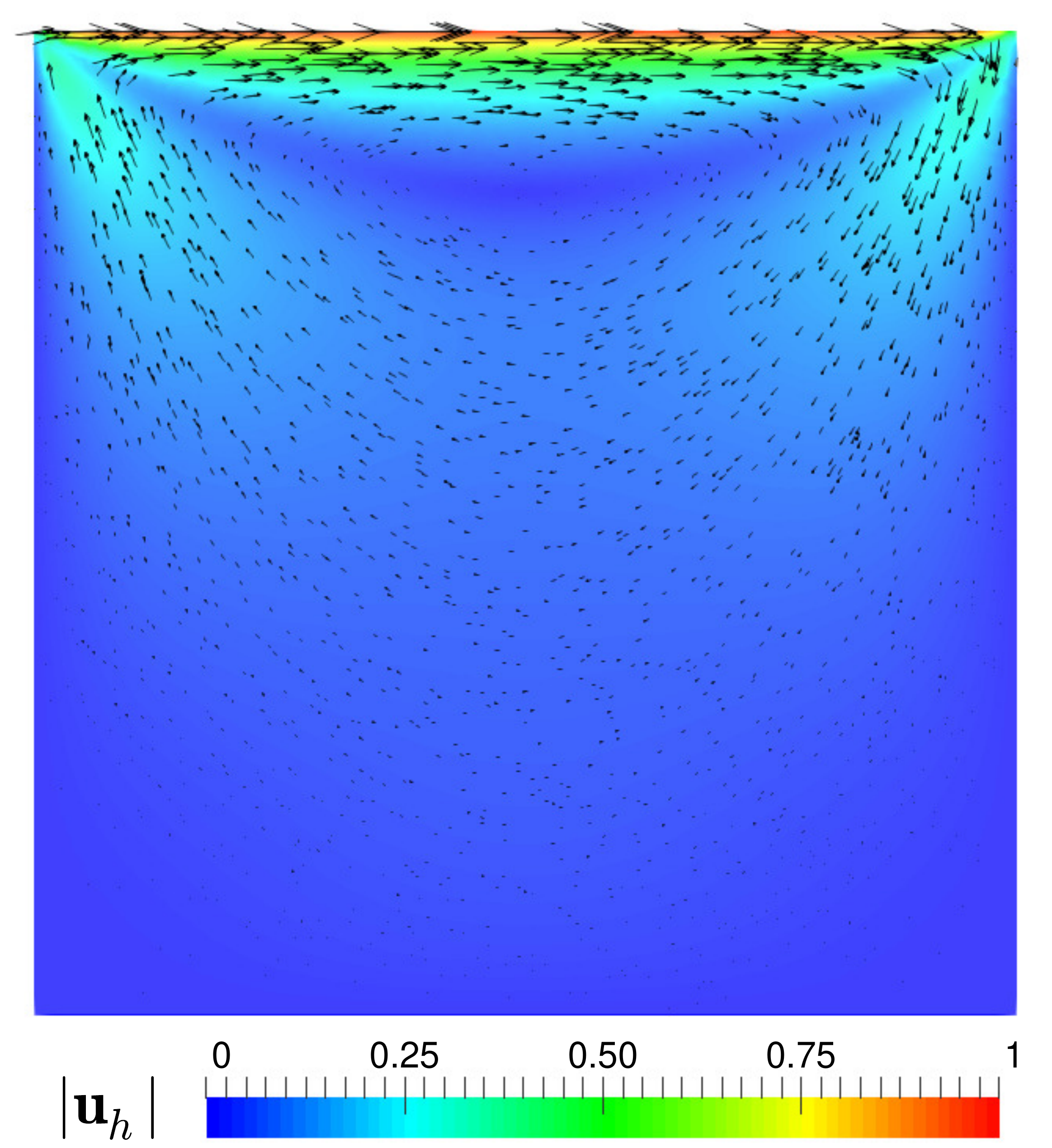}

\bigskip

\includegraphics[width=4.15cm]{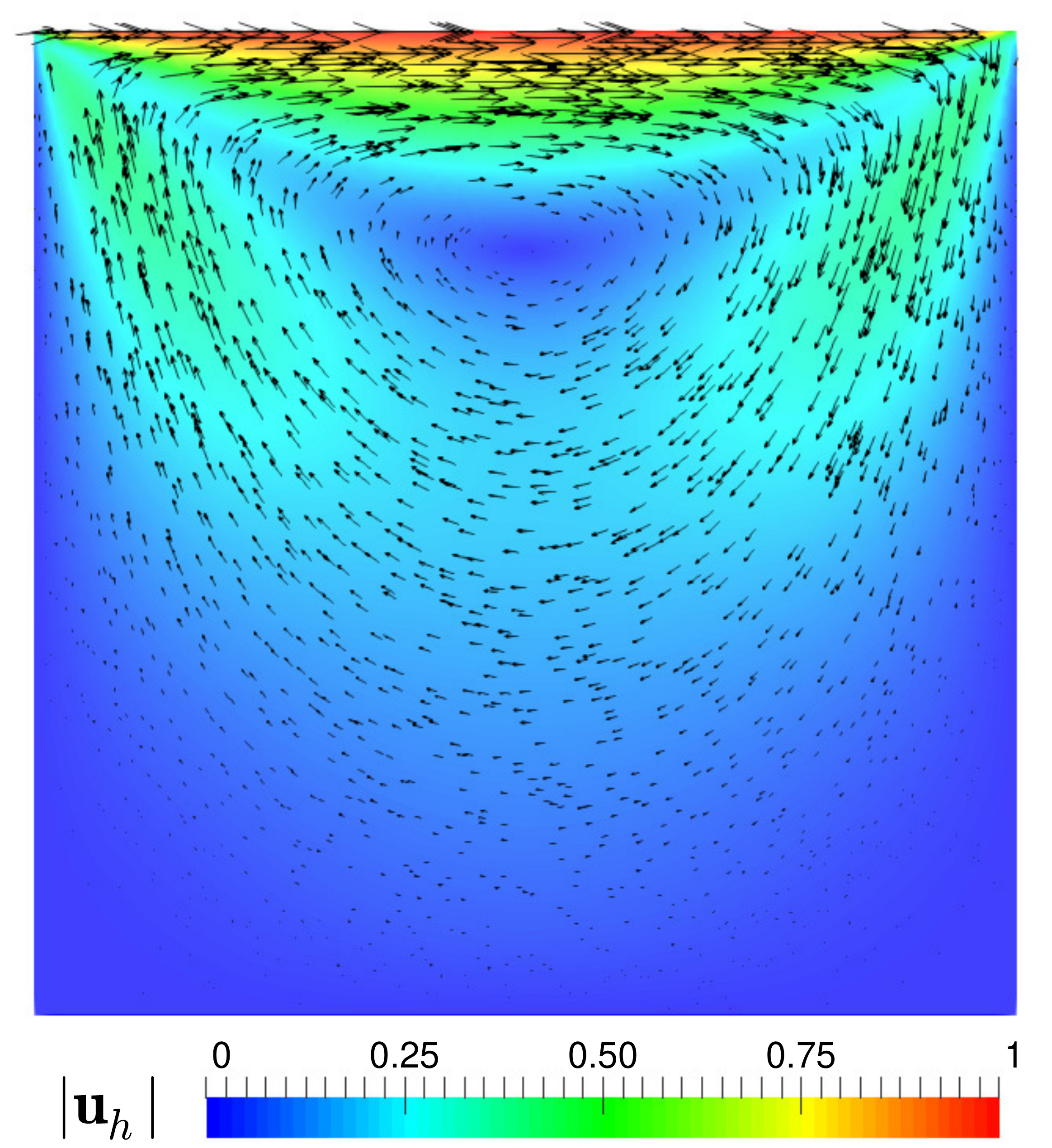}
\includegraphics[width=4.15cm]{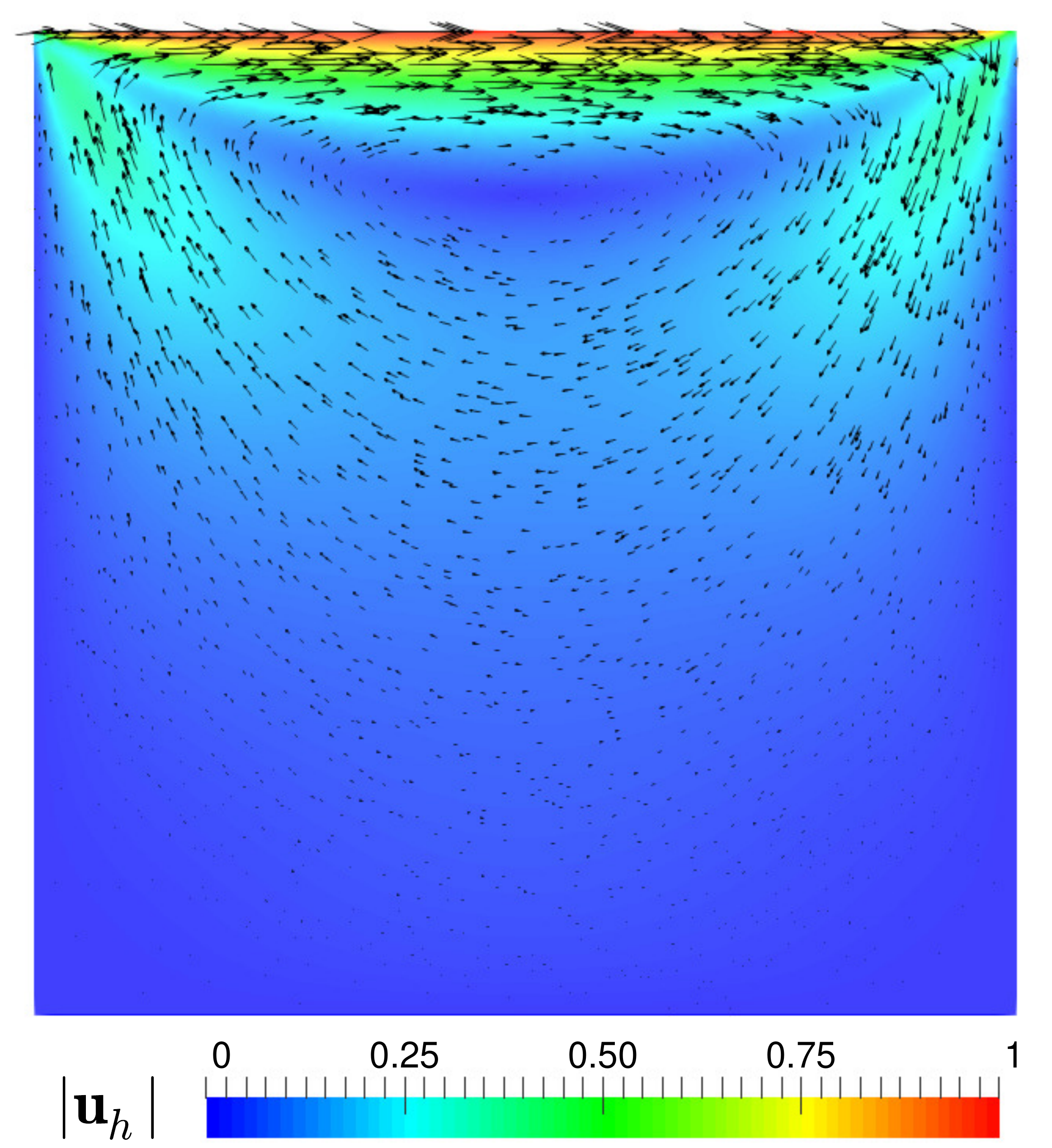}
\includegraphics[width=4.15cm]{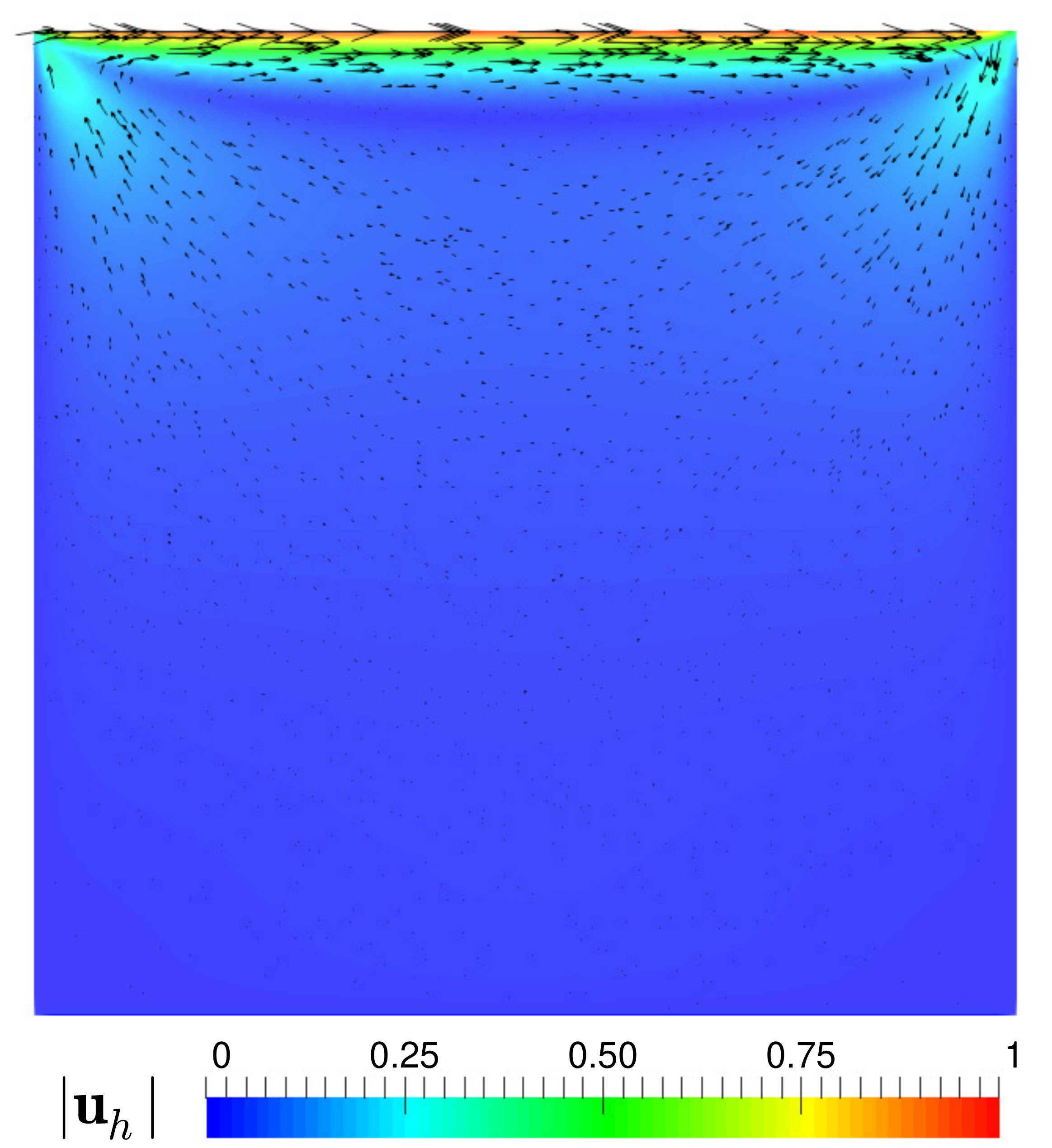}
\caption{Example 3, Approximated magnitude of the velocity at time $T=1$, considering $\alpha = 1$ and $\tF\in \big\{ 1, 10, 100, 1000 \big\}$ (top plots), and $\tF = 1$ and $\alpha\in \big\{ 10, 100, 1000 \big\}$ (bottom plots).}\label{fig:example3a}
\end{flushright}
\end{figure}

%%%%%%%%%%%    Figure4 - example3 - curves  %%%%%%%%%%%
\begin{figure}[ht]
\begin{center}
\includegraphics[width=8cm]{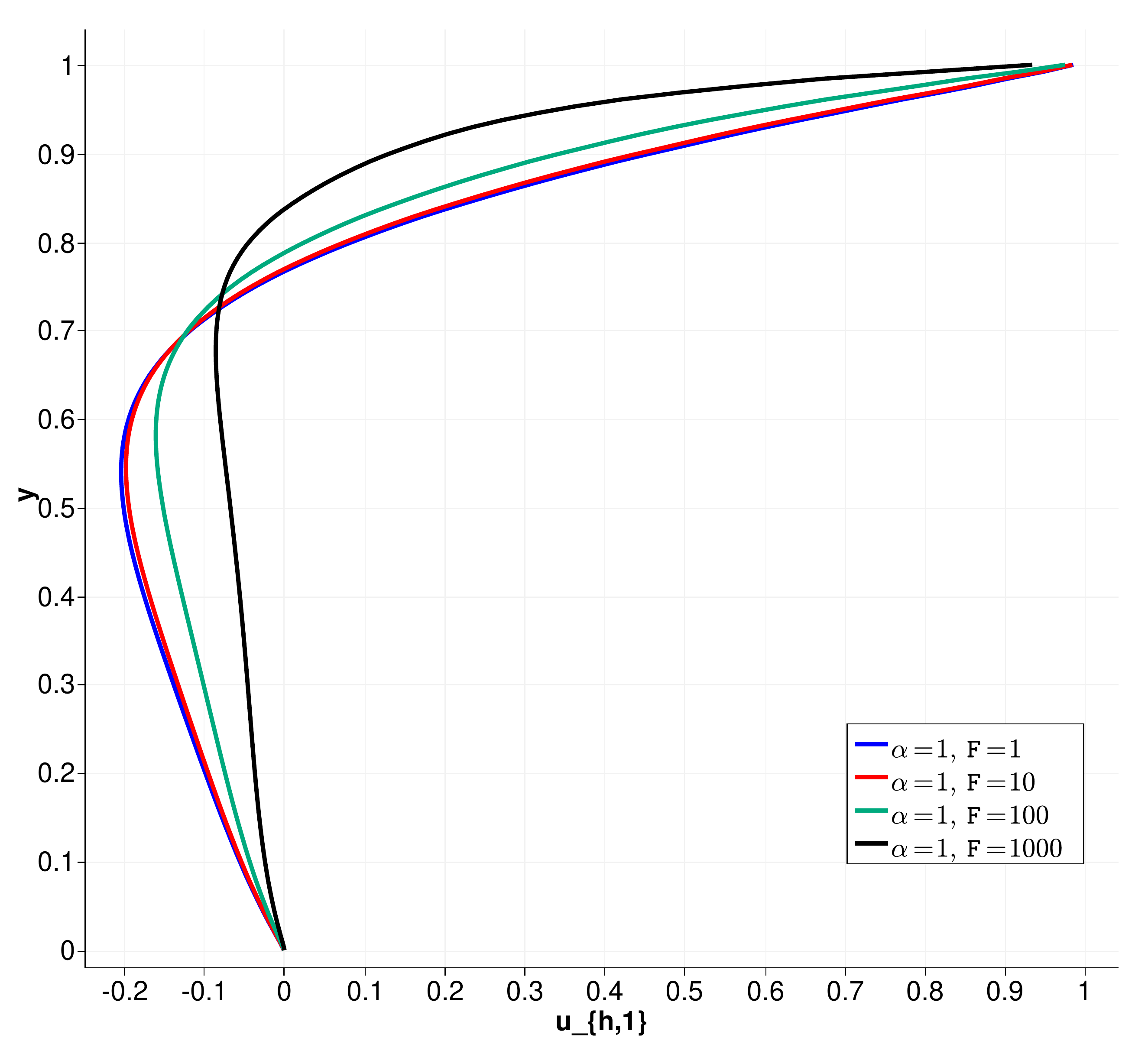}
\hspace{0.5cm}
\includegraphics[width=8cm]{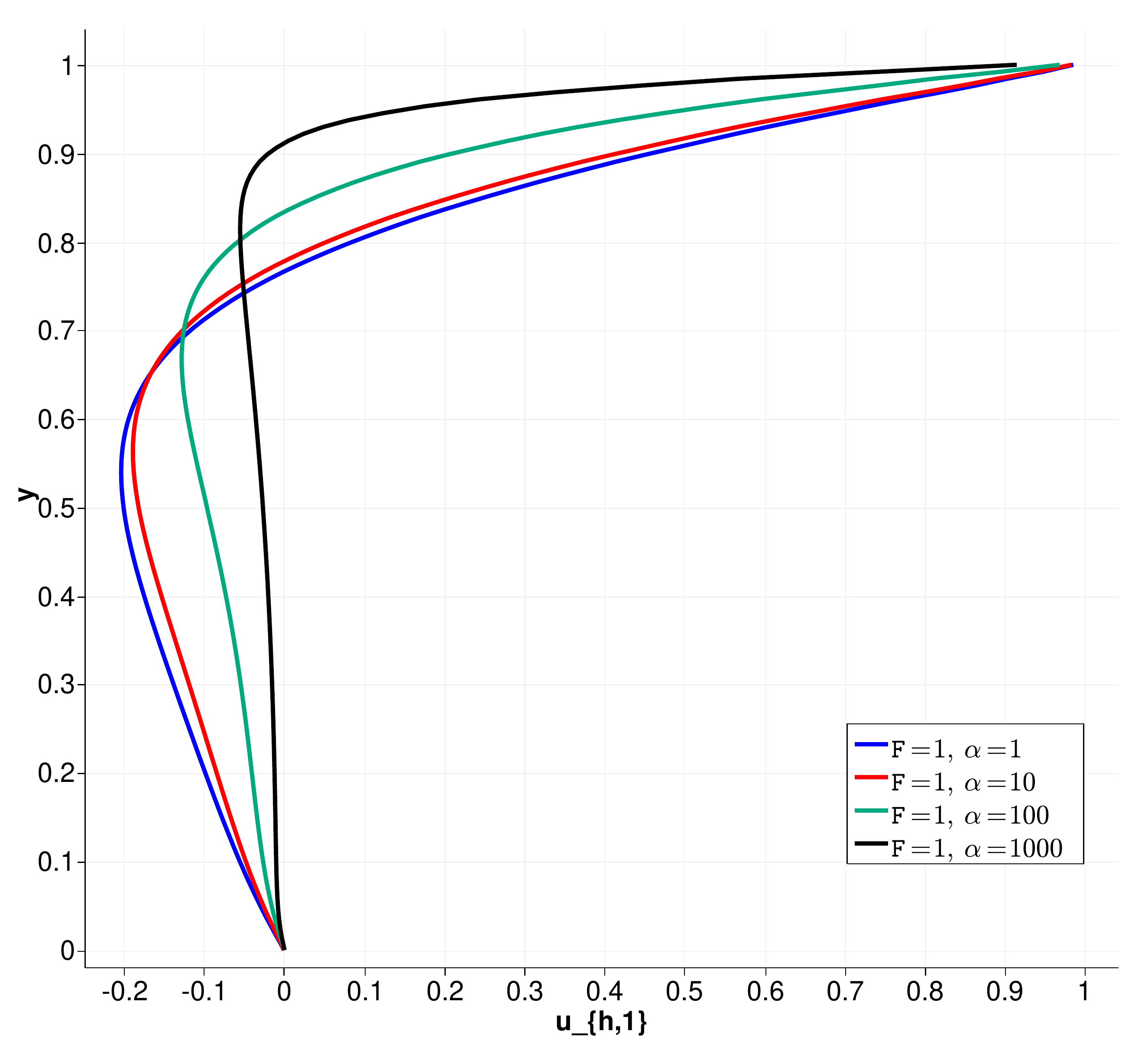}		
\caption{Example 3, Curves associated to the first velocity component along the vertical line through the cavity center considering different values of $\alpha$ and $\tF$.}\label{fig:example3b}
\end{center}
\end{figure}

%************************************************************************

\subsection*{Example 4: Flow through porous media with channel network}

In our last example, inspired by \cite[Section~5.2.4]{akny2019}, we focus on
flow through porous media with channel network.
We consider the square domain $\Omega = (-1,1)^2$ with an internal channel network denoted as $\Omega_{\rc}$, which is described in the first plot of Figure~\ref{fig:example4}.
We consider the Brinkman--Forchheimer model \eqref{eq:Brinkman-Forchheimer-2} in the whole domain $\Omega$ but with different values of the parameters $\alpha$ and $\tF$ for the interior
and the exterior of the channel, that is,
\begin{equation*}
\alpha \,=\, \left\{\begin{array}{rl}
1 &\mbox{in }\,\, \Omega_{\rc} \\
1000 &\mbox{in }\,\, \ov{\Omega}\setminus \Omega_{\rc}
\end{array}\right. \qan
\tF \,=\, \left\{\begin{array}{rl}
10 &\mbox{in }\,\, \Omega_{\rc} \\
1 &\mbox{in }\,\, \ov{\Omega}\setminus \Omega_{\rc}
\end{array}\right..
\end{equation*}
The parameter choice corresponds to a high permeability ($\alpha = 1$) in the channel and
increased inertial effect ($\tF = 10$), compared to low permeability ($\alpha = 1000$)
in the porous media and reduced inertial effect ($\tF = 1$). 
In addition, the body force term is $\f = \0$, the initial condition is zero, and
the boundaries conditions are 
\begin{equation*}
\bu\cdot\bn = 0.2,\quad \bu\cdot\bt = 0 \qon \Gamma_{\mathrm{left}},\quad
\bsi\,\bn = \0 \qon \partial\Omega\setminus\Gamma_{\mathrm{left}},
\end{equation*}
which corresponds to inflow on the left boundary and zero stress outflow on the rest of the
boundary.

In Figure~\ref{fig:example4} we display the computed magnitude of the
velocity, pseudostress tensor component, and velocity gradient
component at times $T=0.01$ and $T=1$. As expected, we observe faster
flow through the channel network, with a significant velocity gradient
across the interface between the channel and the porous media. The
pseudostress is more diffused, since it includes the pressure field.
This example illustrates the ability of the Brinkman--Forchheimer
model to handle heterogeneous media using spatially varying
parameters. The example is particularly challenging, due to the strong
jump discontinuity of the parameters across the two regions, and our
numerical method handled it well.

%%%%%%%%%%%    Figure5 - example4 - fracture domain  %%%%%%%%%%%
\begin{figure}[ht]
\begin{flushright}
\includegraphics[width=4.15cm]{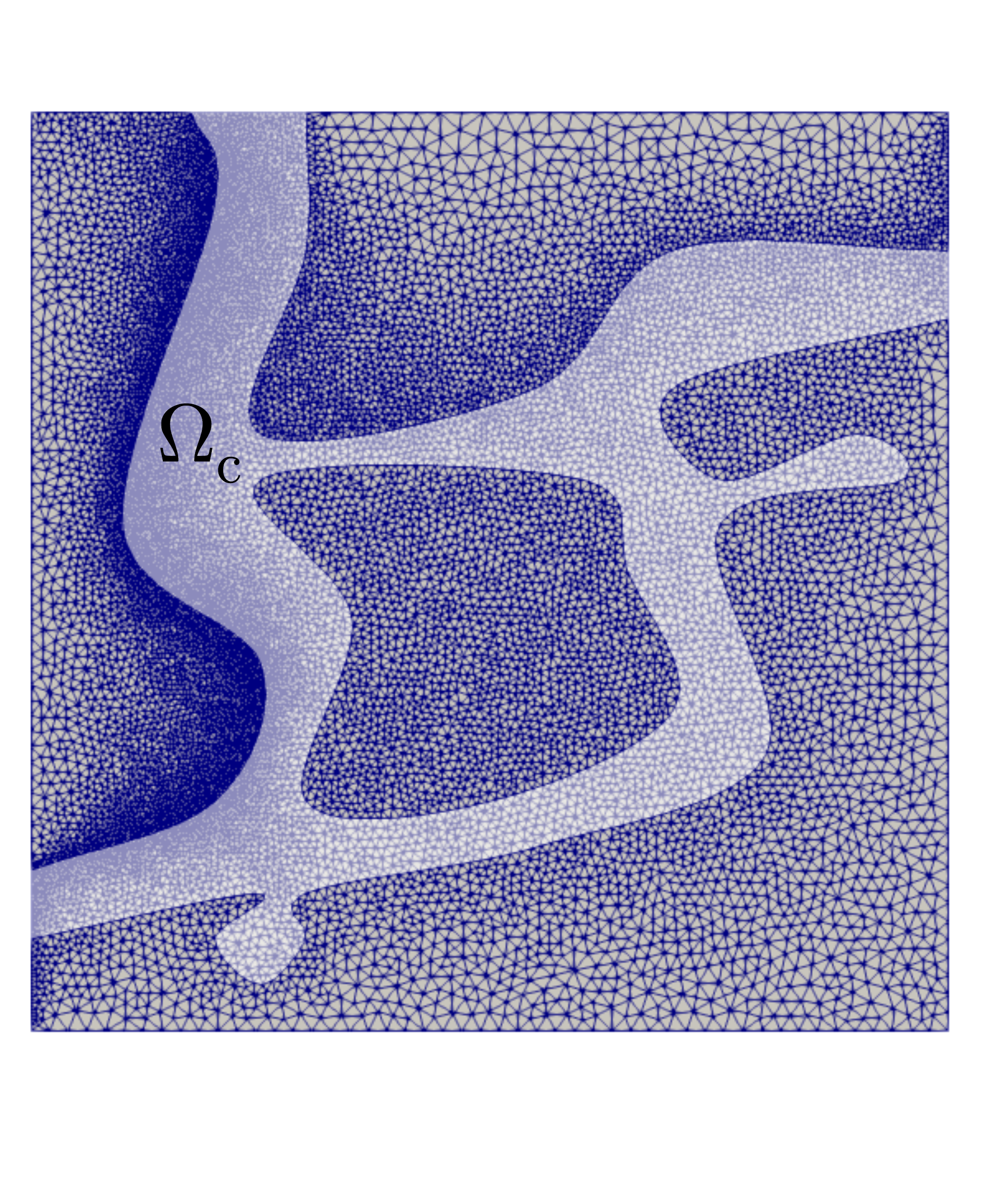}	
\includegraphics[width=4.15cm]{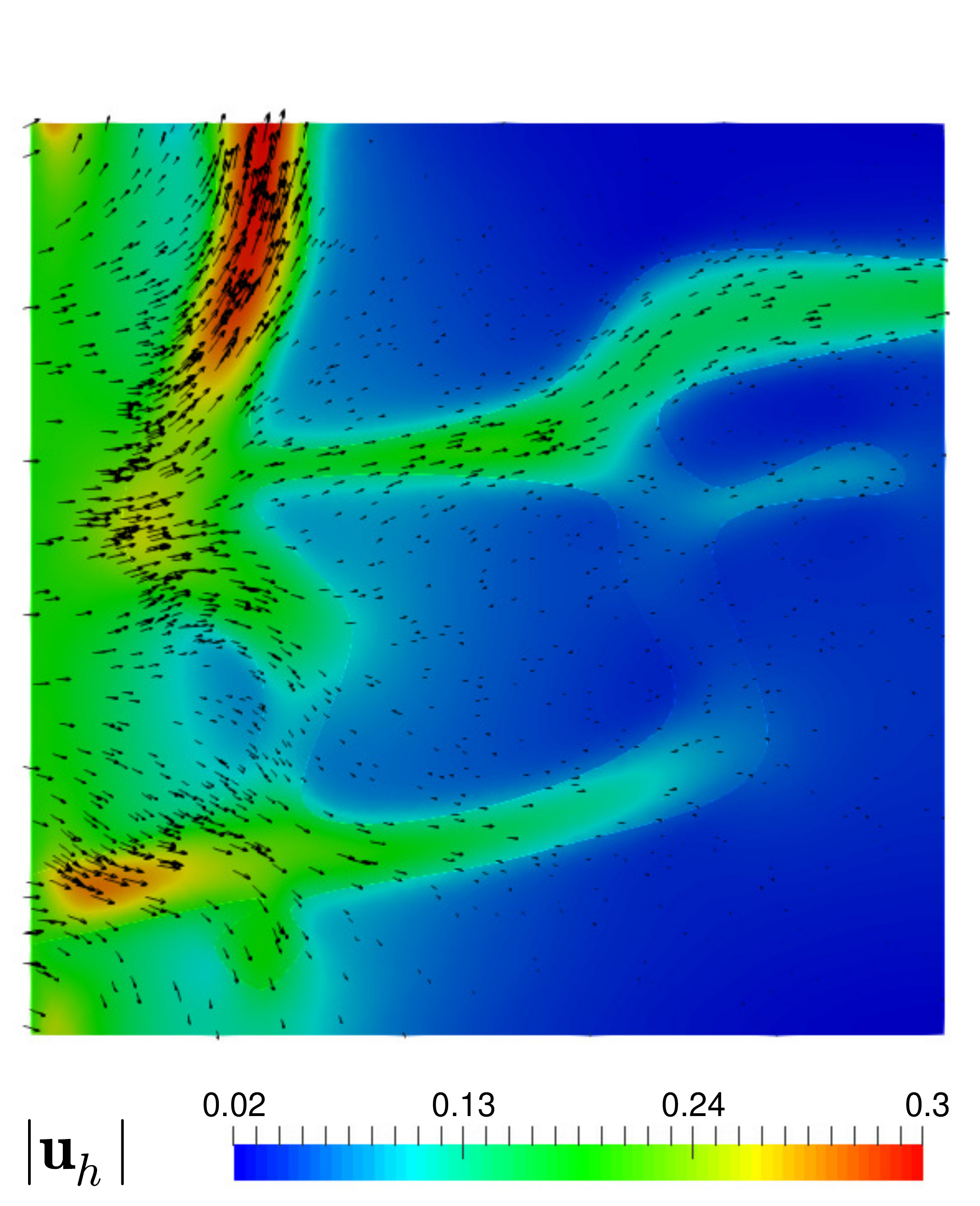}
\includegraphics[width=4.15cm]{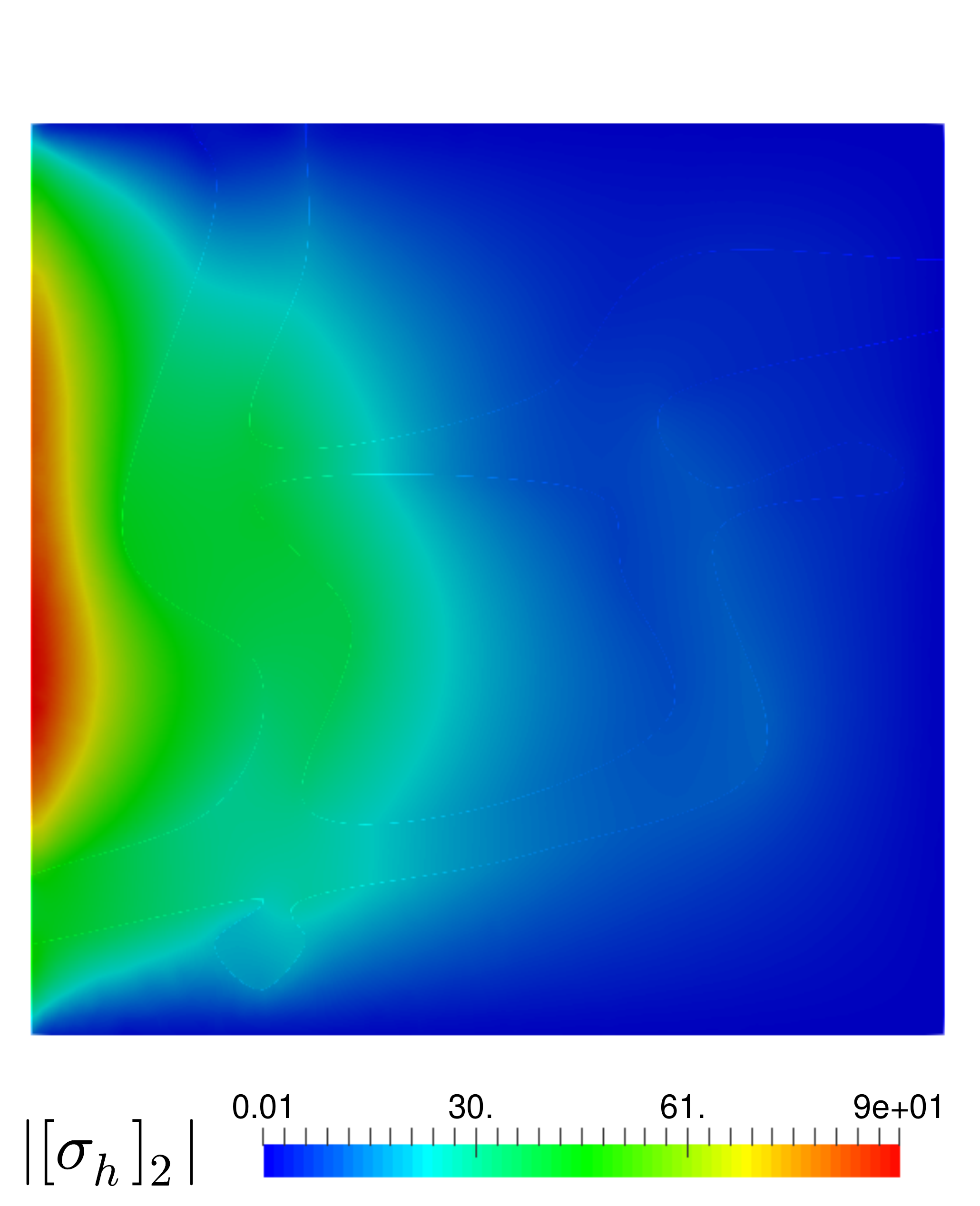}
\includegraphics[width=4.15cm]{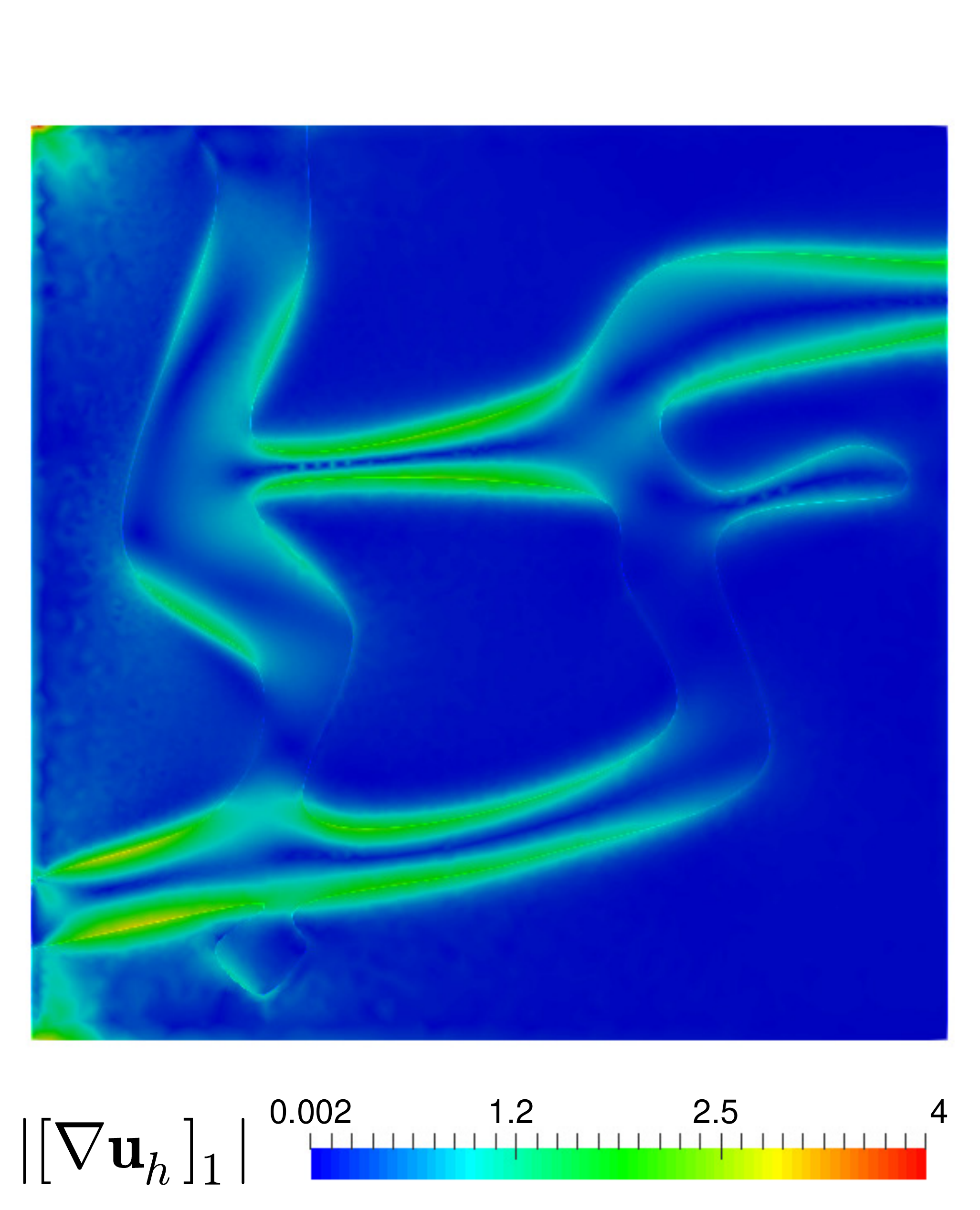}	
		
\includegraphics[width=4.15cm]{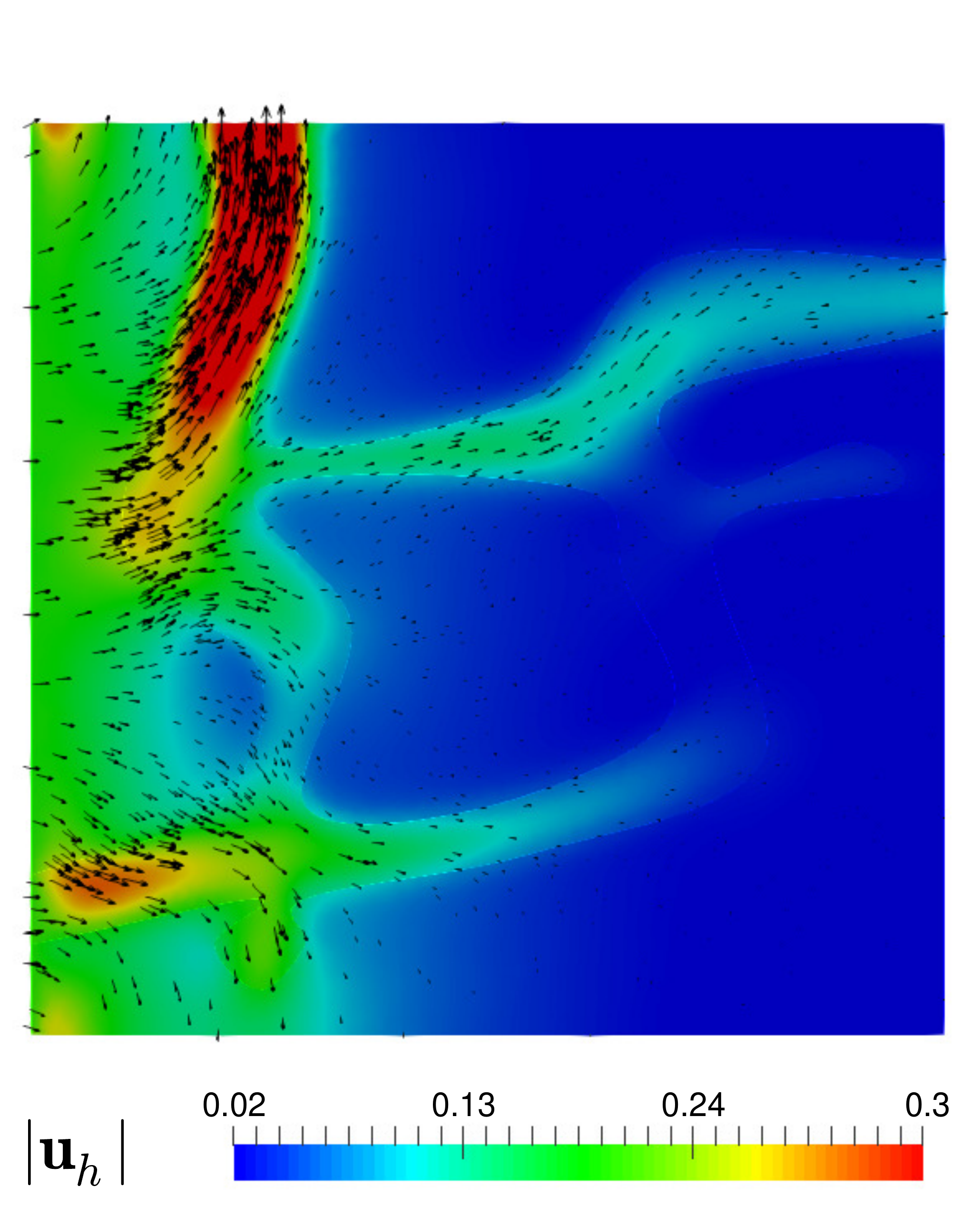}
\includegraphics[width=4.15cm]{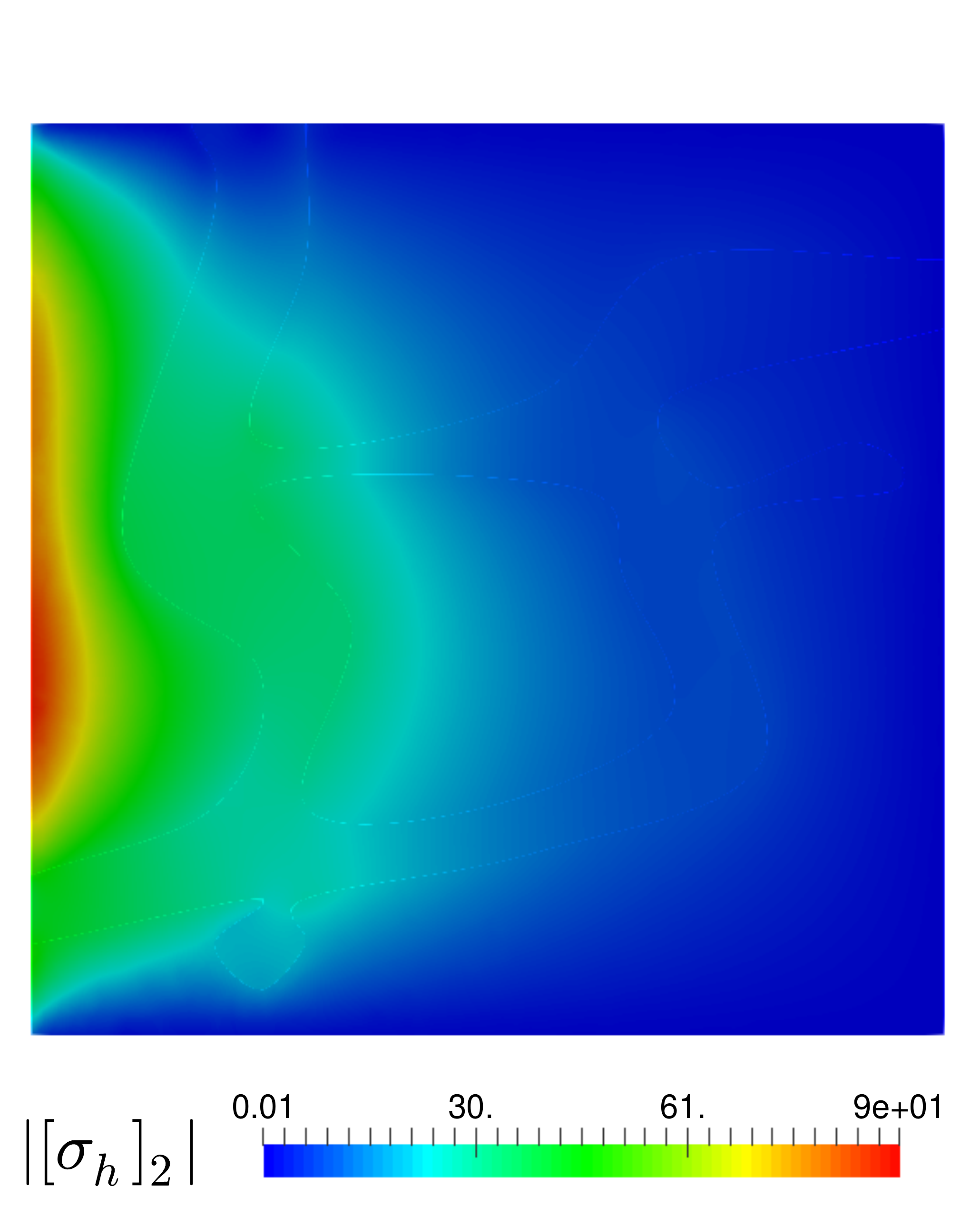}
\includegraphics[width=4.15cm]{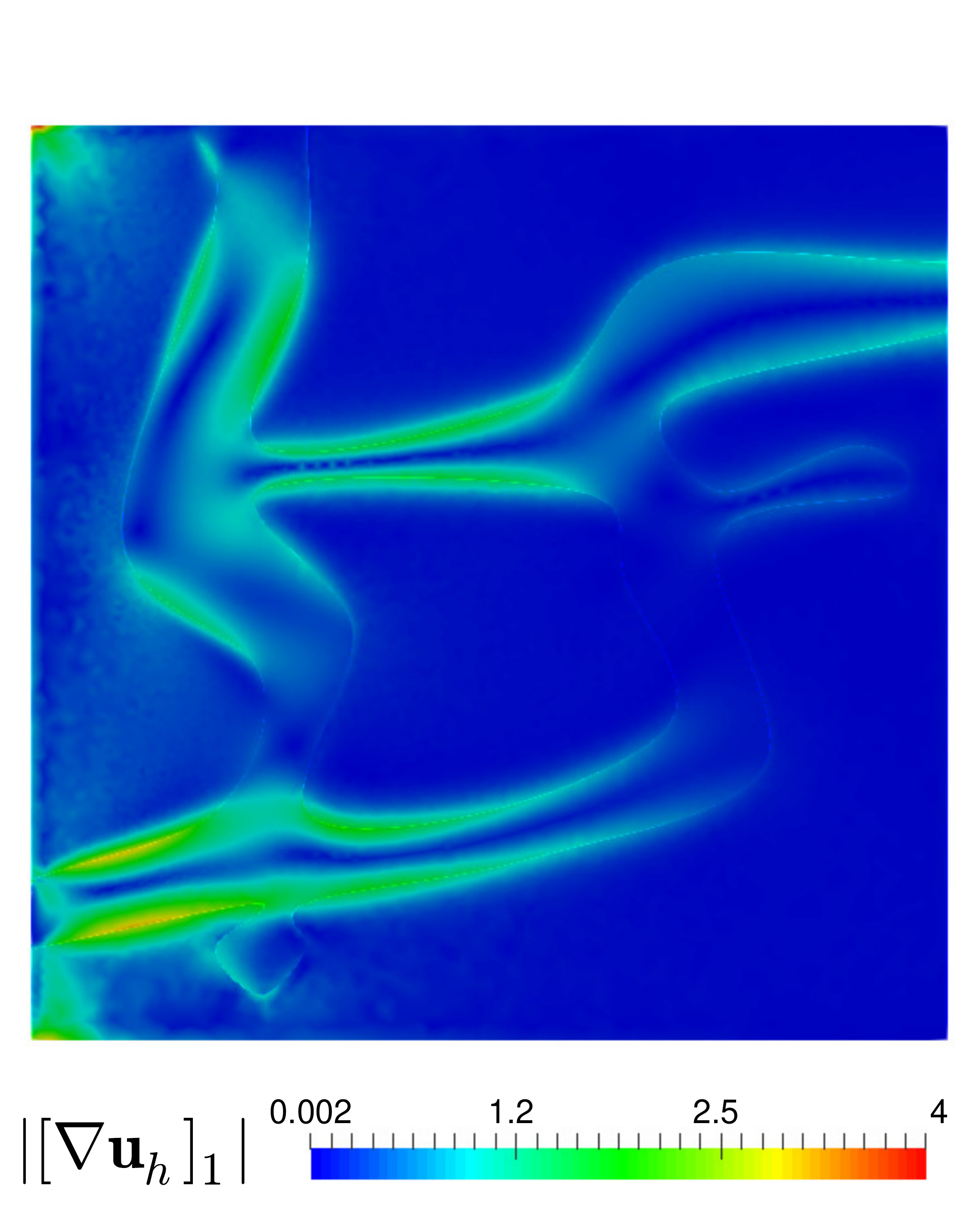}	
\caption{Example 4, Domain configuration, computed magnitude of the velocity, pseudostress tensor component, and velocity gradient component at time $T=0.01$ (top plots), and at time $T=1$ (bottom plots).}\label{fig:example4}
\end{flushright}
\end{figure}

\bibliographystyle{abbrv}
\bibliography{caucao-yotov-1}
\end{document}